\documentclass[10pt,twoside,english,reqno,a4paper]{amsart}
\usepackage{amsmath,amssymb,latexsym,esint,mathrsfs,mathtools}
\usepackage{amsthm}
\usepackage{verbatim,wasysym}

\usepackage[english]{babel} %
\usepackage[T1]{fontenc}
\usepackage{setspace} %
\usepackage{appendix}
\usepackage{etoolbox}
\usepackage{microtype}
\usepackage{color,enumitem,graphicx}

\usepackage{tikz}
\usepackage{pgfplots}
\pgfplotsset{compat=1.18}
\usetikzlibrary{shapes.geometric, calc}

\definecolor{red}{rgb}{0.9,0,0}
\usepackage{palatino}
\definecolor{azul}{rgb}{0,0,1}

\usepackage{float}

\usepackage{geometry}
\usepackage{booktabs}

\def\XXint#1#2#3{{\setbox0=\hbox{$#1{#2#3}{\int}$}
 \vcenter{\hbox{$#2#3$}}\kern-.5\wd0}}

\def\re{\mathbb{R}}

\usepackage[%
colorlinks=true,
citecolor=red,
linkcolor=blue,
urlcolor=green!45!black,
linktocpage,
pdfpagelabels,
bookmarksnumbered,
bookmarksopen]
{hyperref}

\usepackage{csquotes}

\usepackage[hang]{footmisc}
\usepackage{xpatch}
\usepackage{footnotebackref} %
\usepackage[hyperpageref]{backref}

\makeatletter
\xpretocmd{\@adminfootnotes}{\let\@makefntext\BHFN@OldMakefntext}{}{}
\renewcommand\@makefntext[1]{%
 \@ifundefined{@makefnmark}
 {}
 {%
 \renewcommand\@makefnmark{%
 \mbox{%
 \textsuperscript{%
 \normalfont
 \hyperref[\BackrefFootnoteTag]{\@thefnmark}%
 }%
 }\,%
 }%
 \BHFN@OldMakefntext{#1}%
 }%
}
\makeatother

\patchcmd{\section}{\scshape}{\normalfont \Large \bfseries}{}{}
\patchcmd{\subsection}{\normalfont}{\normalfont \large}{}{}
\patchcmd{\subsection}{-.5em}{.5\linespacing}{}{}
\patchcmd{\subsubsection}{\normalfont}{\noindent \normalfont \large \bfseries \textit}{}{}
\patchcmd{\subsubsection}{-.5em}{.5\linespacing}{}{}

\DeclareRobustCommand{\SkipTocEntry}[5]{}
\protected\def\ignorethis#1\endignorethis{}
\let\endignorethis\relax

\let\oldtocsection=\tocsection
\let\oldtocsubsection=\tocsubsection
\let\oldtocsubsubsection=\tocsubsubsection
\renewcommand{\tocsection}[2]{\hspace{0em}\oldtocsection{#1}{#2}}
\renewcommand{\tocsubsection}[2]{\hspace{1em}\oldtocsubsection{#1}{#2}}
\renewcommand{\tocsubsubsection}[2]{\hspace{2em}\oldtocsubsubsection{#1}{#2}}
\newtheorem{lemma}{Lemma}[section]
\newtheorem{corollary}[lemma]{Corollary}
\newtheorem{proposition}[lemma]{Proposition}
\newtheorem{theorem}[lemma]{Theorem}
\newtheorem{oproblem}{Open Problem}

\theoremstyle{definition}

\newtheorem{remark}[lemma]{Remark}
\newtheorem{example}[lemma]{Example}
\newtheorem{definition}[lemma]{Definition}

\numberwithin{equation}{section}

\usepackage{orcidlink}

\newcommand{\R}{\mathbb{R}}

\newcommand{\mc}[1]{\mathcal{#1}}

\DeclareMathOperator{\dist}{dist}
\DeclareMathOperator*{\argmin}{argmin}
\DeclareMathOperator{\sgn}{sgn}

\def\eps{\varepsilon}

\def\norm#1{\|#1\|}

\newcommand{\wto}{\rightharpoonup}

\title[Optimal power concavity of the variational mean curvature]{{Optimal power concavity of the
variational mean curvature %
\\ via singular $1$-Laplace problems}}

\author[M. Gallo]{Marco Gallo \orcidlink{0000-0002-3141-9598}}
\author[F. Petitta]{Francesco Petitta \orcidlink{0000-0001-9265-7912}}
\author[M. Squassina]{Marco Squassina \orcidlink{0000-0003-0858-4648}}

\address[M. Gallo, M. Squassina]{\newline\indent Dipartimento di Matematica e Fisica
	\newline\indent
	Università Cattolica del Sacro Cuore
	\newline\indent
	Italy, Brescia, BS, Via della Garzetta 48, 25133}
\email{\href{mailto:marco.gallo1@unicatt.it}{marco.gallo1@unicatt.it}}
\email{\href{mailto:marco.squassina@unicatt.it}{marco.squassina@unicatt.it}}

\address[F. Petitta]{
\newline\indent Dipartimento di Scienze di Base e Applicate per l'Ingegneria
\newline\indent
Sapienza Università di Roma
\newline\indent
Italy, Roma, RM, Via Antonio Scarpa 16, 00161}
\email{\href{mailto:francesco.petitta@uniroma1.it}{francesco.petitta@uniroma1.it}}

\thanks{\emph{Fundings.}
All the authors are members of INdAM-GNAMPA.
M. Gallo is supported by INdAM-GNAMPA Project
``Aspetti geometrici e qualitativi di equazioni ellitiche e paraboliche", codice CUP \#E53C25002010001\#.
\bigskip
}


\begin{document}

\begin{abstract}
{Let $\Omega \subset \R^N$ be a bounded convex domain and let $H_{\Omega}\leq 0$ denote its
\emph{variational mean curvature}, in the sense of Barozzi, Gonzalez and Tamanini \cite{BGT87}. We prove that
$-H_{\Omega}$ is \emph{harmonic convex}, that is $1/H_{\Omega}$ is a convex function in $\Omega$, and that the
exponent $-1$ is optimal in the class of convex domains. This strictly improves the concavity of $H_{\Omega}$
obtained by Alter and Caselles \cite{AlCa09}{, who also proved, by geometric methods, a closely related property for the classical mean curvature under Minkowski interpolation of convex bodies.
Our approach is analytic:} the result is deduced from a concavity property of the solution of the singular $1$-Laplace Dirichlet problem
$$\begin{cases}
-\Delta_1 u = {u^{-\gamma}} & \quad \hbox{in $\Omega$}, \\
u = 0 &\quad \hbox{on $\partial \Omega$},
\end{cases}
$$
for $\gamma \in (0,1]$, whose unique {positive} solution is $u=(-H_{\Omega})^{-1/\gamma}$. We show that $u^{\gamma}$ is
concave, and we obtain this by proving that the solutions $u_p$ of the approximating problems
$-\Delta_p u_p = u_p^{-\gamma}$ are $\frac{p-1+\gamma}{p}$-concave on \emph{every} bounded convex domain, {without any boundary regularity,} and
then letting $p\to 1^{+}$. {Explicit examples show that no H\"older continuity of $H_{\Omega}$ up to $\partial\Omega$ holds {uniformly} in the class of convex domains with $C^{1,1}$ boundary.}}
\end{abstract}

\maketitle

\vspace{-1em}

\begin{center}
\begin{minipage}{10.5cm}
		\small
		\tableofcontents
\end{minipage}
\end{center}

\vspace{1em}

\section{Introduction}
\label{sec_intro}

Let $\Omega \subset \R^N$, $N \geq 2$, be a bounded convex domain, and for $\lambda>0$ let $\Omega_{\lambda}$
be a solution of the minimization problem
$\min_{E \subset \Omega} \left( P(E)+\lambda |\Omega \setminus E|\right)$ or, equivalently, of
\begin{equation}\label{eq_min_Omegalamda}
\min_{E \subset \Omega} \left( P(E)-\lambda |E|\right),
\end{equation}
$P(E)$ denoting the perimeter of a measurable set $E$ in $\R^N$.
{Denoting by $h(\Omega)$ the Cheeger constant of $\Omega$ (see Section \ref{sec_cheeg}), {for $\lambda<h(\Omega)$, the unique minimizer of \eqref{eq_min_Omegalamda} is the trivial one $\Omega_{\lambda}:=\emptyset$,}
while there exists a unique and convex minimizer $\Omega_{\lambda}$ for
$\lambda>h(\Omega)$ \cite[Proposition 2.4]{AlCa09}; when $\lambda=h(\Omega)$ there exists a unique non-trivial set $K$ of $\Omega$ that minimizes \eqref{eq_min_Omegalamda} (see \cite[Theorem 1]{AlCa09}) and it is called the Cheeger set of $\Omega$, so we set $\Omega_{h(\Omega)}:=K$.
Notice that uniqueness holds once the boundary of $\Omega_{\lambda}$ (which has zero measure) is considered included or not; here we use the convention of $\Omega_{\lambda}$ as \emph{open} set (see also Remark \ref{rem_openclosed}).

 In this way $\Omega_{\lambda}$ is unambiguously defined for every
$\lambda>0$, it is convex, and $\Omega_{\lambda} \subseteq \Omega_{\mu}$ whenever $\lambda<\mu$
\cite[Lemma 4 (i)]{ACC05}, \cite[(2.15)]{GoMa94}{; we also recall that, for convex $\Omega$, one has $\Omega_\lambda=\Omega$ for $\lambda$ large
if and only if $\partial\Omega\in C^{1,1}$, and in this case exactly for
$\lambda\geq\max\{P(\Omega)/|\Omega|,(N-1)\norm{\kappa}_{L^\infty(\partial\Omega)}\}$
(see \cite[Theorem 2]{AlCa09} and Remark \ref{rem_ifflinf}).
This is illustrated in Figure \ref{fig_ellipse_square}: for the ellipse $\Omega_\lambda$
fills $\Omega$ as soon as $\lambda$ reaches the maximal curvature of $\partial\Omega$, while
for the square $\Omega_\lambda$ increases to $\Omega$ but never coincides with it.}}

\begin{figure}[htbp]
\centering
\begin{tikzpicture}

\begin{scope}[shift={(0,0)}]

\draw[thick,dashed,color=gray]
 (0,0) ellipse (3cm and 1cm);

\draw[thick,fill=green!20,draw=green!60!black]
 (2.1213,-0.7071)
 arc (-71.565:71.565:0.7453)
 arc (45:135:3 and 1)
 arc (108.435:251.565:0.7453)
 arc (225:315:3 and 1)
 -- cycle;

\node[gray,font=\small]
 at (-2.65,0.80) {$\Omega$};

\node[green!60!black,font=\small]
 at (0,0) {$\Omega_\lambda$};

\end{scope}

\begin{scope}[shift={(5.3,0)}]

\draw[thick,dashed,color=gray]
 (-1.5,-1.5) rectangle (1.5,1.5);

\draw[thick,fill=green!20,draw=green!60!black]
 (-1,1.5)
 -- (1,1.5)
 arc (90:0:0.5)
 -- (1.5,-1)
 arc (360:270:0.5)
 -- (-1,-1.5)
 arc (270:180:0.5)
 -- (-1.5,1)
 arc (180:90:0.5)
 -- cycle;

\node[gray,font=\small]
 at (-1.27,1.73) {$\Omega$};

\node[green!60!black,font=\small]
 at (0,0) {$\Omega_\lambda$};

\end{scope}

\end{tikzpicture}

\caption{{An elongated ellipse (left, $\partial\Omega\in C^\infty$, not calibrable) and a
square (right, $\partial\Omega\notin C^{1,1}$), with the corresponding $\Omega_\lambda$ for
some $\lambda>h(\Omega)$: the free boundary $\partial\Omega_\lambda\cap\Omega$ consists of
circular arcs of radius $1/\lambda$. For the ellipse $\Omega_\lambda=\Omega$ for every
$\lambda\geq a/b^2$, the maximal curvature of $\partial\Omega$; for the square
$\Omega_\lambda\neq\Omega$ for every $\lambda$, the corners being always rounded off at
radius $1/\lambda$. See Example \ref{ex_examp_optimal} for the related computations.}}
\label{fig_ellipse_square}
\end{figure}
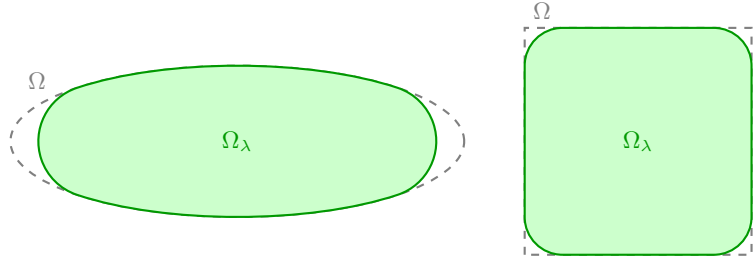

The \emph{canonical variational mean curvature} of $\Omega$ (see \cite{BGT87, ACC05} and
\cite{GMT83, GoMa94, Bar94, BaMa13})
is then defined as the function $H_{\Omega} :\mathbb{R}^{N}\to [-\infty, 0]$ given by
\begin{equation}\label{eq_def_Homega}
H_{\Omega}(x):= \begin{cases}
- \inf\{ \lambda>0 \mid x \in \Omega_{\lambda}\} & \quad \hbox{if $x \in \Omega$,} \\
0 & \quad \hbox{if $x \in \mathbb{R}^{N}\backslash \Omega $.}
\end{cases}
\end{equation}
{Throughout the paper we adopt the sign convention of \cite[(2.1)]{GoMa94}: a function $H\in L^1(\R^N)$ is
called \emph{a} variational mean curvature of a set $E$ of finite perimeter if $E$ minimizes the
\emph{Massari functional}
\begin{equation}\label{eq_massari_funct}
\mc{F}_H(X):= P(X) + \int_X H(x) \, dx
\end{equation}
among all sets $X\subset \R^N$ of finite perimeter, and we write $H\in{\mathscr{V}}^1(E)$; with this convention
$H_{\Omega}\leq 0$ inside $\Omega$. Part of the literature (\cite{Mas75}, \cite[\S 2.4]{ScSc23}) uses instead
$P(X)-\int_X H$, so that the optimal curvature is non-negative: all the comparisons below with
\cite{ACC05,AlCa09}, where $\kappa \geq 0$ denotes the classical mean curvature, have to be read accordingly.}

{The function $H_{\Omega}$ defined in \eqref{eq_def_Homega} belongs to ${\mathscr{V}}^1(\Omega)$.
}
Moreover \cite[Theorems 2.2, 2.3]{GoMa94}, \cite[Theorem 2.7]{ScSc23} imply
\begin{equation}\label{eq_optimality_H}
{\norm{H_{\Omega}}_{L^1(\Omega)}= P(\Omega) \leq \norm{H}_{L^1(\Omega)} \quad \hbox{for every }
H \in {\mathscr{V}}^1(\Omega),}
\end{equation}
{and, if ${\mathscr{V}}^p(\Omega):=\{H \in {\mathscr{V}}^1(\Omega) \mid H \in L^p(\Omega)\}\neq \emptyset$ for
some $p>1$ (see Remark \ref{rem_ifflinf}), then $H_{\Omega}$ is the \emph{unique} minimizer of the $L^p(\Omega)$-norm in
${\mathscr{V}}^p(\Omega)$. On the contrary, the minimizer of the $L^1$-norm is in general not unique, see
\cite[Examples 2.1, 2.2]{GoMa94}.}

The minimality expressed by \eqref{eq_optimality_H} is a distinctive feature of $H_{\Omega}$.

Moreover,  the function $H_{\Omega}$ is constant, equal
to $-h(\Omega)$, on the Cheeger set of $\Omega$, see \eqref{eq_H_const_on_Cheeger} below and
\cite[Introduction]{MAOP25}.

{We stress that $H_{\Omega}$ is not the classical mean curvature: by
\cite[Remark 2.1]{GoMa94}, if $H \in {\mathscr{V}}^1(E)$ is \emph{continuous} at a point $x \in \partial E$ near
which $\partial E$ is smooth, then $H(x)=-(N-1)\kappa(x)$, $\kappa$ being the classical mean curvature of
$\partial E$ at $x$; however $H_{\Omega}$ is in general discontinuous across $\partial \Omega$, and this
identity fails for it. Already for $\Omega=B_r$ one has $H_{B_r}\equiv -N/r$ in $B_r$, while
$-(N-1)\kappa \equiv -(N-1)/r$ (see Example \ref{ex_examp_optimal}); this is {consistent} with
\eqref{eq_optimality_H}, since $\norm{H_{B_r}}_{L^1(B_r)}=\frac{N}{r}|B_r| =P(B_r)$, whereas a naive constant extension of the classical
curvature would give the strictly smaller value $\frac{N-1}{N}P(B_r)$.

See anyway Lemma \ref{lem_explicit_Homega} for a relation between $H_{\Omega}$ and $\kappa$ in a particular case.
Our results below are therefore statements about
$H_{\Omega}$, and not about the classical mean curvature of $\partial \Omega$.}

{Let us highlight that in this paper we will be interested in $(H_{\Omega})_{|\Omega}$, thus the value given in $\R^N \setminus \Omega$ in \eqref{eq_def_Homega} will not have a strong role; indeed, other choices can be made in the definition of $(H_{\Omega})_{|\R^N \setminus \Omega}$, see e.g. \cite{GoMa94, Bar94}.}
{Besides possible discontinuities across $\partial \Omega$, let us mention that} {in \cite[Remark 2.7]{GoMa94} it is stated, with no proof and referring to a generalized maximum
principle, that $H_{\Omega}$ is continuous in $\R^N\setminus \partial \Omega$.}

Finally, the following representation holds \cite[(2.10)]{GoMa94}:

\begin{equation}\label{eq_represent_Omegalambda}
\begin{aligned}
\Omega_{\lambda} &= \operatorname{int}\big\{x \in \Omega \mid H_{\Omega}\geq -\lambda\big\},
\\
& = \big\{x \in \Omega \mid H_{\Omega} > -\lambda\big\} \quad \hbox{if $\lambda>h(\Omega)$};
\end{aligned}
\end{equation}
{notice moreover that
$\{x \in \Omega \mid H_{\Omega}>-\lambda\} =\bigcup_{\mu<\lambda}\Omega_\mu$ for each $\lambda>0$.}
 By convexity of $\Omega_{\lambda}$, \eqref{eq_represent_Omegalambda} implies that $H_{\Omega}$ is \emph{quasi-concave} {(see Definition \ref{def_theta_conc})}.
Actually, in
\cite[Proof of Proposition 2.2 and Remark 2.5]{AlCa09} it is shown that $H_{\Omega}$ is \emph{concave} {(see also Remark \ref{rem_AlCa_related} below)}, which in turn implies
$H_{\Omega} \in Lip_{loc}(\Omega)$.

\smallskip

{Aim of this paper is to improve the above concavity property up to the optimal one, showing that
$-H_{\Omega}$ is \emph{harmonic convex}. To this aim, we exploit the quite intriguing connection between
$H_{\Omega}$ and the solution of a singular Dirichlet problem involving the $1$-Laplacian.}

Consider indeed
\begin{equation}\label{eq_gen_1laplace_a1}
\begin{cases}
-\Delta_1 u = {u^{-\gamma}} & \quad \hbox{in $\Omega$}, \\
u = 0 &\quad \hbox{on $\partial \Omega$},
\end{cases}
\end{equation}
where $\Delta_1 u := \operatorname{div}\big(Du/|Du|\big)$ and {$\gamma>0$}.
{Curiously,} the solution of \eqref{eq_gen_1laplace_a1} {(which is unique by \cite[Theorem 3.5]{MAOP25})} turns out to be {explicit} and related to the notion of
variational mean curvature: indeed, if $\Omega$ is convex, then (see \cite{GOP22,DGOP19,MAOP25} and Section \ref{sec_singular} below)
\begin{equation}\label{eq_sol_eq_Hom}
u(x) = (-H_{\Omega}(x))^{-\frac{1}{\gamma}}.
\end{equation}

{Let us explain why problem \eqref{eq_gen_1laplace_a1} {arises as}
relevant model here. Concavity properties are classical, and more studied, for \emph{sub-homogeneous} problems $-\Delta_pu=u^q$, $q<p-1$ \cite{Mak71, BrLi76, Kor83, Ken85, Sak87, CaSp82, Kaw85, IST20}. For $p=1$ the homogeneity threshold is $q=p-1=0$: sub-homogeneous problems $-\Delta_1u=u^{q}$, $q<0$, are therefore necessarily \emph{singular}, while the borderline case $q=0$, namely the torsion problem $-\Delta_1u=1$, is homogeneous and quite degenerate (see Section \ref{sec_torsion}). {Moreover},
for every $p>1$ the singular problems $-\Delta_pu=u^{-\gamma}$, which we shall use to approximate \eqref{eq_gen_1laplace_a1}, are sub-homogeneous {as well}, and power concavity results are available for them \cite{GaSq25}.}

\medskip

The main result of this paper is the following regularity/concavity property for solutions of
\eqref{eq_gen_1laplace_a1}. {We focus first on $\gamma \in (0,1]$.}

\begin{theorem}
\label{thm_main_ucgamma}
Let $\Omega \subset \R^N$, $N\geq 2$, be open, bounded and convex, and let $\gamma \in (0, 1]$. Then the
solution $u>0$ of \eqref{eq_gen_1laplace_a1} satisfies
\begin{equation}\label{eq_double_bound_u}
\left(\frac{\dist(x,\partial \Omega)}{N}\right)^{\frac{1}{\gamma}} \leq u(x) \leq
\left(\frac{1}{h(\Omega)}\right)^{\frac{1}{\gamma}} ,
\end{equation}
and it is $\gamma$-concave, i.e. $u^{\gamma}$ is concave.
{As a consequence, $u^{\gamma} \in Lip_{loc}(\Omega)$ and
$u \in Lip_{loc}(\Omega)$.}
\end{theorem}

{The bound from below in \eqref{eq_double_bound_u} can actually be improved, see \eqref{eq_lower_bound_u}.}

{
\begin{remark}\label{rem_no_regularity_needed}
No regularity assumption on $\partial \Omega$ is needed in Theorem \ref{thm_main_ucgamma} beyond convexity.
This is an improvement with respect to \cite[Theorem 6.7]{GaSq25}, where $\partial \Omega \in C^{2,\alpha}$ is
required for the corresponding statement when $p>1$: see Theorem \ref{thm_GS_autonom_imp} below, which is what
makes the present generality possible. {The key tool is the stability result of Section \ref{sec_approx} (Lemma \ref{lem_conv_minim} and Proposition \ref{prop_approx_general}), which transfers concavity properties of $p$-Laplace problems from smooth, strongly convex domains and regular data to arbitrary convex domains and less regular data, and which we also apply in Section \ref{sec_further}.}
We stress that the conclusion is stated in $\Omega$, and that one has to distinguish between the local
regularity of $u_{|\Omega}$ and the behaviour of $u$ across $\partial \Omega$: the latter is genuinely
discontinuous in general ({as already mentioned,} %
see Example
\ref{ex_examp_optimal}$(i)$), while {the geometric conditions on $\Omega$ ensuring that $u_{|\Omega}$ is globally Lipschitz, or merely continuous up to $\partial\Omega$, are, to the best of our knowledge, unknown}; see Proposition \ref{prop_no_lip} and Open Problem \ref{op_ext_bound}.
\end{remark}
}

As a consequence of Theorem \ref{thm_main_ucgamma} and of \eqref{eq_sol_eq_Hom} we obtain the announced
improvement of the concavity properties of $H_\Omega$.

\begin{corollary}[{Harmonic convexity of $-H_{\Omega}$}]
\label{cor_var_mean_curv}
Let $\Omega \subset \R^N$, $N\geq 2$, be open, bounded and convex. Then $-H_{\Omega}$ is harmonic convex,
that is $(-H_{\Omega})^{-1}$ is concave in $\Omega$ (equivalently, $H_{\Omega}^{-1}$ is convex); in particular
$H_{\Omega}^{-1} \in Lip_{loc}(\Omega)$. Moreover the exponent $-1$ is optimal in the class of convex domains, {i.e. there exists a domain such that $-H_{\Omega}$ is not $\theta$-convex for any $\theta<-1$}.
\end{corollary}

{Let us highlight that the counterexample for optimality in Corollary \ref{cor_var_mean_curv} is explicit, and given in Example \ref{ex_examp_optimal}$(iii)${; in this example, as well as in Example \ref{ex_examp_optimal}$(ii)$, the harmonic convexity inequality holds with \emph{equality} along explicit segments, so that the exponent $-1$ is not only non-improvable, but saturated}.}
{Since harmonic convexity implies convexity (see Remark \ref{rem_-f_impl_concav}), Corollary
\ref{cor_var_mean_curv} recovers the concavity of $H_{\Omega}$ of \cite[Proposition 2.2]{AlCa09}, and strictly improves it, reaching the optimal power concavity.
}

\smallskip
{We also highlight that Corollary \ref{cor_var_mean_curv} is of interest when $\Omega$ is not calibrable (see Section \ref{sec_cheeg}),
since otherwise $H_{\Omega}$ is constant; the elongated ellipse of Example
\ref{ex_examp_optimal}$(iv)$ shows that this property may fail also for smooth domains, and provides an example in which $H_{\Omega}$ is bounded, non-constant and
harmonic convex.}

\medskip
{A geometric restatement of Corollary \ref{cor_var_mean_curv} is the following inclusion for the Minkowski
combinations of the sets $\Omega_{\lambda}$.}

\begin{corollary}[Minkowski sum of $\Omega_{\lambda}$]\label{cor_minkowski}
{For any $\lambda_1,\lambda_2\geq h(\Omega)$ and $t \in [0,1]$ we have}
\begin{equation}\label{eq_minkow_sum_int}
t \Omega_{\lambda_1} \oplus (1-t)\Omega_{\lambda_2} \subseteq \Omega_{\frac{1}{\frac{t}{\lambda_1}+\frac{1-t}{\lambda_2}}} \subseteq \Omega_{t\lambda_1+(1-t)\lambda_2},
\end{equation}
where $\Omega_{\lambda}$ is the minimizer of \eqref{eq_min_Omegalamda}.
\end{corollary}

The first inclusion in Corollary \ref{cor_minkowski} is equivalent to the harmonic convexity of $-H_{\Omega}$
of Corollary \ref{cor_var_mean_curv}, while the second one is equivalent to the concavity of $H_{\Omega}$ of
\cite[Proposition 2.2]{AlCa09}. The restriction $\lambda_i \geq h(\Omega)$ only serves to exclude the trivial case
$\Omega_{\lambda_i}=\emptyset$. {See Remark \ref{rem_equal_mink} for comments about equality cases in \eqref{eq_minkow_sum_int}.}

{\begin{remark}\label{rem_AlCa_related}
The harmonic convexity of $-H_{\Omega}$ is not entirely surprising: indeed, in \cite[Theorem 3 and Corollary 4.2]{AlCa09} the authors prove a closely related property for the classical mean curvature under Minkowski interpolation of convex bodies, namely that the reciprocal of the (maximal) mean curvature of $(1-t)K\oplus tL$ is concave in $t$, i.e.\ a harmonic convexity property in the parameter $t$; we recall that, for convex sets, the classical mean curvature is linked to calibrability, and hence to $h(\Omega)$, through \eqref{kappa} below. Their approach is mainly geometric, and relies on deep tools from convex geometry and geometric measure theory: Hessians of support functions and principal radii of curvature, sharp matrix inequalities, Blaschke's rolling ball theorem, Hausdorff stability of curvature bounds, approximation of $C^{1,1}$ convex bodies via generalized mean curvature flow and Almgren--Taylor--Wang minimizing movements, together with the calibrability theory developed in \cite{ACC05}. Combining these tools with the characterization of calibrable convex sets, along the lines of \cite[Proposition 6.1]{AlCa09}, one {may expect to}
obtain a geometric proof of Corollary \ref{cor_minkowski}; to the best of our knowledge, however, the harmonic convexity of $-H_\Omega$ is not explicitly stated in the literature. Moreover, our approach is of a different, analytic nature: it relies on the stability of concavity properties for $p$-Laplace problems proved in Section \ref{sec_approx}, and on the passage to the limit $p\to1^+$ in the singular problems \eqref{eq_gen_plaplace_a1}.

{Let us also mention that a Brunn--Minkowski-type inequality with exponent $-1$ for the Cheeger constant is proved in \cite[Lemma 4.4]{Fto25}.}
\end{remark}}

As a {further} consequence of Corollary \ref{cor_var_mean_curv} and relation \eqref{eq_sol_eq_Hom}, we observe that the concavity information {for \eqref{eq_gen_1laplace_a1}} {holds} for any $\gamma$.
\begin{corollary}[Extension to arbitrary $\gamma>0$]
\label{cor_ext_gammamag1}
{Let $\Omega\subset\R^N$ be open, bounded and convex, and let {$\gamma>0$}.
Then the function $u:=(-H_\Omega)^{-1/\gamma}$ belongs to $BV(\Omega)\cap L^\infty(\Omega)$, it is a solution of \eqref{eq_gen_1laplace_a1} (in the sense of Definition \ref{def_solution} below), it satisfies \eqref{eq_double_bound_u}, $u^\gamma$ is concave and $u\in Lip_{loc}(\Omega)$.}
\end{corollary}
Let us remark that, {unlike} Theorem \ref{thm_main_ucgamma}, the validity of $u \in Lip_{loc}$ in Corollary \ref{cor_ext_gammamag1} relies also on the bound from below of $u$. An important feature of Corollary \ref{cor_ext_gammamag1} is that the corresponding result for $p>1$, {$p\neq 2$}, is indeed not available and remains open (see Open Problem \ref{op_gammamag1}).

\medskip

{\noindent \emph{Plan of the paper.} In Section \ref{sec_prelim} we collect the necessary background:
variational mean curvature and
Cheeger sets (\S \ref{sec_cheeg}), the notion of solution for $1$-Laplace problems (\S \ref{sec_notation}), {the connection \eqref{eq_sol_eq_Hom} (\S \ref{sec_singular})}, and the notions of
power concavity together with the results of \cite{GaSq25} and \cite{Roc97} that we shall use
(\S \ref{sec_concavity_notions}). Section \ref{sec_approx} contains a general approximation result
(Proposition \ref{prop_approx_general}) which allows to transfer concavity properties from smooth domains and
smooth data to general convex domains, and its consequences for $p>1$. In Section \ref{sec_limit} we let
$p \to 1^{+}$ and prove Theorem \ref{thm_main_ucgamma} and {related Corollaries}.
Section
\ref{sec_further} contains some further remarks and extensions, and Section \ref{sec_open} some open problems. {Explicit examples are collected in Appendix \ref{sec_examples}.}}

\section{Preliminaries}
\label{sec_prelim}

{\noindent \emph{Notations.} Throughout this paper, $\Omega \subset \R^N$ ($N\geq2$) is an open bounded set, $|E|$ and $P(E)$ denote the Lebesgue
measure and the perimeter of a measurable set $E$, $\omega_N:=|B_1|$, and $\mathcal{H}^{N-1}$ is the
$(N-1)$-dimensional Hausdorff measure. From now on, $B_r(c)$ will denote the \emph{open} ball in $\R^N$.

For $\rho>0$ we denote $d_{\Omega}(\cdot):=\dist(\cdot,\partial \Omega)$ and by
\begin{equation}\label{eq_def_inner_parall}
\Omega^{\rho}:=\{x \in \Omega \mid d_{\Omega}(x)\geq \rho\}
\end{equation}
the {(closed)} inner parallel set of $\Omega$ at distance $\rho$, not to be confused with the minimizing sets $\Omega_{\lambda}$ defined through \eqref{eq_min_Omegalamda}.
{Finally, whenever the (classical) mean curvature $\kappa(x)$ is well defined and bounded for a set $\Omega$, we define
$$\bar{\kappa}:={\operatorname*{ess\,sup}_{x\in\partial \Omega}} \, \kappa(x);$$
here we intend $\kappa(x)$ as its normalized value $\kappa= \frac{1}{N-1} \sum_i \kappa_i$, $\kappa_i$ being principal curvatures.
}

\medskip
\subsection{Cheeger sets and variational mean curvature}
\label{sec_cheeg}

Let $\Omega $ be an open set of $\re^N$: a \emph{Cheeger set} $K$ of $\Omega$ is a subset of $\Omega$ that
minimizes the problem
\begin{equation}\label{che_2}
\min_{E \subset \Omega} \left( P(E)-\lambda_K |E|\right)
\qquad \hbox{with}\qquad \lambda_K=\frac{P(K)}{|K|}.
\end{equation}
It is known \cite[Theorem {1}]{AlCa09} {(see also \cite{CCN07})} that in any non-trivial convex body $\Omega$ there exists a unique
Cheeger set $K\subset \Omega$ and that, moreover, $K$ is convex and $\partial K$ is of class $C^{1,1}$.
The (unique) $\lambda_K$ is thus called the \emph{Cheeger constant} of $\Omega$, denoted by $h(\Omega)$.
When $\lambda_K$ is substituted by a general $\lambda>0$, we obtain $\Omega_{\lambda}$ as minimizer, already discussed in \eqref{eq_min_Omegalamda}; {see also \cite{StZi97} for related isoperimetric problems in convex boxes}.

If $\Omega$ itself minimizes \eqref{che_2} with $\lambda_K=\frac{P(\Omega)}{|\Omega|}$, then $\Omega$ is called
\emph{self-Cheeger}, and in this case
$
h(\Omega)=\frac{P(\Omega)}{|\Omega|}.
$
For a convex set $\Omega$, being self-Cheeger is equivalent \cite[Proposition 2]{ACC05} to being a
{\it calibrable set} in the sense of \cite{ACC05}.
{By \cite[Theorem 4]{BCN02} (see also \cite[Corollary 1]{ACC05}, \cite[Theorem 2]{KaLa06} and \cite[Remark 2.8]{AlCa09}),
a non-trivial
convex body $\Omega \subset \R^N$ is calibrable if and only if $\partial \Omega \in C^{1,1}$ and the {(classical)} mean
curvature $\kappa$ of $\partial \Omega$ satisfies}
 \begin{equation}\label{kappa}
(N-1)\, {\operatorname*{ess\,sup}_{x\in\partial \Omega}} \, \kappa (x)\leq \frac{P(\Omega)}{|\Omega|},
\end{equation}
or, equivalently, if the right-hand side is substituted by $h(\Omega)$.
{In particular, by the regularity results of Massari type \cite{Mas75,GMT83,CCN10}, every self-Cheeger set
has $C^{1,1}$ boundary; the converse is false without \eqref{kappa}, the standard counterexample being an
elongated ellipse, whose curvature at the poles exceeds $h(\Omega)/(N-1)$ (see Example
\ref{ex_examp_optimal}$(iv)$)}.

{Recall $H_{\Omega}$ defined in \eqref{eq_def_Homega}: we see that $H_{\Omega}$ belongs to ${\mathscr{V}}^1(\Omega)$, i.e. it is a variational mean curvature. Indeed, by
\cite[Lemma 2.2]{GoMa94}, ${\Omega}$
minimizes $\mc{F}_{H_{\Omega}}$ among the subsets of $\Omega$; on the other hand,
since $\Omega$ is convex one has $P(X)\geq P(X\cap \Omega)$ for every set $X$ of finite perimeter, while
$\int_X H_{\Omega}=\int_{X\cap \Omega} H_{\Omega}$ because $H_\Omega$ vanishes outside $\Omega$, so that
$$\mc{F}_{H_{\Omega}}(X)\geq \mc{F}_{H_{\Omega}}(X\cap \Omega)\geq \mc{F}_{H_{\Omega}}(\Omega)$$
for every such $X$, and hence the claim.}
{Recalling
{the definitions \eqref{eq_min_Omegalamda} and \eqref{eq_def_Homega},}
a relevant property, which follows at once from \eqref{eq_def_Homega}, $\Omega_{\lambda}=\emptyset$ for $\lambda<h(\Omega)$ and
$\Omega_{h(\Omega)}=K$, is that}
\begin{equation}\label{eq_H_const_on_Cheeger}
{H_{\Omega} \equiv - h(\Omega) \quad \hbox{ on the Cheeger set $K$ of $\Omega$}}
\end{equation}
and $H_{\Omega}\neq -h(\Omega)$ on $\Omega \setminus \overline{K}$;
{in particular $H_{\Omega}$ is constant on the whole of $\Omega$ if and only if $\Omega$ is calibrable
\cite{ACC05}. This is the reason why the results of the present paper are of interest precisely for
non-calibrable domains.}
{Notice moreover that, by monotonicity of $\Omega_{\lambda}$ w.r.t. $\lambda$ we have
\begin{equation}\label{eq_estim_H_dall'alto}
H_{\Omega} \leq -h(\Omega) \quad \hbox{on $\Omega$}.
\end{equation}
{We refer to \cite{Par11,Leo15} for surveys on the Cheeger problem.}}

\smallskip

{We introduce now two tools which will be helpful in getting explicit estimates and representation formulas for $\Omega_{\lambda}$ and $H_{\Omega}$.}

{Recalling} \eqref{eq_def_inner_parall}, for any $\rho>0$ we set
\begin{equation}\label{eq_def_Arho}
A_{\rho}:= \Omega^{\rho}\oplus B_{\rho}(0) = \bigcup\big\{B_{\rho}(c) \mid B_{\rho}(c)\subseteq \Omega \big\}
\end{equation}

the Minkowski sum of $\Omega^{\rho}$ with $B_\rho$, that is, the union of all the balls
of radius $\rho$ contained in $\Omega$; notice that, being $B_{\rho}(0)$ open and $\Omega$ convex, it is not important if \eqref{eq_def_inner_parall} is defined as closed or open set as long as $\rho$ is strictly smaller than the inradius (compare \cite{KaLa06, LeSa20, LNS17,LePr16}).

Let moreover
\begin{equation}\label{eq_def_rOmega}
r_{\Omega}(x) := \sup\big\{\rho>0 \mid \exists\, B_{\rho}(c)\subseteq \Omega \hbox{ with } x \in
\overline{B_{\rho}(c)}\big\}
\end{equation}
be the radius of the largest ball
contained in $\Omega$ and {containing}
$x \in \Omega${; ${2}r_\Omega$ is known in image analysis as the \emph{local thickness} of $\Omega$ \cite{HiRu97}, and it is closely related to the morphological opening of $\Omega$ by balls
\cite{Ser82}.}

{
Recalled $d_{\Omega}(\cdot)=\dist(\cdot,\partial \Omega)$ the
distance function, we have the following preliminary lemma.
\begin{lemma}[Characterization of $r_{\Omega}$]
\label{lem_char_rOmega}
Let $\Omega \subset \R^N$ be open, bounded and convex.
We have
\begin{equation}\label{eq_rOmega_equiv}
r_{\Omega}(x) = \max_{c \in \Omega}\big\{ d_{\Omega}(c) \mid |x-c|\leq d_{\Omega}(c)\big\},
\end{equation}
and \eqref{eq_def_rOmega} is a maximum. Moreover
\begin{equation}\label{eq_rOmega_Arho}
r_{\Omega}(x) = \sup\{ \rho>0 \mid x \in A_{\rho}\}.
\end{equation}
Finally, it holds that $r_{\Omega}\geq d_{\Omega}$.
\end{lemma}

\begin{proof}
Observing that the largest  ball centred at $c$ has radius exactly $d_{\Omega}(c)$, one has \eqref{eq_rOmega_equiv} with supremum in place of maximum. To show that it is attained, we argue by compactness, since the admissible set $\{c \in \Omega \mid d_{\Omega}(c)-|x-c|\geq 0\} \subset \{c \in \Omega \mid d_{\Omega}(c) \geq \frac{1}{2}d_{\Omega}(x) >0\}
$
is closed and stays at positive distance from $\partial \Omega$. This moreover implies that \eqref{eq_def_rOmega} is a maximum as well.

We now verify \eqref{eq_rOmega_Arho}: $\geq$ is clear, thus we prove $\leq$.
Let $r= r_{\Omega}(x)$ with $x \in \overline{B_r(c)}$, $B_r(c) \subset \Omega$; the hardest case is $x \in \partial B_r(c)$, so we assume $|c-x|=r$.
 Let $\rho <r$ be fixed. We state that there exists $t_{\rho} \in (0,1)$ such that $x \in B_{\rho}(t_{\rho} c + (1-t_{\rho}) x) \subset \Omega$. First we have $|x-(t_{\rho} c + (1-t_{\rho}) x)| = t_{\rho} |c-x| \leq t_{\rho} r <\rho$ for $t_{\rho}<\rho/r$.

 Since $B_r(c) \subset \Omega$, we need to discuss the inclusion $B_{\rho}(t_{\rho} c + (1-t_{\rho}) x) \setminus B_r(c) \subset \Omega$.
 For $t=\rho/r$ we have that $B_{\rho}(t c + (1-t) x)$ is internally tangent to $B_r(c)$, with contact point in $x$; since $x \in \Omega$, there exists $B_{\delta}(x) \Subset \Omega$, thus for $t_{\rho}$ close to $\rho/r$ we have
 $$B_{\rho}(t_{\rho} c + (1-t_{\rho}) x) \setminus B_r(c) \subset B_{\delta}(x)\subset \Omega$$
 and hence the claim.

 Hence $x \in A_{\rho}$ and thus $\sup\{ \sigma>0 \mid x \in A_{\sigma}\}\geq \rho$, and the claim follows by sending $\rho$ to $r$.
\end{proof}
}

{We show first that $H_{\Omega}$ never takes the value $-\infty$ inside $\Omega$, with an optimal
quantitative bound. The statement, which will also be used to make Example \ref{ex_examp_optimal} of Appendix
\ref{sec_examples} precise, goes essentially
back to \cite[Remark 2.5]{GoMa94}; we follow the localized version of \cite[Lemma 2.2]{ScSc23}.}

{
\begin{proposition}[Lower bound for $H_{\Omega}$]
\label{prop_finiteness_H}
Let $\Omega\subset\R^N$ be open, bounded {and convex}. For every $x\in\Omega$,
\begin{equation}\label{eq_bound_H_dist}
H_\Omega(x)\ \geq \ -\frac{N}{r_{\Omega}(x)}\geq \ -\frac{N}{d_{\Omega}(x)}. %
\end{equation}
Moreover, the constant $N$ is optimal. %
\end{proposition}

\begin{proof}
Fix $x\in\Omega$; by \eqref{eq_rOmega_Arho}, %
there exists $\rho \nearrow
 r_{\Omega}$ and %
 $c \in \Omega$ %
such that $x \in B_{\rho}(c) \subset \Omega$.
 By
\cite[Lemma 2.2]{ScSc23} (see also \cite[Remark 2.5]{GoMa94}), for every $\lambda\geq N/\rho$ and every
$F \subseteq \Omega$ one has
$$P\big(F\cup B_{\rho}(c)\big) - \lambda \big|F \cup B_{\rho}(c)\big| \ \leq\ P(F)-\lambda|F|,$$
with strict inequality unless $B_{\rho}(c)\subseteq F$ when $\lambda>N/\rho$. Applying this with $F=\Omega_\lambda$,
which is a minimizer of \eqref{eq_min_Omegalamda}, we get $B_{\rho}(c)\subseteq\Omega_\lambda$, and in particular
$x \in \Omega_{\lambda}$, for every $\lambda>N/\rho$. Hence
$$-H_\Omega(x) = \inf\{\lambda>0 \mid x\in\Omega_\lambda\} \leq \frac{N}{\rho},$$
which, by sending $\rho \to r_{\Omega}(x)$, is \eqref{eq_bound_H_dist}. The constant is optimal since, for $\Omega=B_R(x_0)$, one has
$-H_\Omega(x_0)=\frac{N}{R}=\frac{N}{r_{\Omega}(x_0)}=\frac{N}{\dist(x_0,\partial\Omega)}$ (see Example
\ref{ex_examp_optimal}$(i)$). %
\end{proof}
}

\begin{remark}[On $H_{\Omega}$ bounded]
\label{rem_ifflinf}
When $\Omega$ is convex, we have that $H_{\Omega}\in L^{\infty}(\Omega)$ if and only if
$\partial \Omega \in C^{1,1}$.
{More precisely, in this case
$$\norm{H_\Omega}_{L^\infty(\Omega)}=\max\Big\{\frac{P(\Omega)}{|\Omega|},\,(N-1)\norm{\kappa}_{L^\infty(\partial\Omega)}\Big\}=:M_\Omega .$$
Indeed, if $\partial\Omega\in C^{1,1}$, by \cite[Theorem 2]{AlCa09} one has $\Omega_\mu=\Omega$ if and only if $\mu\geq M_\Omega$; hence $-H_\Omega\le M_\Omega$ in $\Omega$ by \eqref{eq_def_Homega}, while for $\lambda<M_\Omega$ the open convex set $\Omega_\lambda$ is strictly contained in $\Omega$, so that $\Omega\setminus\overline{\Omega_\lambda}$ is a non-empty open set on which $-H_\Omega\geq\lambda$. Conversely, if $H_\Omega\in L^\infty(\Omega)$, then $\Omega_\mu=\Omega$ for $\mu>\norm{H_\Omega}_\infty$ by \eqref{eq_represent_Omegalambda}, and $\partial\Omega\in C^{1,1}$ by the argument of \cite[Proposition 2.7]{AlCa09} (see also \cite{MoPe15}).}

{To understand what happens at a non-regular point, consider a square $\Omega \subset \R^2$ (see
Example \ref{ex_examp_optimal}$(iii)$). %
At a corner
$x_0\in\partial \Omega$ the geometric curvature is not defined in the classical sense; on the other hand the
free boundary $\partial \Omega_\lambda \cap \Omega$ near $x_0$ consists of a circular arc of radius $1/\lambda$
(see also Example \ref{ex_examp_optimal}$(ii)$), %
so that in order for
$x_n \to x_0$ to belong to
$\Omega_{\lambda}$ one must let $\lambda \to +\infty$. Therefore $H_{\Omega}(x_n) \to -\infty$, consistently
with Proposition \ref{prop_finiteness_H}, whose bound degenerates as $\dist(x_n,\partial \Omega)\to0$.}

Let us finally mention that, when $\partial \Omega \in C^{1,\alpha}$, $\alpha<1$, then it has been recently proved \cite{ScSc26} that $H_{\Omega} \in L^p(\Omega)$ for any $p<\frac{1+\alpha}{1-\alpha}$; in particular ${\mathscr{V}}^p(\Omega)\neq \emptyset$, and hence $H_{\Omega}$ is the unique minimizer of the $L^p(\Omega)$-norm in ${\mathscr{V}}^p(\Omega)$.
\end{remark}

In the planar case the whole family $\{\Omega_{\lambda}\}$, and hence $H_{\Omega}$ itself, admits a
completely explicit description, which we now record:
{it will turn all the computations of the examples in Appendix
\ref{sec_examples} into elementary geometry, and it will yield a direct proof of Corollary
\ref{cor_var_mean_curv} in the planar case, see Remark \ref{rem_elementary_planar}.}
From \cite[Theorem 2.3, Proposition 2.1, Lemma 5.1]{LeSa20} and \cite[Theorem 1 and Theorem 2]{KaLa06} {(see also \cite[Introduction]{AlCa09},
\cite[Theorem 1.2]{LeSa20})} one deduces the following characterization:

\begin{lemma}[Structure of $\Omega_{\lambda}$ in $\R^2$, \cite{LeSa20, KaLa06}]
\label{lem_planar_structure}
Let $\Omega \subset \R^2$ be open, bounded and convex, and let $\lambda \geq h(\Omega)$. Then the non-trivial
minimizer of \eqref{eq_min_Omegalamda} is unique and
\begin{equation}\label{eq_Omegalambda_planar}
\Omega_{\lambda} = A_{1/\lambda} = \Omega^{1/\lambda}\oplus B_{1/\lambda} .
\end{equation}
Moreover, in the limiting case $\lambda=h(\Omega)$, for which $\Omega_{h(\Omega)}=K$ is the Cheeger set of
$\Omega$,
 the radius
 $$R_{\Omega}:=1/h(\Omega)$$
 is characterized by the \emph{inner Cheeger formula}%
 \footnote{Indeed, the set $\Omega^{R_{\Omega}}$ is also called \emph{inner Cheeger set}.}
\begin{equation}\label{eq_inner_cheeger}
|\Omega^{R_{\Omega}}| = \pi R_{\Omega}^{2} .
\end{equation}

 Finally, the following three statements are equivalent
 \begin{enumerate}
 \item $\Omega$ is self-Cheeger,
 \item $\partial \Omega \in C^{1,1}$ and
 \begin{equation}\label{kappa_2}
 \bar{\kappa} \leq \frac{P(\Omega)}{|\Omega|},
 \end{equation}
 \item
 $\partial \Omega \in C^{1,1}$ and
 \begin{equation}\label{kappa_3}
 \bar{\kappa} \leq h(\Omega).
 \end{equation}
 \end{enumerate}
\end{lemma}

\begin{remark}\label{rem_planar_only}
Lemma \ref{lem_planar_structure} is genuinely two-dimensional. The limiting case $\lambda=h(\Omega)$ of
\eqref{eq_Omegalambda_planar} is the classical characterization of the Cheeger set of a planar convex body of
\cite[Theorem 1]{KaLa06}, later extended to Jordan domains without necks in \cite[Theorem 1.4]{LNS17}, and it
comes together with the inner Cheeger formula \eqref{eq_inner_cheeger}.
{For $N \geq 3$
no analogue is available: the free boundary of $\Omega_{\lambda}$ is a hypersurface of constant mean
curvature, which need not be spherical, so that $\Omega_{\lambda}$ is in general \emph{not} the inner parallel
body enlarged back {(see \cite[Section 6]{KaLa06}, \cite{Fto25} and references therein)}.
Explicit characterizations of Cheeger sets in dimension $N \geq 3$ are available only in very special
cases, such as sets of revolution, and are missing even for the unit cube; see \cite[Section 1]{CaSa23} and
\cite[Section 1]{Fto25} for a discussion.

 This is one of the reasons why the approach through the $1$-Laplacian, which is
dimension-free, is of interest; see also Remark \ref{rem_elementary_planar}.}
\end{remark}

{Lemma \ref{lem_planar_structure} yields a closed formula for $H_{\Omega}$.

\begin{lemma}[Formula for $H_{\Omega}$ in $\R^2$]
\label{lem_explicit_Homega}
We have
\begin{equation}\label{eq_formula_H_planar}
- H_{\Omega}(x) = \max\Big\{ h(\Omega), \frac{1}{r_{\Omega}(x)}\Big\},
\end{equation}
for every open bounded convex $\Omega \subset \R^2$.
In particular the Cheeger set $K$ verifies
\begin{equation}\label{eq_K_romega}
K = \operatorname{int} \left\{r_{\Omega} \geq
 \frac{1}{h(\Omega)}\right\}.
 \end{equation}

{Finally, assume $\partial\Omega\in C^{1,1}$,
and let $\bar x\in\partial\Omega\setminus\partial K$ be a point at which $\partial\Omega$ is twice differentiable with $\kappa(\bar x)=\bar\kappa$. Then}
\begin{equation}\label{eq_magic_coincid}
-H_{\Omega}(x) \to {\bar\kappa} \quad \hbox{for $x \to \bar{x}$}.
\end{equation}
{Notice that $\partial\Omega\setminus\partial K\neq\emptyset$ only if $\Omega$ is not calibrable, in which case $\bar\kappa>P(\Omega)/|\Omega|\geq h(\Omega)$ by the equivalences above.}
\end{lemma}
\begin{proof}
 As a consequence {of \eqref{eq_rOmega_Arho}}
$$\inf\{\lambda>0 \mid x \in A_{1/\lambda}\} = \frac{1}{r_{\Omega}(x)}$$
and combining it with
$$
\{\lambda>0 \mid x \in \Omega_{\lambda}\} =\{\lambda\geq h(\Omega) \mid x \in A_{1/\lambda}\} =[h(\Omega), +\infty) \cap \{\lambda>0 \mid x \in A_{1/\lambda}\}
$$
we have the claim.

To show \eqref{eq_magic_coincid}, notice that, for $x$ close to $\bar{x}$, we have $-H_{\Omega}(x)= 1/r_{\Omega}(x)$.

Let $x_n \in \Omega$, $x_n \to \bar{x}$, and $r_{\Omega}(x_n)=d_{\Omega}(c_n)$; %
by compactness $c_n \to \bar{c}$, and $\bar{x} \in \overline{B_{d_{\Omega}(\bar{c})}(\bar{c})}$, thus it must be $\bar{x} \in \partial B_{d_{\Omega}(\bar{c})}(\bar{c})$ and $d_{\Omega}(\bar{c})\leq 1/\bar{\kappa}$ (the osculating circle). Finally, considering $y_n$ the projection of $x_n$ on $\partial \Omega$, by Blaschke's rolling theorem \cite[Corollary 6]{Rev16} (see also \cite[Lemma 1.7]{LNS17}),
 there exists a ball of radius $1/\bar{\kappa}$ inside $\Omega$ and tangent $y_n$; for $n$ large, $x_n$ belongs to such a ball, thus $d_{\Omega}(c_n) \geq 1/\bar{\kappa}$ and the inequality passes to the limit, obtaining $d_{\Omega}(\bar{c})= 1/\bar{\kappa}$, which is the claim.
\end{proof}}

{Formula \eqref{eq_formula_H_planar} is essentially contained in \cite[Theorem 2]{ACC05b}, where it is shown that, for a planar bounded convex set $\Omega$, the minimizing total variation flow starting from $\chi_\Omega$ is given by $\big(1-t/\min\{R_\Omega,r_\Omega(x)\}\big)^+\chi_\Omega(x)$, while this flow is $(1+tH_\Omega)^+\chi_\Omega$ \cite[Section 5]{ACC05}, \cite{BCN02} {(see also \cite{ACM04})}; we have included the short direct proof above for the reader's convenience.}

Notice that \eqref{eq_magic_coincid} is \emph{not} an instance of
\cite[Remark 2.1]{GoMa94}, which requires $H_{\Omega}$ to be continuous at the boundary point and hence does
not apply, being $H_{\Omega}(\bar{x})=0$ by definition.

\begin{remark}[On the choice $\Omega_{\lambda}$ open]
\label{rem_openclosed}
We point out that our choice of $\Omega_{\lambda}$ open is arbitrary, but in accordance with \cite{KaLa06} (see also \cite{LeSa20, LNS17}); in \cite{AlCa09, ACC05}, for instance, no precise choice is made. Different choices of the representatives would lead to similar statements. For instance, if $\Omega_{\lambda}$ is chosen closed (w.r.t. the topology of $\Omega$), as implicitly in \cite{GoMa94}, then we would have
$\Omega_{\lambda} = \{ x \in \Omega \mid H_{\Omega}\geq - \lambda\}$ in \eqref{eq_represent_Omegalambda},
$K = \{ x \in \Omega \mid r_{\Omega}(x) \geq 1/h(\Omega)\}$ in \eqref{eq_K_romega},
$A_{\rho} = \overline{\Omega^{\rho}\oplus B_{\rho}(0)}^{\Omega} %
= \bigcup\big\{\overline{B_{\rho}(c)} \cap \Omega \mid B_{\rho}(c)\subseteq \Omega \big\}
$ in \eqref{eq_def_Arho} (notice that here it is important that \eqref{eq_def_inner_parall} is closed to have $A_{\rho}$ closed),
while the proof of \eqref{eq_rOmega_Arho} and \eqref{eq_formula_H_planar} %
simplifies a bit (since we would have $x \in A_{\rho} \iff r_{\Omega}(x)\geq \rho$ and hence $\big\{\lambda>0 \mid x \in \Omega_{\lambda}\big\} = \left[\max\left\{h(\Omega),1/r_{\Omega}(x)\right\},+\infty\right)$), as well as modifications of the proofs of {Corollary \ref{cor_var_mean_curv} (where $f$ is now upper semicontinuous)} and Corollary \ref{cor_minkowski} (exploiting that closed sets are closed under infinite intersections) are needed.
\end{remark}

\subsection{%
Notion of solution for the $1$-Laplacian}
\label{sec_notation}

We refer to \cite{AFP00,Mag12} for the theory of $BV$ functions and
sets of finite perimeter. For $k>0$ we set $T_k(s):=\max(-k,\min(s,k))$.}
{Since the $1$-Laplace operator is $0$-homogeneous in the gradient, problem \eqref{eq_gen_1laplace_a1} has
to be understood in a suitably weak way, through the Anzellotti theory of $L^{\infty}$-divergence-measure
vector fields \cite{Anz83}. We denote by
$$\mathscr{DM}^{\infty}(\Omega):=\big\{ z \in L^{\infty}(\Omega)^N \mid \operatorname{div} z \hbox{ is a
Radon measure with finite total variation in } \Omega\big\},$$
by $(z,Dv)$ the Anzellotti pairing between $z \in \mathscr{DM}^{\infty}(\Omega)$ and $v \in BV(\Omega)\cap
L^{\infty}(\Omega)$, and by $[z,\nu]$ the weak trace on $\partial \Omega$ of the normal component of $z$; we
refer to \cite{Anz83,ChFr99,CrDC19} for the precise definitions and for the
Green formula
$\int_{\Omega} v^{*} \operatorname{div} z + \int_{\Omega}(z,Dv) = \int_{\partial \Omega}[vz,\nu]\,
d\mathcal{H}^{N-1}$.}

{
\begin{definition}[\cite{DGOP19,MAOP25}]
\label{def_solution}
Let $\gamma>0$. A non-negative function $u \in BV_{\rm loc}(\Omega)$ is a \emph{solution} of
\eqref{eq_gen_1laplace_a1} if there exists $z \in \mathscr{DM}^{\infty}(\Omega)$ with
$\norm{z}_{L^{\infty}(\Omega)^N}\leq 1$ such that
\begin{equation}\label{eq_def_sol_eq}
-\operatorname{div} z = u^{-\gamma} \quad \hbox{as measures in $\Omega$},
\end{equation}
\begin{equation}\label{eq_def_sol_pairing}
(z, DT_k(u)) = |DT_k(u)| \quad \hbox{as measures in $\Omega$, for every $k>0$},
\end{equation}
and, for $\mathcal{H}^{N-1}$-a.e.\ $x \in \partial \Omega$, either $[z,\nu](x)=-1$ or $$\lim_{\varepsilon\to 0^+}\fint_{\Omega\cap {B_{\varepsilon}(x)}} u (y) dy = 0.$$
\end{definition}
}

{
\begin{remark}\label{rem_boundary_datum}
Conditions \eqref{eq_def_sol_eq}--\eqref{eq_def_sol_pairing} are the weak formulation of the equation, while
the last requirement is the (by now classical) weak formulation of the homogeneous Dirichlet datum; when
$u \in BV(\Omega)$, so that $u$ has a trace on $\partial \Omega$, it is equivalent to
\begin{equation}\label{eq_weak_boundary}
u\big(1+[z,\nu]\big)=0 \qquad \mathcal{H}^{N-1}\hbox{-a.e.\ on } \partial \Omega .
\end{equation}
In particular the boundary datum need \emph{not} be attained pointwise, and $u$ may have a strictly positive
trace. The model example is $\Omega=B_r$, for which $u \equiv (r/N)^{1/\gamma}$ is a (constant) solution: it
suffices to take $z(x)=-x/r$, so that $|z|\leq1$, $-\operatorname{div}z=N/r=u^{-\gamma}$, $(z,Du)=|Du|=0$ and
$[z,\nu]\equiv -1$ on $\partial B_r$.
\end{remark}
}

{We shall repeatedly use the following existence and approximation result, which is
\cite[Theorems 5.1, 5.3]{DGS19} (see also \cite[Theorem 3.4]{DGOP19}, \cite{LOPdL21,OlPe25} {and \cite{CST16}}).}

\begin{theorem} [\cite{DGS19,DGOP19}]
\label{thm_exist_decicco_a1}
Let $\Omega \subset \R^N$ be open, bounded, with $\partial \Omega \in Lip$, and $\gamma \in (0,1]$. Then there
exists a unique positive solution $u \in BV(\Omega) \cap L^{\infty}(\Omega)$ to \eqref{eq_gen_1laplace_a1}
{in the sense of Definition \ref{def_solution}}.
Such solution is a limit of solutions $u_p \in W^{1,p}_0(\Omega) \cap L^{\infty}(\Omega)$, $p>1$, of
\begin{equation}\label{eq_gen_plaplace_a1}
\begin{cases}
-\Delta_p u_p = u_p^{-\gamma} & \quad \hbox{in $\Omega$}, \\
u_p = 0 &\quad \hbox{on $\partial \Omega$},
\end{cases}
\end{equation}
with $u_p \to u$ a.e., strongly in $L^q(\Omega)$ for any $q \in [1, \frac{N}{N-1})$, and weakly in
$L^{\frac{N}{N-1}}(\Omega)$, as $p\to 1^+$.
\end{theorem}

\subsection{Variational mean curvature and singular 1-Laplace problems}
\label{sec_singular}

As we already mentioned, if $\Omega$ is convex, then the solution $u$ of \eqref{eq_gen_1laplace_a1} in the sense of Definition \ref{def_solution} satisfies
\begin{equation}\label{eq_sol_eq_Hom1}
u(x) = (-H_{\Omega}(x))^{-\frac{1}{\gamma}} \qquad {\hbox{for a.e. $x \in \Omega$}}.
\end{equation}

Identity \eqref{eq_sol_eq_Hom1} is stated in \cite[Example 2]{GOP22} and \cite[Example 3]{MAOP25}; for the
reader's convenience we recall why it holds. In \cite{MoPe15} it is proved that $v:=-H_{\Omega}$ is the large
solution of $\Delta_1 v = v$ in $\Omega$, $v=\infty$ on $\partial \Omega$; in particular there exists
$\zeta \in \mathscr{DM}^{\infty}_{\rm loc}(\Omega)$, $|\zeta|\leq1$, with
$\operatorname{div} \zeta = v$ and $(\zeta,Dv)=|Dv|$, and $[\zeta,\nu]=1$ on $\partial \Omega$.
Let us observe that $\zeta$ is globally, and not merely locally, a divergence-measure field: indeed
$\operatorname{div} \zeta = v = -H_{\Omega}$ and, by \eqref{eq_optimality_H},
$\norm{H_{\Omega}}_{L^1(\Omega)}=P(\Omega)<\infty$, so that $\operatorname{div}\zeta$ is a finite measure in
$\Omega$; since $|\zeta|\leq1$, this gives $\zeta \in \mathscr{DM}^{\infty}(\Omega)$. Consequently the normal
trace $[\zeta,\nu]\in L^{\infty}(\partial \Omega)$ is well defined in the sense of Anzellotti \cite{Anz83}.%

Set $u:=v^{-1/\gamma}$ and $z:=-\zeta$. Since $t \mapsto t^{-1/\gamma}$ is decreasing and locally Lipschitz on
$(0,\infty)$, one has $(z,DT_k(u))=|DT_k(u)|$ for every $k>0$, while
$-\operatorname{div} z = \operatorname{div} \zeta = v = u^{-\gamma}$ and $[z,\nu]=-1$; hence $u$ is a solution
in the sense of Definition \ref{def_solution}.

{Moreover $u\in BV(\Omega)\cap L^\infty(\Omega)$: indeed $u\le h(\Omega)^{-1/\gamma}$ by \eqref{eq_estim_H_dall'alto}
and, for every $s>0$, $\{u>s\}=\bigcup_{\mu<s^{-\gamma}}\Omega_\mu$ is a convex subset of $\Omega$, so that $P(\{u>s\})\le P(\Omega)$ and, by the coarea formula {\cite[Theorem 3.40]{AFP00}}, {$|Du|(\Omega)\le\int_0^{\norm{u}_{\infty}}P(\{u>s\})\, ds\le\norm{u}_{L^\infty(\Omega)}P(\Omega)$, so that $u\in BV(\Omega)$}.
Since the solution is unique {in $BV(\Omega)\cap L^\infty(\Omega)$} {by \cite[Theorem 3.5]{MAOP25}, %
}
\eqref{eq_sol_eq_Hom1} follows.
}

{
\begin{remark}[A quantitative lower bound for the solution]
\label{rem_lower_bound_u}
Combining Proposition \ref{prop_finiteness_H} with \eqref{eq_sol_eq_Hom1}, we obtain the %
completely explicit
lower bound for the solution of \eqref{eq_gen_1laplace_a1} on a convex domain
\begin{equation}\label{eq_lower_bound_u}
u(x)\ { \geq \ \left(\frac{r_{\Omega}(x)}{N}\right)^{\frac{1}{\gamma}}} \geq \ \left(\frac{d_{\Omega}(x)}{N}\right)^{\frac{1}{\gamma}}, \qquad x \in \Omega,
\end{equation}
{where the constant is sharp since equality is attained} %
if $\Omega=B_R(x_0)$ and $x=x_0$ %
by Example
\ref{ex_examp_optimal}$(i)$ and $d_{\Omega}(x_0) %
=R$.
Some comments are in order. First, \eqref{eq_lower_bound_u} is a
\emph{reverse} estimate with respect to what one is used to for $p>1$: {for $p>1$ the solution of
\eqref{eq_gen_plaplace_a1} vanishes on $\partial\Omega$ and %
it behaves near $\partial\Omega$ like $d_{\Omega}$ if $\gamma<1$ and like $d_{\Omega}\,(-\log(d_{\Omega}))^{1/p}$ if $\gamma=1$ (see \cite{CRT77,LaMc91} for $p=2$ and \cite{GST07} for $p\neq2$),} while for $p=1$ the boundary datum is lost and only a positive lower bound
survives. Second, estimate
\eqref{eq_lower_bound_u} together with \eqref{eq_estim_H_dall'alto} yield at once the two-sided bound
\eqref{eq_double_bound_u}, which in particular re-proves that $u \in L^{\infty}(\Omega)$, with an explicit
bound;
the
maximum is attained exactly on {$\overline{K}\cap\Omega$, $K$ being} the Cheeger set of $\Omega$.
\end{remark}
}

\subsection{Power concavity: definitions and known results}
\label{sec_concavity_notions}

We recall the following classical definitions (see e.g. \cite{Ken85}).
{
\begin{definition}[$\theta$-concavity and $\theta$-convexity]
\label{def_theta_conc}
Let $\Omega \subset \mathbb{R}^N$ be a convex set and let $f: \Omega \to (0, +\infty)$. Given
$\theta \neq 0$, we say that $f$ is
\begin{itemize}[leftmargin=*]
\item \emph{$\theta$-concave} if $f(t x + (1-t)y) \ge \big( t f(x)^\theta +
(1-t) f(y)^\theta \big)^{1/\theta}$ for all $x,y \in \Omega$ and $t \in [0,1]$; equivalently, if
$f^\theta$ is concave when $\theta > 0$, and convex when $\theta < 0$;
\item \emph{$\theta$-convex} if $f(t x + (1-t)y) \le \big( t f(x)^\theta +
(1-t) f(y)^\theta \big)^{1/\theta}$ for all $x,y \in \Omega$ and $t \in [0,1]$; equivalently, if
$f^\theta$ is convex when $\theta > 0$, and concave when $\theta < 0$.
\end{itemize}
The cases $\theta= 1$ correspond to (concavity and) convexity, and the cases $\theta=-1$ are also called
\emph{harmonic} (concavity and) convexity: %
thus $f$ is harmonic concave if $1/f$ is convex, and
harmonic convex if $1/f$ is concave.
{For $\theta=0$ we say that $f$ is $0$-concave (or \emph{log-concave}) if $f(t x+(1-t)y)\geq f(x)^{t}f(y)^{1-t}$, i.e.\ if $\log(f)$ is concave.} {Finally, we say that $f$ is $(+\infty)$-concave if $f$ is constant, while $f$ is $(-\infty)$-concave if it is quasi-concave (i.e. %
all the superlevel sets are convex).}
\end{definition}
}

\begin{remark}\label{rem_-f_impl_concav}
{It can be seen \cite{Ken85} that $\theta$-concavity gets stronger as $\theta$ \emph{increases}, while $\theta$-convexity gets stronger as $\theta$ \emph{decreases}. In particular, if
$f>0$ and $\theta<-1$, then
$$ \hbox{$f$ is $\theta$-convex} \implies \hbox{$f$ is $(-1)$-convex} \implies \hbox{$f$ is convex}.$$
The last implication, which is the one relevant for us, is immediate: if $1/f$ is concave then, by the
inequality between the weighted harmonic and arithmetic means,
$$f(t x + (1-t)y) \le \Big(\tfrac{t}{f(x)}+\tfrac{1-t}{f(y)}\Big)^{-1} \le
t f(x)+(1-t)f(y).$$
Applying this to $f=-H_{\Omega}$: the harmonic convexity of $-H_{\Omega}$ given by Corollary
\ref{cor_var_mean_curv} implies the convexity of $-H_{\Omega}$, that is the concavity of $H_{\Omega}$ obtained
in \cite[Proposition 2.2]{AlCa09}. Equivalently, if $g<0$ then $g^{-1}$ convex implies $g$ concave.}
\end{remark}

We shall use the following stability result about pointwise limits of convex functions, see e.g.\
\cite[Theorem 10.8]{Roc97}.

\begin{theorem}[Pointwise stability of convexity]
\label{thm_rock}
Let $\Omega \subset \R^N$ be an open, convex set. Let $(v_n)_n$ be a sequence of (finite) convex functions on
$\Omega$. Assume that $v_n \to v$ pointwise on a dense subset of $\Omega$, with $v$ finite. Then {$v$ can be uniquely extended to $\Omega$ in such a way}
$v_n \to v$ in $L^{\infty}_{\rm loc}(\Omega)$ and $v$ is convex.
\end{theorem}

Finally, we recall the following concavity result for the $p$-Laplacian, $p>1$, which is
\cite[Theorem 6.7 and Remark 6.8]{GaSq25} (see also \cite[{Theorem} 2.3]{Gla02}).

\begin{theorem}[\cite{GaSq25}]
\label{thm_GS_autonom}
Let $\Omega \subset \R^N$, $N\geq 2$, be open, bounded, convex and with $\partial \Omega \in C^{2,\alpha}$, and
let $\gamma \in (0, 1]$ and $p>1$. Then the solution $u\in W^{1,p}_0(\Omega) \cap C^{1}(\Omega)$ of
\eqref{eq_gen_plaplace_a1} is such that $u^{\frac{p-1+\gamma}{p}}$ is concave.
\end{theorem}
Such a result will be improved below, in Theorem \ref{thm_GS_autonom_imp}.

{
\begin{remark}\label{rem_which_transformation}
Let us recall where the exponent $\frac{p-1+\gamma}{p}$ comes from, since this will be used in the proofs. The
general machinery of \cite[Section 5.1]{GaSq25} associates to a nonlinearity $g$ {of $-\Delta_p u = g(u)$} the transformation
\begin{equation}\label{eq_def_varphi}
G(t):=\int_0^t g(s)\,ds, \qquad \varphi(t):=\int_1^t \big(G(s)\big)^{-\frac1p}\,ds,
\end{equation}
(for future reference, we also set $\psi:=\varphi^{-1}$, which will be used in Section
\ref{sec_nonautonomous}),
and \cite[Corollary 5.6]{GaSq25} gives, under suitable structural {and regularity} assumptions, the concavity of
$\varphi(u)$.

In our case, $g_0(t)=t^{-\gamma}$. When $\gamma \in (0,1)$ the above expressions make sense and we have $G_0(t) = \hbox{const} \cdot t^{-\gamma+1}$ and $\varphi_0(t) = \hbox{const} \cdot t^{\frac{\gamma-1+p}{p}} + \hbox{const}$, which is equivalent to our transformation from a concavity point of view. Due to regularity reasons, one needs to deal first with a regularized problem, given by $g_{\eta}(t):=(t+\eta)^{-\gamma}$, $\eta>0$, so that $G_{\eta}(t) \to G_0(t)$ and $\varphi_{\eta}(t) \to \varphi_0(t)$ locally uniformly on $(0,\infty)$. Thus the concavity of $\varphi_{\eta}(u_{\eta})$ passes to the limit as
$\eta \to 0^+$ by Theorem \ref{thm_rock}, provided $u_{\eta}\to u$ a.e.\ with $u>0$.

{When $\gamma=1$ the expressions in \eqref{eq_def_varphi} are not well defined: here again we approximate the problem through $g_{\eta}$ and suitable corrections, observing that $G_{\eta}(t) + \log(\eta) \to \log(t)$ (the primitive of $g(t)=t^{-1}$) and $(-\log(\eta))^{1/p} \varphi_{\eta}(t) \to t + \hbox{const}$ locally uniformly on $(0,+\infty)$, which again allows to transfer concavity properties from $u_{\eta}$ to $u$.}

{Let us observe also that, {denoting by $u_{p,\Omega}$ the solution of \eqref{eq_gen_plaplace_a1} in $\Omega$,} $\theta=\frac{p-1+\gamma}{p}$ appears as the exponent which makes the map $s \mapsto (u_{p,s\Omega})^{\theta}(s\cdot)$ homogeneous of degree one (thus {natural} from the concavity point of view), being %
$(u_{p,s\Omega})^{\frac{p-1+\gamma}{p}}(sx) = s (u_{p,\Omega})^{\frac{p-1+\gamma}{p}}(x)$.
}
\end{remark}}

\section{An approximation result: concavity for $p>1$}
\label{sec_approx}

{In this section we prove an approximation result which allows to remove all regularity assumptions,
both on the domain and on the datum, from concavity statements for $p$-Laplace problems. The argument is
inspired by \cite[Section 5]{Sak87}, \cite[Section 4.1]{BMS24} and \cite[proof of Theorem 5.2]{GaSq25}. It will be used
for the autonomous singular problem \eqref{eq_gen_plaplace_a1} (Theorem \ref{thm_GS_autonom_imp}), which is
what we need for the main results, as well as for the weighted torsion problem (Theorem \ref{thm_GS_torsion_imp}) and the weighted singular problem (Theorem \ref{thm_conc_plapl_tors}).}

{
We recall that, given $\Omega \subset \R^N$ {open}, bounded and convex, there always exist \cite[Proposition 2.1]{GaSq25}
\begin{equation}\label{eq_OmegaktoOmega}
\hbox{$\Omega^{k}\subset \Omega^{k+1} \subset \Omega$ smooth, strongly convex sets converging to $\Omega$ in
the Hausdorff distance}.
\end{equation}

\begin{lemma} %
\label{lem_conv_minim}
Let $\Omega \subset \R^N$ be open, bounded and convex, let $p>1$, and let
$f:\Omega \times \R\to \R$ be a Carath\'eodory function, and set
$F(x,t):=\int_0^t f(x,s)\,ds$. Let
$$J(v):= \frac1p \int_{\Omega}|\nabla v|^p - \int_{\Omega} F(x,v), \qquad v \in W^{1,p}_0(\Omega).$$
Let $\Omega^k$ be as in \eqref{eq_OmegaktoOmega},
and let
$f_k:\Omega^k\times\R\to\R$ be Carath\'eodory functions such that
\begin{equation}\label{eq_hyp_f_approx}
f_k(x,\cdot) \hbox{ is non-increasing for a.e.\ $x \in \Omega^k$}, \qquad
\sup_k\|f_k\|_{L^\infty(\Omega^k\times\R)}<+\infty.
\end{equation}
Assume moreover that
\begin{equation}\label{eq_conv_f_approx}
\left\|\sup_{t\in\R}|f_k(\cdot,t)-f(\cdot,t)|\right\|_{L^{p'}(\Omega^k)}
\longrightarrow0.
\end{equation}
Set $F_k(x,t):=\int_0^t f_k(x,s)\,ds$ and let
$$J_k(v):=\frac1p\int_{\Omega^k}|\nabla v|^p-\int_{\Omega^k}F_k(x,v),
\qquad v\in W^{1,p}_0(\Omega^k).$$
Then $J$ admits a unique minimum $u$ on $W^{1,p}_0(\Omega)$, and $J_k$ admits a unique minimum $u_k$
on $W^{1,p}_0(\Omega^k)$ for every $k$. Extend $u_k$ by $0$ on $\Omega\setminus\Omega^k$.
Then $u_k \to %
 u$ in $W^{1,p}_0(\Omega)$.
\end{lemma}

\begin{proof}
Up to extending
$$d_k(x):=\sup_{t\in\R}|f_k(x,t)-f(x,t)|$$
by $0$ on $\Omega\setminus\Omega^k$, by \eqref{eq_conv_f_approx} we have $d_k\to0$ in
$L^{p'}(\Omega)$, and we can assume $d_k\to0$ a.e.\ in $\Omega$; %
by \eqref{eq_hyp_f_approx} we deduce that
$f(x,\cdot)$ is non-increasing for a.e.\ $x\in\Omega$ and for some $C>0$
\begin{equation}\label{eq_bounds_ffkFFk}
|f(x,t)|\leq C, \qquad |f_k(x,t)|\leq C, \qquad |F(x,t)|\leq C|t|, \qquad |F_k(x,t)|\leq C|t|
\end{equation}
for a.e.\ $x$ and every $t\in\R$, and
\begin{equation}\label{eq_bounds_F-Fk}
|F_k(x,t)-F(x,t)|\leq d_k(x)|t|, \qquad |F(x,t)-F(x,s)|\leq C|t-s|
\end{equation}
for a.e.\ $x\in\Omega^k$ and every $t,s \in\R$.

{By \eqref{eq_bounds_ffkFFk} together with H\"older's and Poincar\'e's inequalities, $J$ and $J_k$ are
coercive; being also weakly lower semicontinuous, %
{we obtain} %
minima; these are unique, since $v \mapsto \frac1p\int|\nabla v|^p$ is strictly convex while
$-F(x,\cdot)$ and $-F_k(x,\cdot)$ are convex, $f$ and $f_k$ being non-increasing in $t$.} Denote by $u$ and
$u_k$ the minima of $J$ and $J_k$ respectively. Moreover $u$ is a weak solution of $-\Delta_p u=f(x,u)$ in
$\Omega$ with $u=0$ on $\partial \Omega$, with $f$ bounded by \eqref{eq_bounds_ffkFFk}; hence
$u \in L^{\infty}(\Omega)$ by \cite[Theorem 2.1, Section 5.2]{LaUr68}, {$u\in C(\Omega)$ by standard regularity theory \cite{Tru67} and, $\Omega$ being convex (hence satisfying an exterior cone condition at every boundary point), $u\in C(\overline{\Omega})$ by the Wiener-type criterion of \cite{GaZi77}}.

\emph{Step 1: uniform bounds.} Since $\Omega^{1}\subseteq \Omega^{k}$ for every $k$, the function $u_1$ is
admissible for $J_k$, hence $J_k(u_k)\leq J_k(u_1)$. By \eqref{eq_bounds_ffkFFk} we obtain
$$\frac1p\int_\Omega|\nabla u_k|^p-C\int_\Omega|u_k| \leq J_k(u_k) \leq J_k(u_1)\leq \frac1p\int_\Omega|\nabla u_1|^p+C\int_\Omega|u_1|\leq C,$$
which, combined with H\"older's and Poincar\'e's inequalities applied to $\norm{u_k}_1$, implies that
 $(u_k)_k$ is bounded in $W^{1,p}_0(\Omega)$. Thus up to a %
subsequence,
$u_k \wto \bar u$ in $W^{1,p}_0(\Omega)$, $u_k \to \bar u$ in $L^p(\Omega)$ and a.e.\ in $\Omega$.

\emph{Step 2: $\bar u$ is a minimum.} Let
$$\eps_k:=2\sup_{\overline{\Omega}\setminus\Omega^{k}}|u|$$
and define
$$v_k:=(|u|-\eps_k)_+\sgn(u);$$
clearly $v_k\in W^{1,p}_0(\Omega^k)$ with $\nabla v_k=\nabla u\,\chi_{\{|u|>\eps_k\}}$.
Since $u\in C(\overline{\Omega})\cap W^{1,p}_0(\Omega)$, we have $u=0$ on $\partial\Omega$. Moreover,
being $\Omega^k\nearrow\Omega$ %
in the Hausdorff distance, one has
$\eps_k\to0$.
Thus
$$\|\nabla v_k-\nabla u\|_{L^p(\Omega)}^p
=\int_{\{0<|u|\leq\eps_k\}}|\nabla u|^p\longrightarrow0$$
by dominated convergence, since $\chi_{\{0<|u|\leq\eps_k\}}\to0$ a.e.\ in $\Omega$ and
$\nabla u=0$ a.e.\ on $\{u=0\}$. %
Hence
$v_k\to u$ in $W^{1,p}_0(\Omega)$.
Therefore, by \eqref{eq_bounds_F-Fk}
\begin{align*}
\left|\int_{\Omega^k}F_k(x,v_k)-\int_\Omega F(x,u)\right|
\leq \|d_k\|_{L^{p'}(\Omega)}\|v_k\|_{L^p(\Omega)}
+C\|v_k-u\|_{L^1(\Omega)}
\longrightarrow0.
\end{align*}
Similarly, since $u_k\to\bar u$ in $L^p(\Omega)$,
\begin{align*}
\left|\int_{\Omega^k}F_k(x,u_k)-\int_\Omega F(x,\bar u)\right|
&\leq \|d_k\|_{L^{p'}(\Omega)}\|u_k\|_{L^p(\Omega)}
+C\|u_k-\bar u\|_{L^1(\Omega)}
\longrightarrow0.
\end{align*}
Hence, by lower semicontinuity and by the minimality of $u_k$ in $\Omega^k$,
$$J(\bar{u}) \leq \liminf_{k} J_k(u_k) \leq \limsup_k J_k(u_k)
\leq \limsup_k J_k(v_k) = J(u),$$
which implies, by the uniqueness of the minimum, that $\bar{u}=u$.

\emph{Step 3: strong convergence.} From the previous step
$$J_k(u_k)\longrightarrow J(u) \quad \hbox{ and } \quad \int_{\Omega^k}F_k(x,u_k)\longrightarrow\int_\Omega F(x,u),$$
thus $\norm{\nabla u_k}_p \to \norm{\nabla u}_p$, which implies the strong convergence. Showing that the whole sequence $u_k$ converges to $u$ is standard.
\end{proof}

\begin{proposition}[Stability of concavity under approximation]
\label{prop_approx_general}
In the assumptions of Lemma \ref{lem_conv_minim}, let moreover
$\varphi_k,\varphi:I\to\R$ be continuous and increasing, with $u_k(\Omega^k), u(\Omega) \subset I \subset \R$ and
$\varphi_k\to\varphi$ locally uniformly in $I$. Assume that
$\varphi_k(u_k)$ is concave in $\Omega^k$ for every $k$. Then
$\varphi(u)$ is concave in $\Omega$.
\end{proposition}

\begin{proof}
Fix $\Omega' \Subset \Omega$ open and convex. Since
$\Omega^{k}\nearrow \Omega$ in the Hausdorff distance and $\Omega'$ is compactly contained in $\Omega$, we have
$\Omega' \subset \Omega^{k}$ for $k$ large. On $\Omega'$ the functions $\varphi_k(u_k)$ are concave and, by
Lemma \ref{lem_conv_minim} and the local uniform convergence of $\varphi_k$ to $\varphi$, they converge
a.e.\ to $\varphi(u)$, hence $\varphi(u)$ is concave on $\Omega'$ by Theorem \ref{thm_rock}, applied to
$-\varphi_k(u_k)$. Taking a countable exhaustion $\Omega'_j \nearrow \Omega$, we conclude that $\varphi(u)$ is
concave in $\Omega$. {Since $u$ and $\varphi$ are continuous in $\Omega$, $\varphi(u)$ coincides everywhere with the concave function given by Theorem \ref{thm_rock}.}
\end{proof}
}

{We can now remove the regularity assumption on $\partial \Omega$ in Theorem \ref{thm_GS_autonom}. This is
the technical core of the paper.}

\begin{theorem}
\label{thm_GS_autonom_imp}
Let $\Omega \subset \R^N$, $N\geq 2$, be open, bounded and convex, and let $\gamma \in (0, 1]$ and $p>1$. Then
the solution $u\in W^{1,p}_0(\Omega)\cap C(\Omega)$ of \eqref{eq_gen_plaplace_a1} is such that
$u^{\frac{p-1+\gamma}{p}}$ is concave.
\end{theorem}

\begin{proof}
We start by observing that a solution in $W^{1,p}_0(\Omega)$ exists \cite[Theorem 7.1]{OlPe25}, it is unique
\cite[Corollary 2.3]{DuOl22} (see also \cite{CST16}), and moreover this solution is in $L^{\infty}(\Omega)$
\cite[Theorems 3.1, 4.1, 4.4]{DeC13}; {moreover $u\in C(\Omega)$ by standard local regularity, since $u$ is locally bounded away from zero \cite{DeC13}}. %
Such solutions can be obtained as limit, as $\eta\to0^+$, of the solutions
$u_{\eta}$ of
\begin{equation}\label{eq_gen_smooth_plaplace}
\begin{cases}
-\Delta_p u_\eta = (u_\eta+\eta)^{-\gamma} & \quad \hbox{in $\Omega$}, \\
u_\eta = 0 &\quad \hbox{on $\partial \Omega$},
\end{cases}
\end{equation}
which are unique as well \cite{DeC13,DuOl22}.

{Fix $\eta>0$ and set $g_{\eta}(t):=(t^{+}+\eta)^{-\gamma}$ (here the
truncation $t^{+}$ is used to make functionals well defined in $W^{1,p}_0(\Omega)$, not only in the positive
cone), which is bounded and non-increasing;
let $\varphi_{\eta}$ be as in Remark \ref{rem_which_transformation}. On a smooth, strongly convex domain,
\cite[Corollary 5.6]{GaSq25} (applied as in \cite[Theorem 6.7 and Remark 6.8]{GaSq25}) gives that
$\varphi_{\eta}(u_{\eta})$ is concave.
The datum $g_{\eta}$ satisfies \eqref{eq_hyp_f_approx}, and it does not depend on
$x$, so that we may take $f_k \equiv g_{\eta}$ in Proposition \ref{prop_approx_general}, after having fixed
$\Omega^k$ as in \eqref{eq_OmegaktoOmega}. Proposition
\ref{prop_approx_general}, applied with $\varphi_k \equiv \varphi_{\eta}$, then gives that $\varphi_{\eta}(u_{\eta})$
is concave on \emph{every} bounded convex $\Omega$.}

Finally we let $\eta \to 0^+$. We have $u_{\eta}\to u$ a.e.\ in $\Omega$ and $u>0$ in $\Omega$; being
$\varphi_{\eta}(u_{\eta})$ concave, we can argue as in Remark \ref{rem_which_transformation} and conclude.
\end{proof}

\section{The limit $p \to 1^{+}$: proof of the main results}
\label{sec_limit}

{We now prove the main results.}

\begin{proof}[Proof of Theorem \ref{thm_main_ucgamma}]
{By Theorem \ref{thm_exist_decicco_a1} the solutions $u_p$ of \eqref{eq_gen_plaplace_a1} converge to $u$
a.e.\ in $\Omega$ as $p\to1^{+}$, and $u>0$ a.e.\ in $\Omega$. By Theorem \ref{thm_GS_autonom_imp} each
$u_p^{\frac{p-1+\gamma}{p}}$ is concave. Since $\frac{p-1+\gamma}{p}\to \gamma$ as $p\to1^{+}$ and
$u_p \to u$ a.e.\ with $u$ positive and finite, we get $u_p^{\frac{p-1+\gamma}{p}} \to u^{\gamma}$ a.e.\ in
$\Omega$; hence $u^{\gamma}$ is concave by Theorem \ref{thm_rock}, applied to the convex functions
$-u_p^{\frac{p-1+\gamma}{p}}$.}

{Being concave and positive, $u^{\gamma}$ is locally Lipschitz in $\Omega$. Moreover, since
$\gamma \in (0,1]$, the map $s \mapsto s^{1/\gamma}$ is locally Lipschitz
so that} $u \in Lip_{loc}(\Omega)$.
\end{proof}

\begin{proof}[Proof of Corollary \ref{cor_var_mean_curv}]
{{By \eqref{eq_sol_eq_Hom1} we have $u^{\gamma}=(-H_{\Omega})^{-1}$ a.e., where by Theorem
\ref{thm_main_ucgamma} we can choose $g:=u^{\gamma}$ as its continuous representative.
Let us show that this implies that $(-H_{\Omega})^{-1}$ itself, by its representation \eqref{eq_def_Homega}, is continuous.}
{Indeed, for every $\lambda>0$ the set $\{-H_\Omega<\lambda\}=\bigcup_{\mu<\lambda}\Omega_\mu$ is open, so that $f:=(-H_\Omega)^{-1}$ is lower semicontinuous, and hence $f\le g$ in $\Omega$, $g$ being continuous. If $f(x_0)<g(x_0)$ for some $x_0$, choose $\lambda$ with $f(x_0)<\frac1\lambda<g(x_0)$ and a ball $U\ni x_0$ on which $g>\frac1\lambda$:
then a.e.\ point of $U$ belongs to $\{-H_{\Omega}<\lambda\} \subseteq %
\Omega_\lambda$ and, $\Omega_\lambda$ being open and convex, $U\subseteq\Omega_\lambda$,

so that $-H_\Omega(x_0)\le\lambda$, a contradiction. Accordingly, the solution $u$ of \eqref{eq_gen_1laplace_a1} will always be identified with its representative $(-H_\Omega)^{-1/\gamma}$.}

Hence, Theorem
\ref{thm_main_ucgamma} gives that $(-H_{\Omega})^{-1}$ is concave, that is $-H_{\Omega}$ is harmonic convex
(equivalently, $H_{\Omega}^{-1}=-(-H_{\Omega})^{-1}$ is convex). In particular
$H_{\Omega}^{-1}\in Lip_{loc}(\Omega)$, and $H_{\Omega}$ is concave by Remark \ref{rem_-f_impl_concav}, so
that $H_{\Omega}\in Lip_{loc}(\Omega)$ as well.}

\smallskip
We prove the optimality, taking as competitor the square. Let then $\Omega=(0,\ell)^2$ as in Example \ref{ex_examp_optimal}$(iii)$ of Appendix
\ref{sec_examples}.
We focus on the behaviour of $H_{\Omega}$ along the diagonal emerging from
a vertex, which we place at the origin.
By \eqref{eq_H_square} we have
$$-H_{\Omega}(\delta,\delta)=\frac{1}{(2+\sqrt2)\,\delta} \qquad \hbox{for } 0<\delta<\frac{\ell}{(2+\sqrt{\pi})(2+\sqrt{2})}$$
which implies that $(-H_{\Omega})^{-1}$ is affine along the diagonal (thus $(-1)$-convexity inequality is attained with an equality), and for any $\theta<-1$ we have that $(-H_{\Omega})^{\theta}$ is not concave along the diagonal, thus $\theta=-1$ is optimal.

{The convex hull of Example \ref{ex_examp_optimal}$(ii)$ provides another instance of optimality, with equality along the axis, see \eqref{eq_example_axis}.}
{For dimensions $N\geq 3$, see Remark \ref{rem_opt_anyN}.}
\end{proof}

{\begin{remark}[Optimality in any dimension]\label{rem_opt_anyN}
The optimality of the exponent $-1$ is not a planar phenomenon, and an analogous example can be constructed in every dimension $N\geq3$. Arguing as in \cite[Proposition 2.7]{AlCa09} {(since $\Omega_{\lambda}$ itself is a minimizer of $P-\lambda |\cdot|$)}, each $\Omega_\lambda$ is of class $C^{1,1}$ with principal curvatures bounded by $\lambda$, so that, by Blaschke's rolling theorem for $C^{1,1}$ convex bodies \cite[Corollary 1.13]{Ban99} (see also \cite[Lemma 4.3]{AlCa09}), $\Omega_\lambda$ is a union of balls of radius $1/\lambda$ contained in $\Omega$. This gives $-H_\Omega\geq\max\{h(\Omega),1/r_\Omega\}$ in every dimension %
(in the plane this is an equality, by Lemma \ref{lem_explicit_Homega}). For the cube $\Omega=(0,\ell)^N$ one can check that $r_\Omega(s,\dots,s)\le\frac{\sqrt N}{\sqrt N-1}\,s$ for $s$ small, so that $(-H_\Omega)^{-1}$ vanishes at least linearly along the diagonal approaching a vertex; as in the planar case, this is incompatible with the concavity of $(-H_\Omega)^{\theta}=\big((-H_\Omega)^{-1}\big)^{|\theta|}$ for $\theta<-1$, since a positive concave function on $(0,s_0]$ cannot be $o(s)$ as $s\to0^+$.
\end{remark}}

\begin{remark}[An elementary proof in the plane]
\label{rem_elementary_planar}
{
When $N=2$, Corollary \ref{cor_var_mean_curv} admits a direct geometric proof, which we record both as a
consistency check and because it makes the optimality %
transparent.

By \eqref{eq_formula_H_planar} we have
$$-\frac{1}{H_{\Omega}}=\min\Big\{\frac{1}{h(\Omega)},\, r_{\Omega}\Big\},$$
so it suffices to observe that the function $r_{\Omega}$ defined in \eqref{eq_def_rOmega} is \emph{concave}. This is
immediate from the convexity of $\Omega$: if $B_{\rho_i}(c_i)\subseteq \Omega$ with
$x_i \in \overline{B_{\rho_i}(c_i)}$, $i=0,1$, then setting $c_t=(1-t)c_0+tc_1$, $x_t=(1-t)x_0+tx_1$ and $\rho_t=(1-t)\rho_0+t\rho_1$
one has $B_{\rho_t}(c_t)=(1-t)B_{\rho_0}(c_0)+tB_{\rho_1}(c_1)\subseteq \Omega$, again by convexity, and
$|x_t-c_t|\leq (1-t)\rho_0+t\rho_1=\rho_t$, so that
$x_t \in \overline{B_{\rho_t}}(c_t)$. Hence $r_{\Omega}(x_t) \geq (1-t) \rho_0+t \rho_1$ and
we can pass to the supremum over $\rho_0$ and $\rho_1$, obtaining
$r_{\Omega}(x_t)\geq (1-t)r_{\Omega}(x_0)+t
r_{\Omega}(x_1)$. Being the minimum of two concave functions, $(-H_{\Omega})^{-1}$ is concave.

Moreover,
harmonic convexity inequality holds in a sharp way along segments outside the Cheeger set where $r_{\Omega}$ is affine, a case which occurs e.g. along
$\overline{OP}$ in Example \ref{ex_examp_optimal}$(ii)$ and along the diagonal in Example
\ref{ex_examp_optimal}$(iii)$.
}

Two comments are in order. First, this argument is elementary only modulo Lemma \ref{lem_planar_structure},
that is, modulo the structure theorem of \cite{LeSa20} {(see also \cite{ACC05b})}, which is far from trivial. Second, and more
importantly, it is genuinely two-dimensional, as already commented in Remark \ref{rem_planar_only}.
{For the relation with the geometric approach of \cite[Section 4]{AlCa09}, which works in every dimension, see Remark \ref{rem_AlCa_related}; the proof through Theorem \ref{thm_main_ucgamma} is of a different, analytic nature.}
\end{remark}

{
\begin{proof}[Proof of Corollary \ref{cor_ext_gammamag1}]
The proof follows by Corollary \ref{cor_var_mean_curv} and \eqref{eq_sol_eq_Hom1}, together (when $\gamma > 1$) with the lower bound \eqref{eq_lower_bound_u}.
\end{proof}
}

{Corollary \ref{cor_minkowski} is now essentially a consequence of Corollary \ref{cor_var_mean_curv} and representation \eqref{eq_represent_Omegalambda}:} we give a direct proof for the reader's convenience.

 \begin{proof}[Proof of Corollary \ref{cor_minkowski}]
{If $t\in\{0,1\}$ the claim is trivial, and if $\lambda_1=\lambda_2=\lambda$ it reduces to the convexity of $\Omega_\lambda$; we may thus assume $t\in(0,1)$ and $\lambda_1\neq\lambda_2$.}
Let $z \in t \Omega_{\lambda_1} \oplus (1-t)\Omega_{\lambda_2}$, $t \in (0,1)$, $\lambda_1\neq \lambda_2$; thus there exists $x\in \Omega_{\lambda_1}$,
$y \in \Omega_{\lambda_2}$ such that $z=tx+(1-t)y$.
Hence $\lambda_i \geq h(\Omega)$ and
$$\lambda_1 \geq \inf\{ \lambda>0 \mid x \in \Omega_{\lambda}\}, \qquad \lambda_2 \geq \inf\{ \lambda>0 \mid y \in \Omega_{\lambda}\} ;$$
since $\lambda_1\neq \lambda_2$, %
we can assume (w.l.o.g.) $\lambda_1 > h(\Omega)$, {thus by \eqref{eq_represent_Omegalambda} and $x \in \Omega_{\lambda_1}$ we have that the first inequality must be strict}.
Hence by Corollary \ref{cor_var_mean_curv}
\begin{align*}
\frac{t}{\lambda_1} + \frac{1-t}{\lambda_2} <
\frac{t}{- H_{\Omega}(x)} +\frac{1-t}{- H_{\Omega}(y)} \leq \frac{1}{- H_{\Omega}(tx+(1-t)y)} ;
\end{align*}
thus
$$ \frac{1}{\frac{t}{\lambda_1} + \frac{1-t}{\lambda_2} } > %
\inf\big\{ \lambda>0 \mid tx+(1-t)y \in \Omega_{\lambda}\big\} .$$
By the monotonicity property of $\Omega_{\lambda}$
we obtain
$tx+(1-t)y \in \Omega_{\frac{1}{\frac{t}{\lambda_1}+\frac{1-t}{\lambda_2}}}$,
which is the claim. The last inclusion follows by the inequality among the weighted harmonic mean and the
weighted arithmetic mean.
\end{proof}

{%
{We observe now} that Examples \ref{ex_examp_optimal}$(iii)$--$(v)$ give a negative answer to the question {whether} there exists a uniform $\alpha \in (0,1]$ such that %
$(H_{\Omega})_{|\Omega}$
can be extended in a $C^{0,\alpha}$ way up to the boundary.}

\begin{proposition}[No uniform H\"older boundary regularity]
\label{prop_no_lip}
There exists a bounded convex domain $\Omega \subset \R^2$ with $\partial \Omega \in C^{\infty}$ such that $(H_{\Omega}^{-1})_{|\Omega}$ %
admits no Lipschitz extension to $\overline{\Omega}$. Moreover, for every
$\alpha \in {(0,1]}$ there exists a bounded convex domain $\Omega \subset \R^2$ with
$\partial \Omega \in C^{1,1}$ such that $(H_{\Omega}^{-1})_{|\Omega}$ {admits no $C^{0,\alpha}$ extension to {$\overline{\Omega}$}}. %
{The same statements hold for $(H_{\Omega})|_{\Omega}$ as well.}
\end{proposition}

\begin{proof}
For the first assertion take the elongated ellipse of
Example \ref{ex_examp_optimal}$(iv)$ {(the square of Example \ref{ex_examp_optimal}$(iii)$ behaves in the same way, but it is not smooth)}: in both cases
$(-H_{\Omega})^{-1}$ behaves like $\hbox{const}+\hbox{const}\,\sqrt{\delta}$, $\delta$ being the distance from
the relevant boundary point, so that its derivative blows up. For the second one, given $\alpha \in {(0,1]}$
choose $m \in (0,2)$ with $\beta_m=\frac{m}{m+2}<\alpha$, and take $\Omega=\Omega_m$ as in Example
\ref{ex_examp_optimal}$(v)$: by \eqref{eq_H_perturbed} we have
$(-H_{\Omega_m}(t,0))^{-1}-\frac{b^2}{a} \sim C (a-t)^{\beta_m}$ as $t \to a^-$, which is not
$O\big((a-t)^{\alpha}\big)$.

{Finally, since {the domains of Examples \ref{ex_examp_optimal}$(iv)$--$(v)$ are} $C^{1,1}$ and hence $H_{\Omega}$ is bounded and far from zero (see \eqref{eq_estim_H_dall'alto} and Remark \ref{rem_ifflinf}), the {lack of} regularity holds for $(H_{\Omega})|_{\Omega}$ as well.}
\end{proof}

\noindent It is worth stressing that these examples say something about the \emph{geometry}, and not only
about the regularity theory: by Example \ref{ex_examp_optimal}$(iii)$ the Lipschitz character is lost near a
flat portion of the boundary, whereas by Example \ref{ex_examp_optimal}$(iv)$ it is lost at the point of
maximal curvature. In both cases the local structure of $\partial \Omega$ at the relevant point is as good as
one may wish, and the loss of regularity is caused by the global geometry of $\Omega$ through the function
$r_{\Omega}$ of \eqref{eq_def_rOmega}.

\smallskip

{
{We notice that} a positive answer can be given, with regard to $(H_{\Omega}^{-1})_{|\Omega}$, for convex polytopes. This applies, for instance, to Example \ref{ex_examp_optimal}$(iii)$; notice that in polytopes $(H_{\Omega})_{|\Omega}$ is naturally unbounded, since the domain is not $C^{1,1}$.
\begin{proposition}
\label{prop_polytope}
Let $\Omega \subset \R^N$, $N\geq 2$, be (the interior of) a convex polytope. Then $(H_{\Omega}^{-1})_{|\Omega} \in Lip_{loc}(\Omega)$ and it admits a continuous extension over $\overline{\Omega}$.
\end{proposition}

\begin{proof}
By Corollary \ref{cor_var_mean_curv}, $H_{\Omega}^{-1}$ is bounded and convex. Thus by \cite{GKR68} (see also \cite[Theorems 10.2 and 10.3]{Roc97}) we have the claim.
\end{proof}
}

{
\begin{remark}[Equality in Minkowski sum]
\label{rem_equal_mink}
{The sharpness of the harmonic convexity of $-H_{\Omega}$ can also be seen through the Minkowski sum relation in Corollary \ref{cor_minkowski}.}
Indeed, consider a rectangle $(0,a)\times(0,b)$. Observing that $\Omega^{\rho}=[\rho,a-\rho]\times [\rho,b-\rho]$ is a rectangle as well, and exploiting \eqref{eq_Omegalambda_planar} and properties of Minkowski sum, we obtain
\begin{equation*}
\begin{aligned}
t \Omega_{\lambda_1} \oplus (1-t)\Omega_{\lambda_2}  &= t \Omega^{1/\lambda_1}\oplus tB_{1/\lambda_1} \oplus (1-t)\Omega^{1/\lambda_2}\oplus (1-t)B_{1/\lambda_2} \\
&= t \Omega^{1/\lambda_1} \oplus (1-t)\Omega^{1/\lambda_2} \oplus B_{t/\lambda_1} \oplus B_{(1-t)/\lambda_2} \\
&=  \Omega^{t/\lambda_1+ (1-t)/\lambda_2} \oplus B_{t/\lambda_1+(1-t)/\lambda_2}= \Omega_{(t/\lambda_1+ (1-t)/\lambda_2)^{-1}}
\end{aligned}
\end{equation*}
which means that equality $t \Omega_{\lambda_1} \oplus (1-t)\Omega_{\lambda_2} = \Omega_{\frac{1}{\frac{t}{\lambda_1}+\frac{1-t}{\lambda_2}}}$ can be achieved in nontrivial cases $\Omega_{\lambda_i}\neq \Omega$.

Consider instead{, for $t\in(0,1)$ and $\lambda_1,\lambda_2\geq h(\Omega)$ with $\lambda_1\neq\lambda_2$,} the equality in the ``concavity'' case, i.e. $t \Omega_{\lambda_1} \oplus (1-t)\Omega_{\lambda_2}= \Omega_{t\lambda_1+(1-t)\lambda_2}$; we want to show that this may happen only in the trivial case $\Omega_{\lambda_1}=\Omega_{\lambda_2}=\Omega$. First, we observe the following strict monotonicity property
\begin{equation}\label{eq_strict_monot}
h(\Omega)<\lambda<\mu, \; \Omega_{\mu} \neq \Omega \implies \Omega_{\lambda} \subsetneq \Omega_{\mu}.
\end{equation}
Indeed, let $x \in \Omega_{\lambda}$ and $y \in \Omega \setminus \Omega_{\mu}$, then by \eqref{eq_represent_Omegalambda} $H_{\Omega}(y)\leq -\mu <-\lambda < H_{\Omega}(x)$, thus by continuity of $H_{\Omega}$ there exists $z \in [x,y]$ such that $-\mu < H_{\Omega}(z)< - \lambda$, i.e. $z \in \Omega_{\mu} \setminus \Omega_{\lambda}$, which implies \eqref{eq_strict_monot}.
We show now that $ \Omega_{t\lambda_1+(1-t)\lambda_2}=\Omega$: assuming w.l.o.g. $\lambda_1<\lambda_2$, set $\mu=t\lambda_1+(1-t)\lambda_2$ and $\lambda= \frac{1}{\frac{t}{\lambda_1}+\frac{1-t}{\lambda_2}}$, we have $\mu>\lambda>h(\Omega)$ (since $\lambda_2>h(\Omega)$); if by contradiction $\Omega_{\mu}\neq \Omega$, then \eqref{eq_strict_monot} would {contradict} $ \Omega_{\lambda} = \Omega_{\mu}$, which in {turn} follows by \eqref{eq_minkow_sum_int} and the assumption $t \Omega_{\lambda_1} \oplus (1-t)\Omega_{\lambda_2} = \Omega_{\mu}$.  As a consequence of monotonicity, $\Omega_{\lambda_2} = \Omega$; by the assumption, thus, $t \Omega_{\lambda_1} \oplus (1-t)\Omega = \Omega = t\Omega \oplus (1-t)\Omega$: thanks to {the} cancellation property of convex sets{, applied to the closures (which are compact)} (see e.g. \cite[Remark 1.7.6]{Sch14}), we have {$\overline{\Omega_{\lambda_1}} = \overline{\Omega}$, hence} $\Omega_{\lambda_1} = \Omega${, both sets being open and convex}, and hence the claim.
\end{remark}
}

{We conclude with a further consequence of Corollary~\ref{cor_minkowski},
obtaining the following Brunn--Minkowski inequalities.
\begin{corollary}[Brunn--Minkowski inequalities]
\label{cor_volume_polar}
Let $\Omega\subset\mathbb R^N$, $N\geq2$, be open,
bounded and convex, and let $K$ be its Cheeger set.
Assume $0\in K$. Then, for every
$\lambda_1,\lambda_2\geq h(\Omega)$ and $t\in[0,1]$,
setting $\lambda_t:=
\left(\frac{t}{\lambda_1}
+\frac{1-t}{\lambda_2}\right)^{-1}$, 
we have
\begin{equation}\label{eq_volume_BM}
|\Omega_{\lambda_t}|^{1/N}
\geq
t|\Omega_{\lambda_1}|^{1/N}
+(1-t)|\Omega_{\lambda_2}|^{1/N},
\end{equation}
and
\begin{equation}\label{eq_polar_BM}
|\Omega_{\lambda_t}^{\circ}|^{-1/N}
\geq
t|\Omega_{\lambda_1}^{\circ}|^{-1/N}
+(1-t)|\Omega_{\lambda_2}^{\circ}|^{-1/N},
\end{equation}
where $A^\circ$ denotes the polar body of $A$
with respect to the origin.
\end{corollary}

\begin{proof}
The cases $t=0,1$ are trivial, so let $t\in(0,1)$.
Set $K_i:=\overline{\Omega_{\lambda_i}}$,
$i=1,2,t$. By Corollary~\ref{cor_minkowski}, $tK_1+(1-t)K_2\subseteq K_t$, thus the classical Brunn--Minkowski inequality
immediately yields \eqref{eq_volume_BM}.

For \eqref{eq_polar_BM}, set $C_i:=K_i^\circ$,
$i=1,2$, and $D_t:=(tK_1+(1-t)K_2)^\circ$. 
Since $0\in\operatorname{int}K_i$, the bipolar
theorem gives
$$
D_t=(tC_1^\circ+(1-t)C_2^\circ)^\circ,
$$
which is the \emph{harmonic combination} of $C_1$
and $C_2$ \cite[p.~452]{Fir61}. 
Thus, the harmonic (dual) Brunn--Minkowski inequality
\cite[Theorem 3]{Fir61}
(see also \cite[Theorem 6]{Lut75}) yields
$$
|D_t|^{-1/N}
\geq
t|C_1|^{-1/N}
+(1-t)|C_2|^{-1/N}.
$$
Since, by Corollary~\ref{cor_minkowski}, $K_t^\circ\subseteq D_t$, the second
claim follows.
\end{proof}
}

\section{Some further remarks and extensions}
\label{sec_further}

\subsection{The weighted torsion case $\gamma=0$}
\label{sec_torsion}

{In this subsection we briefly discuss the torsion case, that is
\begin{equation}\label{eq_torsion_1}
\begin{cases}
-\Delta_1 u = a(x) & \quad \hbox{in $\Omega$}, \\
u = 0 &\quad \hbox{on $\partial \Omega$},
\end{cases}
\end{equation}
approximated by
\begin{equation}\label{eq_torsion_p}
\begin{cases}
-\Delta_p u_p = a(x)& \quad \hbox{in $\Omega$}, \\
u_p = 0 &\quad \hbox{on $\partial \Omega$}.
\end{cases}
\end{equation}
As we shall see, this case is much more degenerate than the singular one.

We first record the improvement of
\cite[Theorem 1.4]{GaSq25} that follows from Proposition \ref{prop_approx_general}, and then we summarize in a
 remark what is known about the limit $p \to 1^{+}$.}

\begin{theorem}[\cite{GaSq25}, $\gamma=0$, improved]
\label{thm_GS_torsion_imp}
Let $\Omega \subset \R^N$, $N\geq 2$, be open, bounded and convex, and let $p>1$. Assume $a(x)$ positive on $\Omega$ and
 $\theta$-concave with $\theta \geq 1$. Then the solution $u$ of \eqref{eq_torsion_p}
 is such
that $u^{\frac{\theta(p-1)}{1+\theta p}}$ is concave. If $a$ is constant, then $u^{\frac{p-1}{p}}$ is concave.
\end{theorem}

\begin{proof}
{Note first that $a>0$ and $\theta$-concave with $\theta \geq 1$ imply that $a$ is concave, hence
$a \in L^{\infty}(\Omega)\cap C(\Omega)$ and $\inf_{\Omega'} a>0$ for every $\Omega' \Subset \Omega$.}
Let $\Omega^k$ be as in \eqref{eq_OmegaktoOmega} {and $\eps_k<\frac{1}{2} \dist(\Omega^k, \partial \Omega)$}.

Set $a_k:= (a^{\theta}*\rho_k)^{1/\theta}$, defined on $\Omega^k$ and extended by $0$ outside, $\rho_k$ being a
standard mollifier {with support in $B_{\eps_k}(0)$}; then $a_k \to a$ in $L^{p'}(\Omega)$. %
{Note that $a_k$ is smooth and positive on $\Omega^k$: indeed $a^{\theta}$ is concave, hence $a^{\theta}*\rho_k$
is concave and smooth, and it is bounded away from zero on $\Omega^k$ by the first observation, so that raising to
the power $1/\theta$ preserves smoothness; moreover $a_k$ is $\theta$-concave on $\Omega^k$, since convolution
preserves concavity.}

{By \cite[Theorem 1.4]{GaSq25} the minimizer $u_k$ of
$J_k(v) := \frac{1}{p} \int_{\Omega^k} |\nabla v|^p - \int_{\Omega^k} a_k v$
on $W^{1,p}_0(\Omega^k)$ satisfies that $u_k^{\frac{\theta(p-1)}{1+\theta p}}$ is concave in $\Omega^k$. The datum
$f_k(x,t):=a_k(x)$ is independent of $t$, hence trivially non-increasing and %
bounded, so that
\eqref{eq_hyp_f_approx} and \eqref{eq_conv_f_approx} hold.
Proposition \ref{prop_approx_general}, applied with
$\varphi(t)=t^{\frac{\theta(p-1)}{1+\theta p}}$, gives the claim.}
\end{proof}

{
\begin{remark}[The limit $p \to 1^{+}$ for $\gamma=0$: a summary]
\label{rem_torsion_summary}
Contrary to the singular case, for $\gamma=0$ the limit $p\to1^{+}$ is essentially always degenerate. We
summarize the known picture, referring to \cite{CiTr03}, \cite[Theorems 4.1, 4.3, Remark 4.7]{MdLT08},
\cite[Theorems 3.1, 4.1, 4.2, Remark 4.5]{MdLT09}, \cite{OrPo17}, \cite{BEM16} (see also
\cite[Theorems 3.4, 7.1, Remark 7.2]{DGOP19}, \cite[Section 6.2]{LOPdL21}){, and to \cite{KaFr03} for the relation with the Cheeger constant}.
\begin{itemize}[leftmargin=*]
\item[$(a)$] Problem \eqref{eq_torsion_1} has no uniqueness: if $\phi$ is non-decreasing, locally
Lipschitz and $\phi(0)=0$, and $u$ is a solution, then $\phi(u)$ is a solution as well. %
This is a structural
consequence of the $0$-homogeneity of $\Delta_1$ in the gradient.
\item[$(b)$] Restricting to the solutions obtained as limits of $u_p$, the behaviour is almost dichotomic:
the limit is finite a.e.\ if and only if
$\norm{a}_{W^{-1,\infty}(\Omega)}\leq 1$, but it vanishes identically if the inequality is strict;
similar information holds w.r.t.\ the $L^{N,\infty}(\Omega)$ and $L^{N}(\Omega)$ norms,
e.g. $\norm{a}_{L^{N,\infty}(\Omega)}< (N-1)\omega_N^{1/N}$ or
$\norm{a}_{L^N(\Omega)}\leq N\omega_N^{1/N}$ imply a trivial limit.
When $a\equiv1$ these reduce to the dichotomy
$h(\Omega)>1$ versus $h(\Omega)<1$ of \cite{BEM16}, the critical case $h(\Omega)=1$ being known only for
$\Omega=B_r$ \cite{Kaw90} and, in general, still open.
\item[$(c)$] All the above thresholds are mutually consistent, due to Sobolev-type inequalities in $BV$:
the relevant constants are
$\mathcal{S}_1=\big(N \omega_N^{1/N}\big)^{-1}$ for the embedding $BV(\Omega)\hookrightarrow
L^{\frac{N}{N-1}}(\Omega)$ and $\tilde{\mathcal{S}}_1=\big((N-1)\omega_N^{1/N}\big)^{-1}$ for the
Lorentz refinement $BV(\Omega)\hookrightarrow L^{\frac{N}{N-1},1}(\Omega)$ (the latter being a sharper
form of the former, since $L^{\frac{N}{N-1},1}(\Omega) \hookrightarrow L^{\frac{N}{N-1}}(\Omega)$, and it is
what produces the $L^{N,\infty}$-threshold above {by duality}), so that
$$\norm{a}_{W^{-1,\infty}(\Omega)}\ \leq\ \min\Big\{\mathcal{S}_1 \norm{a}_{L^{N}(\Omega)},\
\tilde{\mathcal{S}}_1 \norm{a}_{L^{N,\infty}(\Omega)}\Big\} .$$
{Here the Lorentz quasi-norm is normalized as $\norm{a}_{L^{N,\infty}(\Omega)}:=\sup_{t>0}t\,|\{|a|>t\}|^{1/N}$, so that $\int_\Omega|a\,v|\le\tilde{\mathcal S}_1\norm{a}_{L^{N,\infty}(\Omega)}\,|Dv|(\R^N)$ for every $v\in BV(\R^N)$ vanishing outside $\Omega$, the threshold being attained by $a(x)=(N-1)/|x|$.} We refer to \cite{Alv77,BeSh88, CiTr03} for these constants, and we stress that the
two arguments of the minimum are not comparable with each other.
\item[$(d)$] The fractional counterpart of this problem is treated in \cite{Buc25}; here the dichotomy is
expressed in terms of the weighted Cheeger constant $h_a(\Omega)$ (adapted to the weight $a(x)$, {see also Open Problem \ref{op_weighted}}), and this is
coherent with the previous results due to the equality
$h_a(\Omega)=\norm{a}_{W^{-1,\infty}(\Omega)}^{-1}$ (which can be checked, but goes beyond the scopes of
the present paper).
In \cite{Buc25} it is also shown that the limit solutions are ``flat'', in the sense
that $\{u=\norm{u}_{\infty}\}$ has positive measure.
\end{itemize}
\end{remark}
}

{
\begin{remark}[Concavity in the torsion case]
\label{rem_caso_gamma0}
Combining Theorem \ref{thm_GS_torsion_imp} with Theorem \ref{thm_rock}, and observing that
$$\Big(\tfrac{\theta(p-1)}{1+\theta p}\Big)^{-1}\Big(u_p^{\frac{\theta(p-1)}{1+\theta p}}-1\Big)
\longrightarrow \log(u) \qquad \hbox{as } p \to 1^{+},$$
one obtains the following conditional statement: if $\Omega$ is open, bounded and convex, $a>0$ is
$\theta$-concave with $\theta \geq 1$, and $u$ is a limit of the $u_p$'s as in Remark
\ref{rem_torsion_summary}, then $\log(u)$ is concave on every open convex set $\Lambda$ such that
$0<u<+\infty$ a.e.\ in $\Lambda$; in particular $u$ is continuous there, and its maximum set is convex, so
that the ``flat zone'' of \cite{OrPo17} is a convex set{, in accordance with} the fractional case \cite{Buc25}.

By Remark \ref{rem_torsion_summary}$(b)$, however, such a $\Lambda$ can be non-empty only in the critical
regime, and we do not know whether this actually occurs.
\end{remark}
}

\subsection{The non-autonomous case: a non-optimal result for $p>1$}
\label{sec_nonautonomous}

We consider now, for $\gamma \in (0,1]$,
\begin{equation}\label{eq_nonaut_1}
\begin{cases}
-\Delta_1 u = a(x) u^{-\gamma}& \quad \hbox{in $\Omega$}, \\
u = 0 &\quad \hbox{on $\partial \Omega$},
\end{cases}
\end{equation}
and its $p$-Laplace approximation
\begin{equation}\label{eq_nonaut_p}
\begin{cases}
-\Delta_p u_p = a(x) u_{p}^{-\gamma} & \quad \hbox{in $\Omega$}, \\
u_p = 0 &\quad \hbox{on $\partial \Omega$}.
\end{cases}
\end{equation}
Existence of a non-negative solution $u \in BV(\Omega)\cap L^{\infty}(\Omega)$ of \eqref{eq_nonaut_1}
for $0\leq a \in L^N(\Omega)$, together with its approximation by the solutions $u_p$ of
\eqref{eq_nonaut_p} and with uniqueness and positivity when $a>0$, is proved in
\cite[Theorems 5.1, 5.3]{DGS19} (see also \cite{DGOP19,LOPdL21}).

In accordance with \cite[Theorems 1.4, 6.7]{GaSq25}, one would expect the solution $u_p$ to be
$\frac{\theta(p-1+\gamma)}{1+\theta p}$-concave whenever $a>0$ is $\theta$-concave with $\theta\geq1$, and
consequently the solution of \eqref{eq_nonaut_1} to be $\frac{\theta \gamma}{1+\theta}$-concave.
Unfortunately, achieving this seems to be a hard task, for the reasons explained in
\cite[Remark 6.9(ii)]{GaSq25} and recalled in Remark \ref{rem_stronger_result} below. What one obtains is the
following weaker statement. In what follows, we {adopt the convention} $\frac{1}{0}=\infty$.

\begin{theorem}\label{thm_conc_plapl_tors}
Let $\Omega \subset \R^N$, $N\geq 2$, be open, bounded and convex, and let $\gamma \in (0, 1]$, $p>1$. Assume
$a>0$ on $\Omega$ and $a$ $\theta$-concave with $\theta \geq 1$.

Assume moreover
$$\gamma \in \left(0, \frac{p-1}{p}\right], \qquad \theta \geq \frac{p-1+\gamma}{p-1-p\gamma} .$$
Then the solution $u$ of \eqref{eq_nonaut_p} is $\frac{\theta(p-1)}{1+\theta p}$-concave.
\end{theorem}

The result is weaker than the expected one: essentially, we only recover the concavity of the torsion problem
$\gamma=0$, while in the singular case a better exponent is expected. Moreover the assumption on $\theta$ is
more restrictive, and the restriction on $\gamma$ does not allow to pass to the limit $p\to1^{+}$.
{As we shall see in the proof, however, this last restriction is \emph{sharp within the method}: by
\eqref{eq_rho_secondo} below, the condition $\gamma \leq \frac{p-1}{p}$ is {necessary} %
for the
concavity of the auxiliary function $\rho$ associated with the multiplicative splitting $g(t)=t^{\delta}$. Let
us stress that this concerns the function $\rho$ only, and is unrelated to the (failing) concavity of
$t^2\psi''/\psi'$ discussed in Remark \ref{rem_stronger_result}, which is what obstructs the \emph{optimal}
exponent.}

\begin{proof}[Proof of Theorem \ref{thm_conc_plapl_tors}]
{Since $\inf_{\Omega'} a>0$ on any $\Omega' \Subset \Omega$, by the approximation process of Proposition \ref{prop_approx_general} (and combining the ideas of the proofs of Theorems \ref{thm_GS_autonom_imp} and \ref{thm_GS_torsion_imp}) we can assume $\Omega$ smooth, $\inf_{\Omega} a>0$ and $a \in C^{1,\alpha}(\Omega)$.}

We first study the concavity of the solutions of
\begin{equation}\label{eq_gen_approx_plaplace}
\begin{cases}
-\Delta_p u = a(x) \dfrac{1}{(u+\eta)^{\gamma}} & \quad \hbox{in $\Omega$}, \\
u = 0 &\quad \hbox{on $\partial \Omega$},
\end{cases}
\end{equation}
for $\eta>0$, as done in \cite[Section 6.2]{GaSq25}, and then we let $\eta\to0^{+}$.
We write $f(x,t) := a(x) (t+\eta)^{-\gamma}$ as (see \cite[(F1), Section 5.1]{GaSq25}, with $k\equiv 0$ therein)
\begin{equation}\label{eq_decomp_fxt}
f(x,t) = \big(a(x) (t+\eta)^{-\gamma} t^{-\delta}\big)\, t^{\delta} =: h(x,t)\, g(t)
\end{equation}
for some $\delta \in [0,p-1)$ to be fixed, and we set $G(t):= \int_0^t g \sim t^{\delta+1}$ and $\psi$ so that \cite[(F4), Section 5.1]{GaSq25}
$$\frac{g(\psi(t))}{(\psi'(t))^{p-1}} = \frac{\psi''(t)}{\psi'(t)}.$$
Such $\psi$ is
the inverse of

$$\varphi(t):= \int_{0}
^t (G(s))^{-\frac{1}{p}}\, ds \sim t^{\frac{p-1-\delta}{p}},$$
so that, for some $C>0$,
$$\psi(t)= C t^{\frac{p}{p-1-\delta}}, \qquad \frac{\psi''(t)}{\psi'(t)} = \frac{1+\delta}{p-1-\delta}\,
\frac{1}{t} = \frac{C'}{t}.$$
To apply \cite[Corollary 5.6]{GaSq25} we need to check, for a suitable $\Theta >0$ making
$h(x,\psi(t))-\Theta>0$ on the domain of reference, namely \cite[(A1), Section 5.1.3]{GaSq25}
$$0<\Theta < \inf_{x \in \Omega} a(x)\, \inf_{t \in (0, {M}
] } \big((t+\eta)^{-\gamma}
t^{-\delta}\big),$$
{for some suitable $M>0$},
the following conditions: %
\begin{itemize}[leftmargin=*]
\item ($\varphi_1$)-($\varphi_2$): technical assumptions on $\varphi$ straightforwardly verifiable, see \cite[Section 5]{GaSq25};
\item (f1)--(f3), (fs): sub $p$-growth, positivity, regularity in $(0,+\infty)$, coercivity and uniqueness, straightforwardly verifiable;
\item (F2)-(F3): $G>0$ and $h>0$, which is clear;
\item (A2): $\frac{\psi''}{\psi'}$ non-increasing, which is clear, and
$t \mapsto \frac{\psi''(t)}{\psi'(t)} \big(h(x,\psi(t)) - \Theta\big)$ non-increasing (Step 1);
\item (A4): $\frac{\psi''}{\psi'}>0$, which is clear;
\item (A5): $t^2 \frac{\psi''}{\psi'}$ concave, which is clear since it is linear, and
$(x,t) \mapsto t^2\frac{\psi''(t)}{\psi'(t)} \big(h(x,\psi(t)) - \Theta\big)$ jointly concave (Step 2).
\end{itemize}
For the sake of simplicity we set $\mu:= \frac{p}{p-1-\delta}>1$, so that $\psi(t) \sim t^{\mu}$ and
$\varphi(t) \sim t^{1/\mu}$.

\smallskip

\noindent\textbf{Step 1: monotonicity.} We have
$$\frac{\psi''(t)}{\psi'(t)} \big(h(x,\psi(t)) - \Theta\big) = C' \frac{1}{t} \left(C^{-\delta}a(x)
\big(Ct^{\mu}+\eta\big)^{-\gamma}t^{-\mu \delta}- \Theta\right),$$
and, as long as the bracket is positive (which is how $\Theta$ has been introduced), the above expression is
decreasing.

\smallskip

\noindent\textbf{Step 2: concavity.} We have
$$ t^2\frac{\psi''(t)}{\psi'(t)} \big(h(x,\psi(t)) - \Theta\big) = C' t\, h(x,\psi(t)) - \Theta C' t, $$
and, the second term being linear, it suffices to show that the first one is concave. We have
$$t\, h(x,\psi(t)) = C^{-\delta-\gamma}a(x) \big(t^{\mu}+\eta'\big)^{-\gamma}t^{-\mu \delta +1},
\qquad \eta':=\eta C^{-1}.$$
To ensure the joint concavity we rely on the Kennington-type property \cite[Proposition 2.2]{GaSq25}: if $a$ is
$\theta$-concave and $t \mapsto \big(t^{\mu}+\eta'\big)^{-\gamma}t^{1-\mu \delta }$ is $\omega$-concave, with
$\theta, \omega \in [0,+\infty]$ and $\omega^{-1}+\theta^{-1}=1$, then the product is jointly concave. We are
thus led to study the concavity of
$$\rho(t):= \big(t^{\mu}+1\big)^{-\gamma \omega }t^{(1-\mu \delta) \omega}$$
(by the arbitrariness of $\eta$, up to a change of variable it is not restrictive to take $\eta'=1$).
We need at least $1-\mu \delta \geq 0$, so that $\rho(0)\neq \infty$, that is
$\delta \leq \frac{p-1}{p+1}$; and $1-\mu\delta \geq \mu \gamma$, so that $\rho(\infty) \neq 0$, that is
$$\delta \leq \frac{p-1-p\gamma}{p+1},$$
which forces $\gamma \leq \frac{p-1}{p}$.
Moreover, in order for $\rho$ to be asymptotically concave both at $0$ and at $\infty$, we need
$(1-\mu \delta) \omega \leq 1$, that is
$$\delta \geq \frac{(p-1)(\omega-1)}{(p+1)\omega-1} = \frac{p-1}{1+p\theta},$$
which, combined with the previous restriction, imposes $\theta \geq \frac{p-1+\gamma}{p-1-p\gamma}$.
We choose the optimal $\delta = \frac{(p-1)(\omega-1)}{(p+1)\omega-1}$, i.e.\ $(1-\mu \delta) \omega = 1$,
obtaining
$$\rho(t) = (t^{\mu}+1)^{-\gamma \omega} t, \qquad \mu = \frac{\omega-1+\omega p}{(p-1)\omega}.$$

{We now check that $\rho''\leq 0$ under the above restrictions. Setting $b:=\gamma \omega>0$ and
$A(t):=t^{\mu}+1$, a direct computation gives
\begin{equation}\label{eq_rho_primo}
\rho'(t) = A(t)^{-b-1}\big(1+(1-b\mu)\,t^{\mu}\big),
\end{equation}
\begin{equation}\label{eq_rho_secondo}
\rho''(t) = -\,b\,\mu\, t^{\mu-1} A(t)^{-b-2}\Big[(\mu+1) + (1-b\mu)\,t^{\mu}\Big] .
\end{equation}
Since $b,\mu>0$ and $A>0$, we conclude that
$$\rho''\leq 0 \hbox{ on } (0,+\infty) \qquad \Longleftrightarrow \qquad b\mu=\gamma \omega \mu \leq 1 ,$$
and the latter is exactly the condition $1-\mu\delta \geq \mu\gamma$ imposed above, because the optimal choice
of $\delta$ gives $(1-\mu\delta)\omega=1$, i.e.\ $\mu \gamma \leq 1/\omega$. In other words, the restrictions
on $\gamma$ and $\theta$ in the statement are \emph{necessary and sufficient} for the concavity of $\rho$ with
this choice of $\psi$, and are therefore intrinsic to the method rather than technical.}

{
We thus obtain the $\frac{\theta(p-1)}{1+\theta p}$-concavity for the solutions $u_{\eta}$ of
\eqref{eq_gen_approx_plaplace}. Suitably applying Proposition \ref{prop_approx_general} we can weaken the regularity, and then we conclude by passing to the limit $\eta\to 0^{+}$ (via Theorem \ref{thm_rock} and \cite[Lemma 4.5 and proof of Theorem 4.6]{CST16}).
}
\end{proof}

{
\begin{remark}[Why the expected exponent is out of reach]
\label{rem_stronger_result}
In order to obtain the expected $\frac{\theta(p-1+\gamma)}{1+\theta p}$-concavity of $u_p$, %
one would need a different decomposition in \eqref{eq_decomp_fxt},
$$a(x)(t+\eta)^{-\gamma} = f(x,t) = h(x,t) g(t);$$
once a suitable $g$ is chosen, setting, as in \eqref{eq_def_varphi},
$G(t)=\int_0^t g$, $\varphi(t)=\int_{1}^t G^{-1/p}$ and $\psi=\varphi^{-1}$, the key quantity in play is
\begin{equation}\label{eq_kappa_formula}
{\mathcal{Q}}(t):=\frac{\psi''(t)}{\psi'(t)} = \frac1p\, G(\psi(t))^{\frac{1-p}{p}}\, g(\psi(t)) .
\end{equation}

{Two very different situations then occur, according to whether $a$ is constant or not.}

{\emph{The autonomous case.} If $a$ is constant, then we can choose $h\equiv 1$ and $g(t)=g_{\eta}(t)=(t+\eta)^{-\gamma}$, %
and by \cite[Remark 5.8]{GaSq25} conditions $(A_2)$, $(A_4)$ and $(B_4)$ collapse
to the single requirement}
\begin{equation}\label{eq_autonomous_condition}
{t \mapsto {\mathcal{Q}}(t) \quad \hbox{non-increasing and harmonic concave,}}
\end{equation}
{for which a sufficient condition, expressed directly on $g$, is that $G^{1/p}$ be concave and $g/G$ be
harmonic concave. For $g_{\eta}(t)=(t+\eta)^{-\gamma}$ with $\gamma \in (0,1]$ both hold, and this is exactly
how Theorem \ref{thm_GS_autonom} is proved in \cite[Theorem 6.7]{GaSq25}, observing that $G_{\eta}(t)=\int_0^t g_{\eta}$, $\varphi_{\eta}(t)=\int_{1}^t G_{\eta}^{-1/p} \sim t^{\frac{p-1+\gamma}{p}}$ as $\eta \to 0$. We stress that the quantity
$t^2{\mathcal{Q}}(t)$ plays \emph{no role} here: the factor $t^2$ is needed only to handle sums, that is, only when
an $x$-dependent weight is present.}

\emph{The non-autonomous case.} If $a$ is merely $\theta$-concave with $\theta<+\infty$, the situation is more involved \cite[Remark 6.9(ii)]{GaSq25}: condition
$(A_5)$ of \cite{GaSq25} requires the joint concavity of $(x,t)\mapsto t^2 {\mathcal{Q}}(t)\big(h(x,\psi(t))-\Theta\big)$.

To gain this, we need a suitable decomposition of $f(x,t)=h(x,t) g(t)$.

Consider first $h(x,t) \equiv a(x)$, $g(t)=(t+\eta)^{-\gamma}$ as in the autonomous case.
Condition $(A_5)$ of \cite{GaSq25} (see also (B4) therein) requires the joint concavity of $(x,t)\mapsto t^2 {\mathcal{Q}}(t)\big(a(x)-\Theta\big)$,
and by \cite[Proposition 2.2]{GaSq25} this essentially holds provided $t^2{\mathcal{Q}}(t)$ is $\omega$-concave with
$\omega=\frac{\theta}{\theta-1}>1$. This is a genuinely stronger requirement than
\eqref{eq_autonomous_condition}, and it {fails: indeed, for every $\eta\geq 0$ one has $t^2 \mathcal{Q}(t) \sim c\,t$ as $t\to +\infty$ (when $\gamma<1$; a similar argument works for $\gamma=1$),}
which is not $\omega$-concave for any $\omega>1${, since positive $\omega$-concave functions grow sublinearly}.
Notice anyway that this is not unexpected, since, if true, we would obtain a $\frac{p-1+\gamma}{p}$-concavity, which is better than the expected $\frac{\theta(p-1+\gamma)}{1+\theta p}$-concavity.

If one instead follows \eqref{eq_decomp_fxt}, as already checked in Step 2 of the proof of Theorem \ref{thm_conc_plapl_tors}, the best (i.e. smallest) choice of $\delta$ is $\delta=\frac{p-1}{1+\theta p}$, while to gain the desired concavity for $u$ one would need $\delta = \frac{p-1-\gamma p \theta}{1+\theta p}$.
{We do not know whether a different regularization of the nonlinearity, or a cleverer decomposition of $f(x,t)$, may lead to the expected concavity; see Open Problem \ref{op_optimal_concavity}.}
\end{remark}
}

\subsection{Some open problems}
\label{sec_open}

{We conclude by collecting some natural questions left open by the previous analysis.}

\begin{oproblem}[The case $\gamma>1$]
\label{op_gammamag1}
In this paper we {mainly} treated the case $\gamma\in (0,1]$, since the method due to \cite{GaSq25} imposes this restriction (see \cite[Remark 6.9(i)]{GaSq25}).

Notice however that concavity properties for the classical Laplacian, $p=2$, {have} been obtained through different methods also in the case $\gamma>1$ \cite{BGP99},
 and {it is natural to ask whether they hold for $p \neq 2$ as well}.
Anyway, for $p=1$ such concavity is now proved, as shown in Corollary \ref{cor_ext_gammamag1}.
\end{oproblem}

\begin{oproblem}[The threshold $\gamma \leq \frac{p-1}{p}$]
\label{op_threshold}
{By \eqref{eq_rho_secondo}, the restriction $\gamma \leq \frac{p-1}{p}$ of Theorem
\ref{thm_conc_plapl_tors} is {necessary} %
for the concavity of the auxiliary function $\rho$
associated with the choice $g(t)=t^{\delta}$. Is it a genuine structural threshold for the problem, or only a
limitation of this particular family of transformations? The question is not academic: since
$\frac{p-1}{p}\to0$ as $p \to 1^{+}$, the restriction is exactly what prevents one from passing to the limit
in the non-autonomous problem \eqref{eq_nonaut_1}.
If it could be removed, one would obtain, letting
$p \to 1^{+}$ in the exponent $\frac{\theta(p-1)}{1+\theta p}$ of Theorem \ref{thm_conc_plapl_tors}, that the
solution of \eqref{eq_nonaut_1} with $\gamma \in (0,1]$ is $0$-concave, i.e.\ \emph{log-concave}, and in
particular continuous{, a property which, to the best of our knowledge, is not known}. Obtaining instead the exponent $\frac{\theta\gamma}{1+\theta}$, which is the limit of
the \emph{expected} exponent $\frac{\theta(p-1+\gamma)}{1+\theta p}$, would require in addition a positive
answer to Open Problem \ref{op_optimal_concavity}.
}
\end{oproblem}

\begin{oproblem}[Optimal concavity for the singular $p$-Laplacian]
\label{op_optimal_concavity}
Can the expected $\frac{\theta(p-1+\gamma)}{1+\theta p}$-concavity be established for the solutions $u_p$
of \eqref{eq_nonaut_p}?
As explained in Remark \ref{rem_stronger_result}, the method in \cite{GaSq25} fails due to $(A_5)$, even if we do not exclude that a suitable decomposition of $f(x,t)$, or a cleverer regularization (e.g., a smooth approximation of $\max\{u,\eta\}^{-\gamma}$, which coincides with $u^{-\gamma}$ out of the collar $\{u\leq \eta\}$, observing that the assumptions in \cite{GaSq25} must be verified out of a collar, see (5.4) therein), could work. %
Consequently, a positive answer would require either a relaxation of $(A_5)$ of \cite{GaSq25} (see e.g. the more abstract $(B_4)$ therein), or a genuinely different method.
\end{oproblem}

\begin{oproblem}[Boundary behaviour of $H_{\Omega}$]
\label{op_ext_bound}
{Assume $\partial \Omega \in C^{1,1}$, so that $H_{\Omega}\in L^{\infty}(\Omega)$ by Remark
\ref{rem_ifflinf}. Does $H_{\Omega}|_{\Omega}$ extend continuously up to $\overline{\Omega}$? By Corollary
\ref{cor_var_mean_curv} the function $H_{\Omega}$ is concave and bounded in $\Omega$, hence it admits a
continuous extension along each segment; the delicate point is the uniformity near $\partial \Omega$.
By Proposition \ref{prop_no_lip}, no {uniform} H\"older modulus of continuity can be expected, the answer
being negative for every exponent $\alpha \in (0,1]$; weaker moduli, logarithmic ones for instance, are not
excluded by our examples. The question of plain continuity, however, remains open, as
does the corresponding one for $u_{|\Omega}=(-H_{\Omega})^{-1/\gamma}$. {We recall that bounded concave functions on convex domains need not admit a continuous extension to the closure when {$\Omega$ is not a polytope}, %
this property characterizing polyhedral sets {(see Proposition \ref{prop_polytope} and \cite{GKR68})}; the question thus concerns the specific structure of $H_\Omega$.}
Let us stress that, with our convention, $H_{\Omega}$ is never continuous \emph{across}
$\partial \Omega$, already for $\Omega=B_r$ (see Example \ref{ex_examp_optimal}$(i)$), so that the question has
to be phrased for the restriction of $H_{\Omega}$ to $\Omega$.
Notice that, with the alternative convention of \cite[(2.16)--(2.17)]{GoMa94}, where one sets
$H_{\Omega}:=-H_{\R^N \setminus \Omega}\geq0$ outside $\Omega$, the sign of $H_{\Omega}$ jumps across
$\partial \Omega$ anyway, so the discontinuity is not an artifact of our convention; the question is therefore
genuinely about $H_{\Omega}|_{\Omega}$ in any case.}
\end{oproblem}

\begin{oproblem}[The weighted variational mean curvature]
\label{op_weighted}
Can one relate the solutions of \eqref{eq_nonaut_1} to the \emph{weighted} canonical variational mean
curvature $H_{\Omega, a}$? We recall that the weighted construction is classical: given a positive weight
$a$, one considers the measure $\mu:=a\,\mathcal{L}^N\llcorner \Omega$ and minimizes
$\mc{F}^{\mu}_{\lambda}(E):=P(E)-\lambda\,\mu(E) = P(E) - \lambda \int_E a(x)\,dx$ among $E \subseteq \Omega$;
existence of a variational mean curvature associated with a general measure goes back to
\cite[Theorem 1]{BGT87} and \cite{Bar94}, and the corresponding minimal (canonical) choice is constructed in
\cite[Construction 2.6 and Theorem 2.7]{ScSc23}{; weighted Cheeger sets have been studied in \cite{IoLR05,CaCo07,BCC07,Sar18}}. Selecting (if not unique) a suitable $\Omega_{\lambda,a}\in %
\argmin \mc{F}^{\mu}_{\lambda}$
and $H_{\Omega,a}(x):=-\inf\{\lambda {a(x)}%
\mid x \in \Omega_{\lambda,a}\}$, one has
$\Omega_{\lambda}=\Omega_{\lambda,\chi_{\Omega}}$ and $H_{\Omega}=H_{\Omega,\chi_{\Omega}}$.

An affirmative answer would give considerably more interest to the concavity properties of
\eqref{eq_nonaut_1}: by Theorem \ref{thm_conc_plapl_tors} one %
{may hope to} obtain (probably non-optimal) concavity
information on $H_{\Omega,a}$ itself. We stress, however, that at present such information cannot be
deduced from \eqref{eq_nonaut_p} and Theorem \ref{thm_conc_plapl_tors}, because of the threshold discussed in
Open Problem \ref{op_threshold}.

\end{oproblem}

\begin{oproblem}[Beyond convexity]
In the present paper, convexity of $\Omega$ is strongly used, starting from the notion itself of harmonic convexity of ${-H_{\Omega}}$.
For a planar Jordan domain with no necks of radius $1/\lambda$ for any $\lambda \geq h(\Omega)$ %
(in the sense of \cite[Definition 1.1]{LeSa20}), formula \eqref{eq_formula_H_planar} does survive,
since one still has $\Omega_{\lambda}=A_{1/\lambda}$, up to replacing
$\Omega_{\lambda}$ by the maximal minimizer \cite[Theorem 2.3]{LeSa20}; hence, $H_{\Omega}$ remains completely
described by the purely geometric function $r_{\Omega}$ of \eqref{eq_def_rOmega}.
For non-convex domains, it remains open to study geometric properties of pre-images of sublevel sets of $-H_{\Omega}$, and to study regularity properties of $-H_{\Omega}$.
\end{oproblem}

\appendix

\section{Examples} %
\label{sec_examples}

\noindent In this appendix we collect the explicit computations of $H_{\Omega}$ that have been used
throughout the paper. All of them rest on the planar structure result of Lemma \ref{lem_planar_structure} and
on the ensuing formula \eqref{eq_formula_H_planar}, so that in each case only the geometric quantity
$r_{\Omega}$ of \eqref{eq_def_rOmega} has to be computed. Examples \ref{ex_examp_optimal}$(i)$ and $(ii)$ are
used in Sections \ref{sec_intro}--\ref{sec_cheeg}, Example \ref{ex_examp_optimal}$(iii)$ provides the
optimality in Corollary \ref{cor_var_mean_curv}, and Examples \ref{ex_examp_optimal}$(iii)$--$(v)$ are the
ones behind Proposition \ref{prop_no_lip}.

\begin{example}\label{ex_examp_optimal}
We collect here some explicit computations; examples $(i)$, $(iii)$ and $(iv)$ are taken from \cite{GoMa94}.
\begin{itemize}[leftmargin=*]
\item[$(i)$]
{Let $\Omega = B_r \subset \R^N$ be an open ball. Then
$$H_{\Omega} \equiv - \frac{N}{r} = -h(B_r) \qquad \hbox{in } B_r .$$
Indeed, one may restrict \eqref{eq_min_Omegalamda} to the balls $B_s \subseteq B_r$: if $E \subseteq B_r$
and $B_s$ denotes the ball with $|B_s|=|E|$, then $s \leq r$, so that $B_s$ is (up to a translation) an
admissible competitor, and the isoperimetric inequality gives $P(E)\geq P(B_s)$, whence
$P(E)-\lambda|E| \geq P(B_s)-\lambda|B_s|$. Testing with $B_s$, $s \leq r$, gives
$P(B_s)-\lambda|B_s| = \omega_N s^{N-1}(N-\lambda s)$, which is negative if and only if $\lambda>N/s$; hence
$\Omega_{\lambda}=\emptyset$ for $\lambda<N/r$ and $\Omega_{\lambda}=B_r$ for $\lambda \geq N/r$, and
\eqref{eq_def_Homega} gives the claim. In particular $B_r$ is self-Cheeger, with
$h(B_r)=P(B_r)/|B_r| = N/r$, consistently with \eqref{eq_H_const_on_Cheeger}.

Note that $N/r$ is \emph{not} the classical mean curvature $(N-1)/r$ of $\partial B_r$: as explained in the
Introduction, $H_{\Omega}$ is the canonical variational mean curvature, and it is discontinuous across
$\partial \Omega$, so that \cite[Remark 2.1]{GoMa94} does not apply to it. Correspondingly, by
\eqref{eq_sol_eq_Hom1} the solution of \eqref{eq_gen_1laplace_a1} on a ball is the constant
$u \equiv (r/N)^{1/\gamma}$, which does not vanish on $\partial B_r$: this is the elementary example behind
Remark \ref{rem_boundary_datum}.}

\item[$(ii)$] Let $\Omega$ be the convex hull in $\R^2$ of $B_r(0)$ and a point $P$ outside it, at distance
$d>r$ from the origin, and let $2\beta \in (0,\pi)$ be the angle of $\Omega$ at the vertex $P$, so that
$\sin \beta = r/d$.
Recall, for
$\rho \in (0, r]$, $A_{\rho}=\Omega^{\rho}\oplus B_{\rho}$ {(see \eqref{eq_def_Arho})}, %
that is $\Omega$ with the corner at $P$ rounded off at radius $\rho$. {Then, {recalling that} $R_{\Omega}=1/h(\Omega)$,
\begin{equation}\label{eq_example_hull}
H_{\Omega}(x) =
\begin{cases}
- h(\Omega) & \quad \hbox{for $x \in 
 {\Omega \cap \overline{A_{R_{\Omega}}}}$} \\[3pt]
- \dfrac{1}{\rho} & \quad \hbox{for $x \in \Omega \cap \partial A_{\rho}$, $\rho \in (0,R_{\Omega})$,}
\end{cases}
\qquad
h(\Omega) = \frac{1}{r} + \sqrt{\frac{\pi}{|\Omega|}} \, \in \Big(\frac{1}{r}, \frac{2}{r}\Big),
\end{equation}
and $A_{R_{\Omega}}$ is the (unique) Cheeger set of $\Omega$.

Let us justify \eqref{eq_example_hull}, since the computation is elementary and can be used for the optimality
in Corollary \ref{cor_var_mean_curv}. Set $c:=\cot \beta - \frac{\pi}{2}+\beta>0$. Elementary trigonometry
gives, for every $\rho \in (0,r]$,
\begin{equation}\label{eq_example_hull_geom}
P(A_{\rho}) = P(\Omega)-2c\,\rho , \qquad |A_{\rho}| = |\Omega| - c\,\rho^{2},
\end{equation}
so that
$$P(A_{\rho}) - \lambda |A_{\rho}| = \big(P(\Omega)-\lambda|\Omega|\big) - 2c\rho + \lambda c \rho^{2},
\qquad \frac{d}{d\rho}\Big[P(A_{\rho}) - \lambda |A_{\rho}|\Big] = 2c\,(\lambda \rho - 1) .$$
Since the coefficient of $\rho^{2}$ is positive, the minimum over this family is attained at
$\rho=1/\lambda$; this is of course nothing but the Euler--Lagrange condition, the free boundary of
$\Omega_{\lambda}$ having constant curvature $\lambda$, hence radius $1/\lambda$.

The minimization over the sole family $\{A_{\rho}\}$ is not restrictive: by Lemma
\ref{lem_planar_structure} one has $\Omega_{\lambda}=A_{1/\lambda}$ for every $\lambda \geq h(\Omega)$, and
$\Omega_{\lambda}=\emptyset$ otherwise, which gives the second line of \eqref{eq_example_hull}.
The threshold is computed through {\eqref{eq_Omegalambda_planar} and \eqref{eq_inner_cheeger}}:
$A_{R_{\Omega}}=\Omega^{R_{\Omega}}\oplus B_{R_{\Omega}}$, where $R_{\Omega}$ is determined by $|\Omega^{R_{\Omega}}|=\pi R_{\Omega}^{2}$, and
$h(\Omega)=1/R_{\Omega}$. Here $\Omega^{\rho}=(1-\rho/r)\,\Omega$ by homothety, so that $(1-R_{\Omega}/r)^{2}|\Omega|=\pi R_{\Omega}^{2}$,
which gives the value of $h(\Omega)$ in \eqref{eq_example_hull}; since $|\Omega|>\pi r^{2}$ one has
$\frac1r < h(\Omega) < \frac{2}{r}$.%

Consequently $H_{\Omega}$ is non-constant, $H_{\Omega}(x)\to -\infty$ as $x \to P$, and
$\partial \Omega \notin C^{1,1}$ at $P$, in accordance with Remark \ref{rem_ifflinf}. Moreover $H_{\Omega}$ is
constant, equal to $-h(\Omega)$, exactly on the Cheeger set of $\Omega$, as predicted by
\eqref{eq_H_const_on_Cheeger}.

Finally, let us record the behaviour along the axis, which is what yields the optimality in Corollary
\ref{cor_var_mean_curv}. If $x$ lies on the segment $\overline{OP}$ at distance $t$ from the origin, the
inscribed disk of radius $\rho$ tangent to both sides of the corner has center at distance $\rho/\sin\beta$
from $P$, hence $d-t = \rho\,\frac{d-r}{r}$, that is $\rho = \frac{r(d-t)}{d-r}$. Therefore
\begin{equation}\label{eq_example_axis}
-H_{\Omega}(t) = \frac{d-r}{r\,(d-t)} \qquad \hbox{for } t \in (t_{R_{\Omega}},d), \qquad t_{R_{\Omega}}:= d - R_{\Omega}\,\frac{d-r}{r},
\end{equation}
so that $t \mapsto \big(-H_{\Omega}(t)\big)^{-1}$ is \emph{affine} on $(t_{R_{\Omega}},d)$: the harmonic convexity of
Corollary \ref{cor_var_mean_curv} holds here with equality.}
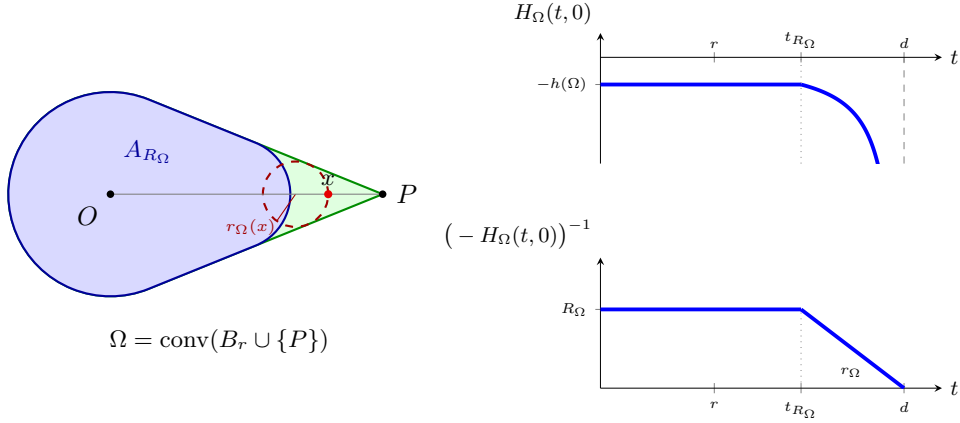
\begin{figure}[htbp]
\centering
\begin{tikzpicture}[scale=0.9]

\pgfmathsetmacro{\r}{1.5}
\pgfmathsetmacro{\d}{4}
\pgfmathsetmacro{\ROmega}{0.814165}
\pgfmathsetmacro{\tROmega}{2.643059}
\pgfmathsetmacro{\cxx}{1.828894}
\pgfmathsetmacro{\bet}{asin(\r/\d)}
\pgfmathsetmacro{\ang}{90-\bet}
\pgfmathsetmacro{\tx}{\r*\r/\d}
\pgfmathsetmacro{\ty}{\r*sqrt(1-(\r*\r)/(\d*\d))}
\pgfmathsetmacro{\tsam}{3.2}
\pgfmathsetmacro{\rhosam}{\r*(\d-\tsam)/(\d-\r)}
\pgfmathsetmacro{\csam}{\d-\rhosam*\d/\r}

\begin{scope}
\coordinate (O) at (0,0);
\coordinate (P) at (\d,0);
\coordinate (T1) at (\tx,\ty);

\fill[green!12]
 (T1) arc[start angle=\ang, end angle={360-\ang}, radius=\r] -- (P) -- cycle;
\draw[thick, green!55!black]
 (T1) arc[start angle=\ang, end angle={360-\ang}, radius=\r] -- (P) -- cycle;

\fill[blue!15]
 (T1) arc[start angle=\ang, end angle={360-\ang}, radius=\r]
 -- ({\cxx+\ROmega*cos(-\ang)},{\ROmega*sin(-\ang)})
 arc[start angle=-\ang, end angle=\ang, radius=\ROmega] -- cycle;
\draw[thick, blue!60!black]
 (T1) arc[start angle=\ang, end angle={360-\ang}, radius=\r]
 -- ({\cxx+\ROmega*cos(-\ang)},{\ROmega*sin(-\ang)})
 arc[start angle=-\ang, end angle=\ang, radius=\ROmega] -- cycle;
\node[blue!60!black, font=\small] at (0.55,0.62) {$A_{R_\Omega}$};

\draw[red!70!black, thick, dashed] (\csam,0) circle (\rhosam);
\draw[red!70!black, thin] 
 (\csam,0) -- ({\csam+\rhosam*cos(235)},{\rhosam*sin(235)});
\node[red!70!black, font=\tiny, anchor=north east]
 at ({\csam+0.56*\rhosam*cos(235)},{0.56*\rhosam*sin(235)})
 {$r_\Omega(x)$};

\draw[thin, gray] (O) -- (P);
\fill[red] (\tsam,0) circle (1.7pt);
\node[above, font=\small] at (\tsam,-0.01) {$x$};

\node[circle,fill,inner sep=1pt,label=below left:$O$] at (O) {};
\node[circle,fill,inner sep=1pt,label=right:$P$] at (P) {};
\node[font=\small] at (1.6,-2.15) {$\Omega=\mathrm{conv}(B_r\cup\{P\})$};
\end{scope}

\begin{scope}[shift={(7.2,0.45)}]
\begin{axis}[
 width=6.6cm, height=3.5cm,
 xlabel={$t$}, ylabel={$H_\Omega(t,0)$},
 xmin=0, xmax=4.5, ymin=-4.8, ymax=1.1,
 axis lines=middle,
 xtick={1.5,2.643059,4}, xticklabels={,,},
 ytick={-1.228253}, yticklabels={$-h(\Omega)$},
 samples=300,
 label style={font=\small}, tick label style={font=\tiny},
 xlabel style={at={(ticklabel* cs:1.0)},anchor=west},
 ylabel style={at={(ticklabel* cs:1.0)},anchor=south east}
]
\addplot[ultra thick,blue,domain=0:2.643059] {-1.228253};
\addplot[ultra thick,blue,domain=2.643059:3.72]
 {-(4-1.5)/(1.5*(4-x))};

\node[font=\tiny,anchor=south]
 at (axis cs:1.5,0.05) {$r$};
\node[font=\tiny,anchor=south]
 at (axis cs:2.643059,0.05) {$t_{R_\Omega}$};
\node[font=\tiny,anchor=south]
 at (axis cs:4,0.05) {$d$};

\draw[dashed,gray]
 (axis cs:4,-4.8) -- (axis cs:4,-0.04);
\draw[dotted,gray]
 (axis cs:2.643059,-4.8) -- (axis cs:2.643059,-0.04);

\end{axis}
\end{scope}

\begin{scope}[shift={(7.2,-2.85)}]
\begin{axis}[
 width=6.6cm, height=3.5cm,
 xlabel={$t$}, ylabel={$\big(-H_\Omega(t,0)\big)^{-1}$},
 xmin=0, xmax=4.5, ymin=0, ymax=1.35,
 axis lines=middle,
 xtick={1.5,2.643059,4}, xticklabels={$r$,$t_{R_\Omega}$,$d$},
 ytick={0.814165}, yticklabels={$R_\Omega$},
 samples=2,
 label style={font=\small}, tick label style={font=\tiny},
 xlabel style={at={(ticklabel* cs:1.0)},anchor=west},
 ylabel style={at={(ticklabel* cs:1.0)},anchor=south east}
]
\addplot[ultra thick,blue,domain=0:2.643059] {0.814165};
\addplot[ultra thick,blue,domain=2.643059:4] {0.6*(4-x)};
\draw[dotted,gray] (axis cs:2.643059,0) -- (axis cs:2.643059,0.814165);
\node[font=\tiny,anchor=north] at (axis cs:3.30,0.30) {$r_\Omega$};
\end{axis}
\end{scope}

\end{tikzpicture}
\caption{{\small Left: the domain $\Omega$, its Cheeger set $A_{R_\Omega}$, and the largest disk contained in
$\Omega$ through a point $x$ of the axis, of radius $r_{\Omega}(x)$ as in \eqref{eq_def_rOmega}. Right: the
graphs of $H_{\Omega}$ and of $\big(-H_{\Omega}\big)^{-1}=\min\{R_{\Omega},r_{\Omega}\}$ along $\overline{OP}$,
parametrized by the distance $t$ from $O$. Since $H_{\Omega}$ is a branch of a hyperbola with vertical
asymptote at $t=d$, its reciprocal is affine.
}}
\label{fig_hull}
\end{figure}

\item[$(iii)$] \emph{(The square).}
Let $\Omega=(0,\ell)^2 \subset \R^2$ be the open square of side $\ell$.
Since $\Omega^{\rho}$ is the square of side $\ell-2\rho$, the inner Cheeger formula \eqref{eq_inner_cheeger} $|\Omega^{R_{\Omega}}|=\pi R_{\Omega}^2$ %
gives $(\ell-2R_{\Omega})^2=\pi R_{\Omega}^2$, whence
\begin{equation}\label{eq_cheeger_square}
h(\Omega)=\frac1{R_{\Omega}}=\frac{2+\sqrt{\pi}}{\ell},
\end{equation}
and the Cheeger set is $A_{R_{\Omega}}$, i.e.\ the square with the four corners rounded off at radius $R_{\Omega}$.

By Lemma \ref{lem_explicit_Homega} %
we can compute $H_{\Omega}$ explicitly.
Let $x \in \Omega$ at distances $d_1,d_2$ from the two sides meeting at the (nearest) vertex, which we
place at the origin (see Figure \ref{fig_square_corner}).
Let $\rho=r_{\Omega}(x)$ be the maximal radius, $x \in \overline{B_{\rho}(c)}$, $B_{\rho}(c) \subset \Omega$. We show $\rho=d_1+d_2+\sqrt{2d_1d_2}=:\bar{d}$: first, it is easy to show that $x \in \partial B_{\bar{d}}((\bar{d},\bar{d}))$ if $\bar{d}<\ell/2$, to which we restrict from now on (otherwise, $r_{\Omega}(x)=\ell/2$); %
thus $\rho \geq \bar{d}$. Observed that $\bar{d} \geq \max\{d_1,d_2\}$, then if $\rho \leq \max\{d_1,d_2\}$ then $\rho \leq \bar{d}$ and the claim follows. Thus we can assume $d_i< \rho$. Since $B_{\rho}(c) \subset \Omega$, we clearly have $c_i \geq \rho$. Hence, being $c_i \geq \rho > d_i$, we obtain $(\rho-d_1)^2+(\rho-d_2)^2\leq (c_1-d_1)^2+(c_2-d_2)^2 \leq \rho^2$. Solving we obtain $\rho \leq \bar{d}$ and hence the claim. %

Hence, noticed that $R_{\Omega}<\ell/2$ (thus coherent with our above restriction), by \eqref{eq_formula_H_planar},
\begin{equation}\label{eq_H_square}
H_{\Omega}(x) = - \frac{1}{d_1+d_2+\sqrt{2 d_1 d_2}} \qquad \hbox{whenever } d_1+d_2+\sqrt{2d_1d_2}<R_{\Omega} .
\end{equation}
Three consequences are worth recording.
\begin{itemize}[leftmargin=*]
\item[$\bullet$] $H_{\Omega}(x)\to-\infty$ as $(d_1,d_2)\to(0,0)$, in accordance with Remark
\ref{rem_ifflinf}, since the square is not of class $C^{1,1}$.
\item[$\bullet$] Along the diagonal $d_1=d_2=d$ one has $\big(-H_{\Omega}\big)^{-1}=(2+\sqrt2)\,d$, which is
\emph{affine}: the square provides therefore an %
instance of optimality of
the exponent $-1$ in Corollary \ref{cor_var_mean_curv}.
\item[$\bullet$] Letting instead $d_1 \to 0^+$ with $d_2<R_{\Omega}$ fixed, that is approaching a \emph{flat} portion
of $\partial \Omega$ at distance less than $R_{\Omega}$ from a vertex, we get
$\big(-H_{\Omega}\big)^{-1} = d_2+\sqrt{2d_2}\,\sqrt{d_1}+O(d_1)$, whose derivative in $d_1$ blows up. Hence
$\big(-H_{\Omega}\big)^{-1}$ admits at best a $C^{0,1/2}$ extension up to $\partial \Omega$, and no Lipschitz
extension: see Proposition \ref{prop_no_lip} below.
\end{itemize}

\begin{figure}[htbp]
\centering
\begin{tikzpicture}

\begin{scope}[scale=3.1]
\pgfmathsetmacro{\RR}{0.795238} 
\pgfmathsetmacro{\da}{0.15} 
\pgfmathsetmacro{\db}{0.25} 
\pgfmathsetmacro{\rh}{0.673861} 

\fill[green!10] (0,0) rectangle (1.42,1.42);

\fill[orange!18] (0,0) -- (\RR,0)
 arc[start angle=270, end angle=180, radius=\RR] -- cycle;

\draw[thick,green!55!black] (0,1.42) -- (0,0) -- (1.42,0);
\draw[thick,blue!60!black] (\RR,0) arc[start angle=270, end angle=180, radius=\RR];
\node[blue!60!black,font=\tiny,anchor=west] at (0.47,0.11) {$\partial A_{R_{\Omega}}$};

\draw[gray,dashed,thin] (0,0) -- (1.38,1.38);
\node[gray,font=\tiny,anchor=south east,rotate=45] at (1.10,1.10) {$d_1=d_2$};

\draw[red!70!black,thick,dashed] (\rh,\rh) circle (\rh);
\fill[red!70!black] (\rh,\rh) circle (0.6pt);
\draw[red!70!black,->,>=stealth] (\rh,\rh) -- (\da,\db);
\node[red!70!black,font=\tiny,anchor=north west] at (0.16,0.59) {$r_\Omega(x)$};

\fill[red] (\da,\db) circle (0.8pt);
\node[font=\tiny,anchor=south east] at ({\da+0.01},{\db-0.09}) {$x$};

\draw[dotted] (\da,\db) -- (\da,-0.16);
\draw[dotted] (\da,\db) -- (-0.16,\db);
\draw[<->,>=stealth,thin] (0,-0.11) -- (\da,-0.11);
\node[font=\tiny,anchor=north] at ({\da/2},-0.12) {$d_1$};
\draw[<->,>=stealth,thin] (-0.11,0) -- (-0.11,\db);
\node[font=\tiny,anchor=east] at (-0.12,{\db/2}) {$d_2$};
\end{scope}

\begin{scope}[shift={(5.9,0.15)}]
\begin{axis}[
 width=6.9cm, height=5.2cm,
 xlabel={$\delta=\dist(\cdot,\partial\Omega)$},
 ylabel={$\big(-H_\Omega\big)^{-1}$},
 xmin=0, xmax=0.36, ymin=0, ymax=0.98,
 axis lines=left,
 xtick={0,0.1,0.2,0.3}, ytick={0.3,0.795238},
 yticklabels={$d_2^{*}$,$R_{\Omega}$},
 samples=250,
 label style={font=\small}, tick label style={font=\tiny}
]
\draw[dotted,gray] (axis cs:0,0.795238) -- (axis cs:0.36,0.795238);
\addplot[ultra thick,blue,domain=0:0.36] {min(0.795238,3.414214*x)};
\addplot[ultra thick,red!75!black,domain=0:0.36] {min(0.795238,x+0.3+sqrt(0.6*x))};
\node[blue,font=\tiny,anchor=north west] at (axis cs:0.155,0.57) {along $d_1=d_2$};
\node[red!75!black,font=\tiny,anchor=north west] at (axis cs:0.008,0.81)
 {at fixed $d_2=d_2^{*}$};
\end{axis}
\end{scope}

\end{tikzpicture}
\caption{{\small The square, $\ell=3$. Left: near a vertex the largest disk through $x=(d_1,d_2)$ is
centred at $(\rho,\rho)$, with $\rho=r_{\Omega}(x)=d_1+d_2+\sqrt{2d_1d_2}$; the shaded corner region is
$\Omega\setminus A_{R_{\Omega}}$, where $H_{\Omega}$ is non-constant. Right: the profiles of
$\big(-H_{\Omega}\big)^{-1}$ against the distance $\delta$ from $\partial\Omega$. Along the diagonal the
profile is affine, whence the optimality of the exponent $-1$; approaching a flat side at fixed
$d_2=d_2^{*}<R_{\Omega}$ it has vertical tangent, which rules out any Lipschitz extension.}}
\label{fig_square_corner}
\end{figure}
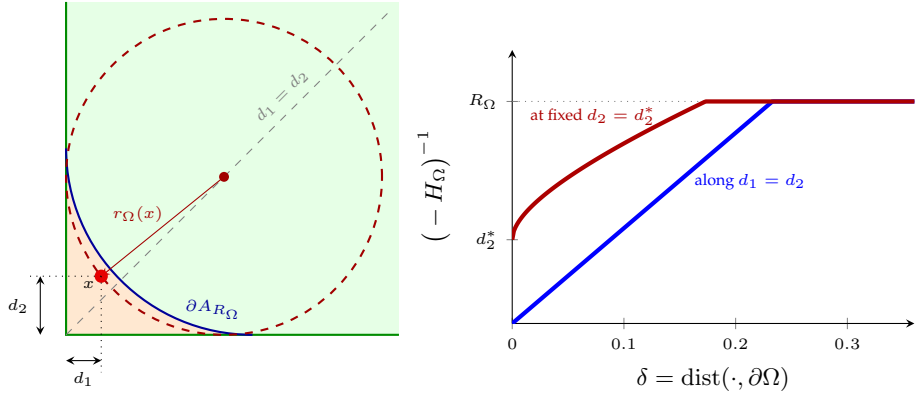

\item[$(iv)$] \emph{(The elongated ellipse.)}
Let $\Omega=\{x^2/a^2+y^2/b^2<1\}$ with $a>b>0$, and let
$e=\sqrt{1-b^2/a^2}$ be its eccentricity. The maximal curvature of $\partial \Omega$ is
$\bar \kappa = a/b^2$, attained at $(\pm a,0)$,
so that by \eqref{kappa_2} the ellipse is calibrable if and only
if $\bar \kappa\,|\Omega| \leq P(\Omega)$,
 i.e.\ $\pi a/b \leq 4 E(e)$, {where
$E(k)=\int_0^{\pi/2}\sqrt{1-k^2\sin^2\theta}\,d\theta$}{.}
The perimeter is not elementary, but a numerical solution of the resulting
equation gives the critical eccentricity
\begin{equation}\label{eq_critical_ecc}
e^{*}\approx 0.79117 ,
\end{equation}
see \cite[p.~107]{KaLa06} and \cite[Section 5]{GeHu23}: for $e>e^{*}$ the ellipse is \emph{not} calibrable,
and $H_{\Omega}$ is non-constant by \eqref{eq_H_const_on_Cheeger}, although $H_{\Omega}\in L^{\infty}(\Omega)$
by Remark \ref{rem_ifflinf}. This is thus a \emph{smooth} domain for which Corollary
\ref{cor_var_mean_curv} is not vacuous.

Assume $e>e^{*}$ and let us discuss $H_{\Omega}$ in a portion of $\Omega \setminus K$, namely close to the point of maximum curvature. Indeed, by \eqref{eq_formula_H_planar} we have $-H_{\Omega}=\max\{h(\Omega),1/r_{\Omega}\}$; once shown
\begin{equation}\label{eq_romega_ell}
r_{\Omega}(t,0)\to b^2/a=1/\bar\kappa \quad \hbox{as $t \to a^-$},
\end{equation}
being $\Omega$ not calibrable, by \eqref{kappa_3} we have $r_{\Omega}(t,0)<1/h(\Omega)$ for every $t$ close to $a$, thus here $-H_{\Omega} = 1/r_{\Omega} \neq h(\Omega)$.
We compute now $r_{\Omega}$, and in particular \eqref{eq_romega_ell}.
Writing the upper half of $\partial \Omega$ as $x_1=w(y)$, $w(y)=a\sqrt{1-y^2/b^2}$, and imposing that the
disk be tangent at $P=(w(y),y)$,
its centre lies on the normal line at $P$, of direction
$\nu=(-1,w'(y))$ (see Figure \ref{fig_ellipse_constr});
note that, by symmetry, the centre may be assumed to lie on the major axis: indeed, if a disk
$B_{\rho}(c)\subseteq \Omega$ contains $x=(t,0)$, so does its mirror image $B_{\rho}(\bar c)$ with respect to
the axis, and the Minkowski combination $\frac12 B_{\rho}(c)\oplus\frac12 B_{\rho}(\bar c)=B_{\rho}
\big(\frac{c+\bar c}{2}\big)$ is again contained in $\Omega$ by convexity, has the same radius, and still
contains $x$, its centre being on the axis.
{Let $\delta(q):=\dist((q,0),\partial\Omega)$ for $|q|<a$. Then $r_\Omega(t,0)=\max\{\delta(q)\mid |t-q|\le\delta(q)\}$, and for $t>0$ centres with $q<0$ can be discarded (if $q<0$ is admissible, so is $-q$, with the same radius). Being $\Omega$ convex and symmetric, $\delta$ is concave and even, hence non-increasing on $[0,a)$, while $q\mapsto q+\delta(q)$ is non-decreasing, $\delta$ being $1$-Lipschitz. Therefore, for $t\in[\delta(0),a)$,
\begin{equation}\label{eq_rOmega_symm}
r_\Omega(t,0)=\delta(q_t)=t-q_t,\qquad q_t:=\min\{q\ge0\mid q+\delta(q)\ge t\},
\end{equation}
and the corresponding disk is contained in $\Omega$. For the ellipse, writing $c^2:=a^2-b^2$, one has $\delta(q)=b\sqrt{1-q^2/c^2}$ for $0\le q\le c^2/a$ and $\delta(q)=a-q$ for $q\geq c^2/a$; since $q+\delta(q)=a$ at $q=c^2/a$, for $t\in[b,a)$ the point $q_t$ is the unique solution in $[0,c^2/a)$ of $q+b\sqrt{1-q^2/c^2}=t$, and the disk tangent at $P$ described below is the largest one. }
Imposing that the second component of $Q=P+\theta \nu$ vanishes, we obtain $\theta=-y/w'(y)$ and hence %
$Q=\Big(w(y)+\frac{y}{w'(y)},\,0\Big)=\Big(\frac{a^2-b^2}{ab}\sqrt{b^2-y^2},\,0\Big),$
and
$
 \rho=|P-Q|=%
 \frac{\sqrt{b^4+(a^2-b^2)y^2}}{a} .$
Thus we can express $y$ in terms of $\rho$, substitute it in $Q$ and solve %
$t=x_Q+\rho$ obtaining $\rho=r_{\Omega}(t,0)$, that is
 $$r_{\Omega}(t,0)=\frac{b^2t+b\sqrt{(a^2-b^2)(a^2-t^2)}}{a^2},$$
 and therefore \eqref{eq_romega_ell} together with
 \begin{equation}\label{eq_H_ellipse}
H_{\Omega}(t,0) = -\frac{a^2}{b^2 t+b\sqrt{(a^2-b^2)(a^2-t^2)}} \quad \hbox{for $t \approx a$}.
\end{equation}
{Note that $H_{\Omega}(t,0)\to -a/b^2=-\bar \kappa$ as $t\to a^-$: the limit of
$-H_{\Omega}$ at the vertex is the classical curvature there, in accordance with \eqref{eq_magic_coincid};
we notice that
it is generally false for other points of the boundary, %
{e.g.\ along the minor axis one has $r_\Omega(0,y)\geq b\geq R_\Omega$ (the inscribed disk $B_b(0)$ contains these points, and $R_\Omega\le b$ since $\Omega^{R_\Omega}\neq\emptyset$ by \eqref{eq_inner_cheeger}), so that $-H_\Omega(0,y)=h(\Omega)\geq1/b$ for $|y|<b$, i.e.\ $(0,\pm b)\in\partial K$, whereas the classical curvature at $(0,\pm b)$ is $b/a^2<1/b$.}

Moreover
$$\big(-H_{\Omega}(t,0)\big)^{-1} = \frac{b^2}{a}+\frac{b\sqrt{2a(a^2-b^2)}}{a^2}\,\sqrt{a-t}+O(a-t)
\qquad \hbox{as } t \to a^-,$$
so that again no Lipschitz extension up to $\partial \Omega$ is possible, and $C^{0,1/2}$ is the best one can
hope for. %
}

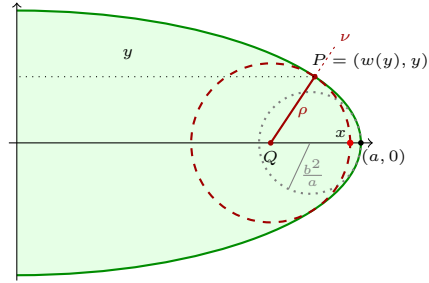
\begin{figure}[htbp]
\centering
\begin{tikzpicture}[scale=1.75]

\pgfmathsetmacro{\aa}{2.6}
\pgfmathsetmacro{\bb}{1}

\pgfmathsetmacro{\Px}{2.251666}
\pgfmathsetmacro{\Py}{0.5}
\pgfmathsetmacro{\Qx}{1.918579}
\pgfmathsetmacro{\rh}{0.600788}
\pgfmathsetmacro{\tt}{2.519368}

\pgfmathsetmacro{\Ox}{2.215385}
\pgfmathsetmacro{\ro}{0.384615}

\fill[green!10]
 (0,-\bb)
 arc[
 start angle=-90,
 end angle=90,
 x radius=\aa,
 y radius=\bb
 ]
 -- cycle;

\draw[thick,green!55!black]
 (0,-\bb)
 arc[
 start angle=-90,
 end angle=90,
 x radius=\aa,
 y radius=\bb
 ];

\draw[->,thin]
 (-0.06,0) -- ({\aa+0.09},0);

\draw[->,thin]
 (0,-1.04) -- (0,1.06);

\draw[gray,dotted,thick]
 (\Ox,0) circle (\ro);

\pgfmathsetmacro{\rang}{245}
\coordinate (Orad) at
 ({\Ox+\ro*cos(\rang)},
 {\ro*sin(\rang)});

\draw[gray,thin]
 (\Ox,0) -- (Orad);

\node[
 gray,
 font=\fontsize{4.4}{4.8}\selectfont,
 anchor=center
]
 at (2.235,-0.225)
 {$\frac{b^2}{a}$};

\draw[red!70!black,thick,dashed]
 (\Qx,0) circle (\rh);

\fill[red!70!black]
 (\Qx,0) circle (0.6pt);

\draw[red!70!black,thick]
 (\Qx,0) -- (\Px,\Py);

\node[
 red!70!black,
 font=\tiny,
 anchor=west
]
 at (2.05,0.235)
 {$\rho$};

\draw[
 red!70!black,
 thin,
 dash pattern=on 1pt off 1.5pt
]
 (\Qx,0)
 --
 ($(\Qx,0)!1.45!(\Px,\Py)$);

\node[
 red!70!black,
 font=\tiny,
 anchor=south west
]
 at (2.37,0.7)
 {$\nu$};

\fill[red!70!black]
 (\Px,\Py) circle (0.6pt);

\node[
 font=\tiny,
 anchor=south west
]
 at (2.155,0.485)
 {$P=(w(y),y)$};

\fill[red]
 (\tt,0) circle (0.7pt);

\node[
 font=\tiny,
 anchor=south east
]
 at (2.56,-0.025)
 {$x$};

\node[
 font=\tiny,
 anchor=north
]
 at (\Qx,+0.01)
 {$Q$};

\fill[black]
 (\aa,0) circle (0.6pt);

\node[
 font=\tiny,
 anchor=north west
]
 at ({\aa-0.07},+0.022)
 {$(a,0)$};

\draw[dotted]
 (0.02,\Py) -- (\Px,\Py);

\node[
 font=\tiny,
 anchor=south
]
 at (0.84,0.56)
 {$y$};

\end{tikzpicture}

\caption{{\small
The elongated ellipse, $a=2.6$, $b=1$, which is not calibrable.
The largest disk through $x=(t,0)$ is tangent to $\partial\Omega$
at $P=(w(y),y)$, its centre $Q$ being the intersection with the axis
of the normal line at $P$, and
$r_{\Omega}(t,0)=\rho=|P-Q|$.
As $y\to0$ the disk tends to the osculating circle at the vertex
(dotted), of radius $b^2/a$.}}
\label{fig_ellipse_constr}
\end{figure}

Let us point out a key difference between $(iii)$ and $(iv)$: in both cases Lipschitz regularity is lost while
approaching the boundary, but in the square it is lost near a \emph{flat} portion of $\partial \Omega$ (and
\emph{not} along the diagonal towards the vertex, where $(-H_{\Omega})^{-1}$ is affine), whereas in the
ellipse it is lost near the point of \emph{maximal curvature}. This shows that the boundary behaviour of
$H_{\Omega}$ is a global phenomenon, and not a local one.

\item[$(v)$] \emph{(No uniform H\"older exponent.)}
Example $(iv)$ can be perturbed so as to destroy any uniform
H\"older regularity up to the boundary.
For $m \in (0,2)$ set
$$\Phi_m(\sigma):=\sigma^2-|\sigma|^{m+2}+|\sigma|^{2m+2}, \qquad
\Omega_m:=\Big\{(x,y)\in \R^2 \ \Big|\ F_m(x,y):=\frac{x^2}{a^2}+\Phi_m\Big(\frac{y}{b}\Big)-1<0\Big\} .$$
The correction $-|\sigma|^{m+2}$ is what produces the effect, while $|\sigma|^{2m+2}$ guarantees that
$\Omega_m$ is well defined and convex.
Indeed, setting $t:=|\sigma|^m$,
$$\Phi_m''(\sigma)=2-(m+1)(m+2)\,t+(2m+1)(2m+2)\,t^2,$$
whose discriminant $(m+1)(m^3+5m^2-24m-12)$ is negative for $m\in(0,2)$; hence $\Phi_m$ is strictly convex,
and so it is $F_m$, %
thus its sublevel sets, $\Omega_m$ in particular, are convex.
Moreover $F_m \in C^2$ and $\nabla F_m \neq (0,0)$ on $F_m=0$, thus $\partial \Omega_m \in C^2$ by the implicit function theorem,
so that
$H_{\Omega_m}\in L^{\infty}(\Omega_m)$ by Remark \ref{rem_ifflinf}.

Notice that, formally, $\Omega_0$ is the ellipse. Moreover, since $\Phi_m(\sigma) \geq 1$ for $|\sigma|\geq 1$, while $\Phi_m(\sigma) \leq \sigma^2$ for $|\sigma|\leq 1$, we have $\Omega_0\subset \Omega_m$.

We consider as in $(iv)$ the function $w_m(y)=a\sqrt{1-\Phi_m(y/b)}$; since $w_m(y)=a-\frac{a}{2b^2}y^2+O(y^{m+2})$, %
and since $m>0$ the remainder does not contribute to $w_m''(0)$: the curvature of $\partial \Omega_m$ at the vertex is therefore $a/b^2$, exactly as for the ellipse.
In particular, by monotonicity of the Cheeger constant, $h(\Omega_m) \leq h(\Omega_0)$, and for $e>e^*$ we have $h(\Omega_m) \leq h(\Omega_0) < \bar{\kappa}(\Omega_0) = a/b^2 \leq \bar{\kappa}(\Omega_m)$, thus $\Omega_m$ is not calibrable.

{The argument leading to \eqref{eq_rOmega_symm} applies to $\Omega_m$, which is convex and symmetric; let $\delta_m(q):=\dist((q,0),\partial\Omega_m)$ and $q_0:=a-b^2/a$. The osculating disk of the ellipse at $(a,0)$, namely $B_{b^2/a}((q_0,0))$, is contained in $\Omega_0\subset\Omega_m$, and its closure meets $\partial\Omega_m$ only at $(a,0)$; hence $\delta_m(q)=a-q$ for $q\in[q_0,a)$. On the other hand $\delta_m(q)<a-q$ for $q<q_0$, since the curvature of $\partial\Omega_m$ at $(a,0)$ is $a/b^2$ and therefore the vertex is not even a local minimum of the distance from $(q,0)$. By \eqref{eq_rOmega_symm}, for $t<a$ close to $a$ we thus have $r_{\Omega_m}(t,0)=\delta_m(q_t)$ with $q_t<q_0$, $q_t\to q_0$; by compactness, every nearest point $P_t$ of $(q_t,0)$ on $\partial\Omega_m$ converges to $(a,0)$, with $P_t\neq(a,0)$. Writing (up to symmetry) $P_t=(w_m(y),y)$ with $y=y_t\to0^+$, the segment joining $(q_t,0)$ and $P_t$ is normal to $\partial\Omega_m$, so that the relations $t=x_Q+\rho$, $\rho=|P-Q|$ of $(iv)$ hold along $y=y_t$, and the asymptotic analysis below is performed along $y_t\to0^+$.} %

Then, {unlike} the ellipse, we do not make direct computations but we perform an asymptotic analysis as
$y \to 0^+$, writing $\sigma=y/b$.

Since $w_m'(y)=O(y)$ and $w_m(y)=a-\frac{a}{2}\Phi_m(\sigma)+O(\sigma^4)$,
\begin{align*}
a-t &= a-w_m(y)-\frac{y}{w_m'(y)}\Big(1-\sqrt{1+w_m'(y)^2}\Big)
= \frac{a}{2}\Big[\Phi_m(\sigma)-\frac{\sigma}{2}\Phi_m'(\sigma)\Big] + O(\sigma^4) \\
&= \frac{am}{4}\,\sigma^{m+2} + O(\sigma^{\min\{4,2m+2\}}),
\end{align*}
where we used that in $\Phi_m(\sigma)-\frac{\sigma}{2}\Phi_m'(\sigma)$ the quadratic terms cancel, which is of key importance.
On the other hand %
\begin{align*}
\rho&=-\frac{y}{w_m'(y)}\sqrt{1+w_m'(y)^2}
= -\frac{y}{w_m'(y)}-\frac{y\,w_m'(y)}{2}+O(y^{3})
= \frac{2b^2}{a}\,\frac{\sigma}{\Phi_m'(\sigma)}+O(\sigma^{2}) \\
&= \frac{b^2}{a}\Big(1+\frac{m+2}{2}\sigma^{m}\Big)+O\big(\sigma^{\min\{2,\,2m\}}\big),
\end{align*}
{so that $\rho \to b^2/a$, the radius of curvature at the vertex, and}
$\rho-\frac{b^2}{a} {= \frac{b^2(m+2)}{2a}\,\sigma^{m}+o(\sigma^m)} .$
Eliminating $\sigma$ between the two asymptotics we obtain $\rho-\frac{b^2}{a}\sim C\,(a-t)^{\frac{m}{m+2}}$, and
therefore
\begin{equation}\label{eq_H_perturbed}
\big(-H_{\Omega_m}(t,0)\big)^{-1} = \frac{b^2}{a}+C\,(a-t)^{\beta_m}+o\big((a-t)^{\beta_m}\big),
\qquad \beta_m=\frac{m}{m+2} \in \Big(0,\frac12\Big),
\end{equation}
as $t \to a^{-}$. As $m \to 0^+$ we have $\beta_m \to 0$, so that no H\"older exponent can be chosen uniformly
in the class of convex domains with $C^{1,1}$ boundary. We notice that formally $\Omega_m \to \Omega_0$ as $m \to 0$, but the found H\"older exponent of $\Omega_0$ (along the axis) is $1/2$, which highlights that the convergence of the domains is \emph{singular}. On the other hand,
$\beta_m \to 1/2$ as $m \to 2^-$, even if $\Omega_2$ is not the ellipse.
\begin{figure}[ht]
\centering
\begin{tikzpicture}[
 x=1.28cm,
 y=1.28cm,
 line cap=round,
 line join=round
]

\def\a{2}
\def\panel{2.32}
\def\ymin{-0.12}
\def\ymax{1.25}

\foreach \i/\m/\lab in {
 0/0.1/{0.1},
 1/0.5/{\frac12},
 2/1/{1},
 3/1.5/{\frac32},
 4/1.99/{1.99}
}{
 \begin{scope}[shift={({\i*\panel},0)}]

 \begin{scope}
 \clip (-0.04,\ymin) rectangle (2.08,\ymax);

 \draw[blue,thick]
 plot[
 domain=0:1,
 samples=70,
 variable=\s
 ]
 ({\a*sqrt(1-\s*\s)},{\s});

 \draw[red!80!black,thick]
 plot[
 domain=0:1,
 samples=80,
 variable=\s
 ]
 ({
 \a*sqrt(
 1-\s*\s
 +pow(\s,\m+2)
 -pow(\s,2*\m+2)
 )
 },{\s});

\draw[green!55!black,thin]
 plot[
 domain=0:1,
 samples=80,
 variable=\s
 ]
 ({
 \a-\s
 },{
 pow(\s,\m+2)-pow(\s,2*\m+2)
 });

\draw[orange!90!black,thin]
 plot[
 domain=0:1,
 samples=90,
 variable=\u
 ]
 ({
 \a-pow(\u,4)
 },{
 (\m+1)*(\m+2)*pow(\u,4*\m)
 -(2*\m+1)*(2*\m+2)*pow(\u,8*\m)
 });

 \end{scope}

 \draw[->,thin] (-0.04,0)--(2.10,0);
 \draw[->,thin] (0,\ymin)--(0,1.29);

 \fill (2,0) circle (0.45pt);

 \node[anchor=north] at (1,-0.16)
 {$m=\lab$};

 \end{scope}
}

%
%
%
%

\end{tikzpicture}

\caption{
Comparison between the ellipse $\Omega_0$ (in blue) and $\Omega_m$ (in red), together with the error $E_m(s)=s^2-\Phi_m(s)$ (in green) and $E_m''(s)$ (in orange), translated by $s \mapsto a-s$, so that $s=0$ corresponds to the right vertex.
}
\label{fig:Omega_m_excess}
\end{figure}
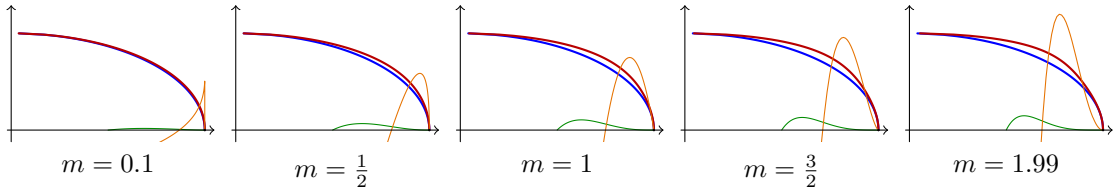
\end{itemize}
\end{example}
\noindent The following two examples, which we borrow from \cite{GoMa94}, illustrate Remark
\ref{rem_ifflinf}; note that they are not (all) convex {(see \cite{GoMa94} for the precise definition of $H_{\Omega}$ in this case)}, so they are recorded here only for completeness.

\begin{example}\label{ex_nonconvex} $\,$
\begin{itemize}[leftmargin=*]
\item[$(vi)$] Let $\alpha >0$ and let $B$ be the disk in $\R^2$ tangent to $|y|=x^{\alpha}$ at $(1, \pm 1)$,
and consider $\Omega = B \cup \{x \in (0,1), \, |y| 
{<} x^{\alpha}\}$. Then $H_{\Omega} \in L^{\infty}(\Omega)$
if and only if $\alpha \leq \frac{1}{2}$, that is, if and only if $\partial \Omega \in C^{1,1}$. Notice that
$\Omega$ is not convex for $\alpha > 1$.
\begin{figure}[H]
\centering
\begin{tikzpicture}[scale=0.6]
\foreach \a/\xshift/\lab [count=\i] in
 {0.25/0/{1/4},
 0.5/5/{1/2},
 0.75/10/{3/4},
 1.0/15/{1},
 1.25/20/{5/4}}
{
 \begin{scope}[shift={(\xshift,0)}]
 \pgfmathsetmacro{\cx}{1 + \a}
 \pgfmathsetmacro{\Rag}{sqrt(\a*\a + 1)}
 \pgfmathsetmacro{\ang}{atan2(1,\a)}
 \pgfmathsetmacro{\startAng}{180 - \ang}
 \pgfmathsetmacro{\endAng}{180 + \ang}

 \fill[green!20]
 (0,0) -- plot[domain=0:1, samples=50] (\x,{pow(\x,\a)})
 arc[start angle=\startAng, end angle=\endAng-360, radius=\Rag]
 -- plot[domain=1:0, samples=50] (\x,{-pow(\x,\a)})
 -- cycle;

 \draw[->,gray!30] (-0.2,0) -- ({\cx + \Rag + 0.5},0);
 \draw[->,gray!30] (0,-1.5) -- (0,1.5);

 \draw[thick, green!60!black]
 (0,0) -- plot[domain=0:1, samples=50] (\x,{pow(\x,\a)})
 arc[start angle=\startAng, end angle=\endAng-360, radius=\Rag]
 -- plot[domain=1:0, samples=50] (\x,{-pow(\x,\a)})
 -- cycle;

 \draw[dashed, gray!60, thin] (1,1) -- (1,-1);

 \fill (1,1) circle (1.5pt);
 \fill (1,-1) circle (1.5pt);

 \node[below] at (1,-2.2) {\small $\alpha = \lab$};
 \node[below] at (1,-2.7)
 {\tiny \ifnum\i<3 $H \in L^\infty$ \else $H \notin L^\infty$ \fi};
 \end{scope}
}
\end{tikzpicture}
\caption{{\small The domains in Example \ref{ex_nonconvex}$(vi)$ for several values of $\alpha$.}}
\label{fig_nonconvex_alpha}
\end{figure}
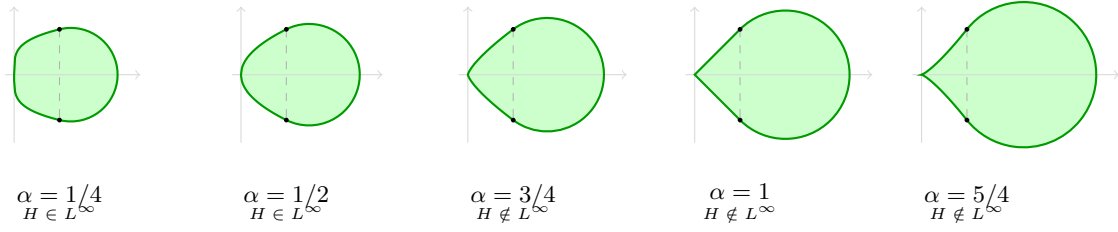
\item[$(vii)$] Let $\Omega= \{(x,y,z) \in B_1(0) \mid z < (x^2-y^2) (-\log(x^2+y^2)^{1/2})^{1/2}\}$. Then
$H_{\Omega} \in L^{\infty}(B_1)$ and $\partial \Omega \in C^{1,\alpha}$ for each $\alpha \in (0,1)$, but
$\partial \Omega \notin C^{1,1}$. Notice that $\Omega$ is not convex.
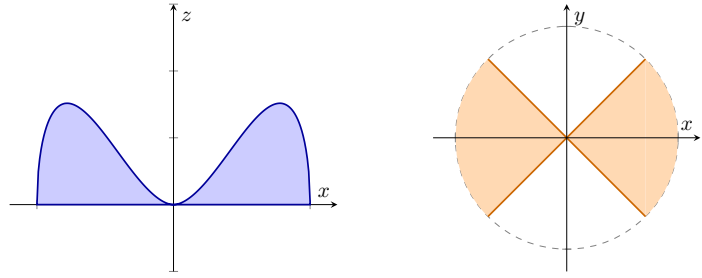
\begin{figure}[H]
\centering
\begin{tikzpicture}[scale=0.8]
\begin{scope}
\begin{axis}[
 width=7cm, height=6cm,
 axis lines=middle,
 xlabel={$x$}, ylabel={$z$},
 xmin=-1.2, xmax=1.2,
 ymin=-0.2, ymax=0.6,
 xtick={-1, 1},
 xticklabels={},
 yticklabels={},
 samples=200
]
 \addplot[draw=blue!60!black, thick, fill=blue!20, domain=-1:1, stack plots=y]
 {x^2 * sqrt(-ln(abs(x)+1e-9))} \closedcycle;
\end{axis}
\end{scope}

\begin{scope}[shift={(7,0)}]
\begin{axis}[
 width=7cm, height=6cm,
 axis lines=middle,
 axis on top,
 xlabel={$x$}, ylabel={$y$},
 xmin=-1.2, xmax=1.2,
 ymin=-1.2, ymax=1.2,
 unit vector ratio=1 1,
 xtick={-1, 1}, ytick={-1, 1},
 xticklabels={},
 yticklabels={},
]
 \draw[dashed, gray] (axis cs:0,0) circle [radius=1];
 \addplot[fill=orange!30, draw=none, domain=0:0.707] {x} \closedcycle;
 \addplot[fill=orange!30, draw=none, domain=0:0.707] {-x} \closedcycle;
 \addplot[fill=orange!30, draw=none, domain=0.707:1] {sqrt(1-x^2)} \closedcycle;
 \addplot[fill=orange!30, draw=none, domain=0.707:1] {-sqrt(1-x^2)} \closedcycle;
 \addplot[fill=orange!30, draw=none, domain=-0.707:0] {x} \closedcycle;
 \addplot[fill=orange!30, draw=none, domain=-0.707:0] {-x} \closedcycle;
 \addplot[fill=orange!30, draw=none, domain=-1:-0.707] {sqrt(1-x^2)} \closedcycle;
 \addplot[fill=orange!30, draw=none, domain=-1:-0.707] {-sqrt(1-x^2)} \closedcycle;
 \draw[thick, orange!80!black] (axis cs:-0.707,-0.707) -- (axis cs:0.707,0.707);
 \draw[thick, orange!80!black] (axis cs:-0.707,0.707) -- (axis cs:0.707,-0.707);
\end{axis}
\end{scope}

\end{tikzpicture}
\caption{{\small Cross sections of the domain in Example \ref{ex_nonconvex}$(vii)$ at $y=0$ (left) and $z=0$ (right).}}
\label{fig_nonconvex_3d}
\end{figure}
\end{itemize}
\end{example}

\begingroup
\renewcommand{\addcontentsline}[3]{}
\section*{AI declaration}
 The authors acknowledge the use of Claude Opus 5 and OpenAI GPT-5.6 Sol as AI tools %
 for detecting
typos and mistakes and for bibliographic verifications and cross-checks, as well as for generation of images. All mathematical
arguments and proofs in the final manuscript were checked and written by
the authors, who retain full responsibility for the whole content.
\endgroup


\begin{thebibliography}{99}

\bibitem{AlCa09} %
F. Alter, V. Caselles,
\emph{Uniqueness of the Cheeger set of a convex body},
Nonlinear Anal. \textbf{\href{https://doi.org/10.1016/j.na.2007.11.032}{70}} (2009), 32--44.

\bibitem{ACC05} %
F. Alter, V. Caselles, A. Chambolle,
\emph{A characterization of convex calibrable sets in $\mathbb{R}^N$},
Math. Ann. \textbf{\href{https://doi.org/10.1007/s00208-004-0628-9}{332}} (2005), 329--366.

\bibitem{ACC05b}
{F. Alter, V. Caselles, A. Chambolle, \emph{Evolution of characteristic functions of convex sets in the plane by the minimizing total variation flow}, Interfaces Free Bound. \textbf{\href{https://doi.org/10.4171/IFB/112}{7}} (2005), no.~1, 29--53.}

\bibitem{Alv77} %
A. Alvino,
\emph{Sulla diseguaglianza di Sobolev in spazi di Lorentz},
Boll. Un. Mat. Ital. A (5) \textbf{\href{https://www.iris.unina.it/handle/11588/148941}{14}} (1977), 148--156.

\bibitem{AFP00} %
L. Ambrosio, N. Fusco, D. Pallara,
\emph{Functions of bounded variation and free discontinuity problems},
Oxford Mathematical Monographs, ISBN \textbf{\href{https://doi.org/10.1093/oso/9780198502456.001.0001}{9780198502456}},
Oxford University Press, Oxford, 2000.

\bibitem{ACM04}
{F. Andreu-Vaillo, V. Caselles, J.$\,$M. Maz\'on, \emph{Parabolic quasilinear equations minimizing linear growth functionals}, Progress in Mathematics \textbf{\href{https://doi.org/10.1007/978-3-0348-7928-6}{223}}, Birkh\"auser, Basel, 2004.}

\bibitem{Anz83} %
{G. Anzellotti,
\emph{Pairings between measures and bounded functions and compensated compactness},
Ann. Mat. Pura Appl. (4) \textbf{\href{https://doi.org/10.1007/BF01781073}{135}} (1983), 293--318.}

\bibitem{Ban99}
{V. Bangert, \emph{Convex hypersurfaces with bounded first mean curvature measure}, Calc. Var. Partial Differential Equations \textbf{\href{https://doi.org/10.1007/s005260050126}{8}} (1999), 259--278.}

\bibitem{Bar94} %
E. Barozzi,
\emph{The curvature of a set with finite area},
Atti Accad. Naz. Lincei Rend. Lincei Mat. Appl. \textbf{\href{https://www.bdim.eu/item?id=RLIN_1994_9_5_2_149_0}{5}} (1994), no.~2, 149--159.

\bibitem{BGT87} %
{E. Barozzi, E.$\,$H.$\,$A. Gonzalez, I. Tamanini,
\emph{The mean curvature of a set of finite perimeter},
Proc. Amer. Math. Soc. \textbf{\href{https://doi.org/10.1090/S0002-9939-1987-0870791-8}{99}} (1987), no.~2, 313--316.}

\bibitem{BaMa13} %
{E. Barozzi, U. Massari,
\emph{On sets with variational mean curvature in $L^n(\mathbb{R}^n)$},
Ann. Univ. Ferrara Sez. VII Sci. Mat. \textbf{\href{https://doi.org/10.1007/s11565-013-0179-7}{59}} (2013), 259--268.}

\bibitem{BCN02} %
G. Bellettini, V. Caselles, M. Novaga,
\emph{The total variation flow in $\mathbb{R}^N$},
J. Differential Equations \textbf{\href{https://doi.org/10.1006/jdeq.2001.4150}{184}} (2002), no.~2, 475--525.

\bibitem{BeSh88} %
C. Bennett, R. Sharpley,
\emph{Interpolation of operators},
Pure and Applied Mathematics %
\textbf{\href{https://shop.elsevier.com/books/interpolation-of-operators/bennett/978-0-12-088730-9}{129}},
Academic Press, Boston, 1988.

\bibitem{BGP99}
S. Berhanu, F. Gladiali, G. Porru,
\emph{Qualitative properties of solutions to elliptic singular problems},
J. Inequal. Appl. \textbf{\href{https://doi.org/10.1155/S1025583499000223}{3}} (1999), 313--330.

\bibitem{BMS24} %
W. Borrelli, S. Mosconi, M. Squassina,
\emph{Concavity properties for solutions to $p$-Laplace equations with concave nonlinearities},
Adv. Calc. Var. \textbf{\href{https://doi.org/10.1515/acv-2021-0100}{17}} (2024), no.~1, 79--97.

\bibitem{BrLi76}
H.$\,$J. Brascamp, E.$\,$H. Lieb,
\emph{On extensions of the Brunn-Minkowski and Prékopa-Leindler theorems, including inequalities for log concave functions, and with an application to the diffusion equation},
J. Funct. Anal. \textbf{\href{https://doi.org/10.1016/0022-1236(76)90004-5}{22}} (1976), no.~4, 366--389.

\bibitem{Buc25} %
C. Bucur,
\emph{Solutions of the fractional $1$-Laplacian: existence, asymptotics and flatness results},
NoDEA Nonlinear Differential Equations Appl. \textbf{\href{https://doi.org/10.1007/s00030-025-01052-8}{32}} (2025), Art. No. 52, pp.~45.

\bibitem{BEM16} %
H. Bueno, G. Ercole, S.$\,$S. Macedo,
\emph{Asymptotic behavior of the $p$-torsion functions as $p$ goes to $1$},
Arch. Math. \textbf{\href{https://doi.org/10.1007/s00013-016-0922-2}{107}} (2016), 63--72.

\bibitem{BCC07}
{G. Buttazzo, G. Carlier, M. Comte, \emph{On the selection of maximal Cheeger sets}, Differential Integral Equations \textbf{\href{https://doi.org/10.57262/die/1356039307}{20}} (2007), no.~9, 991--1004.}

\bibitem{CaSp82}
{L.$\,$A. Caffarelli, J. Spruck, \emph{Convexity properties of solutions to some classical variational problems}, Comm. Partial Differential Equations \textbf{\href{https://doi.org/10.1080/03605308208820254}{7}} (1982), no.~11, 1337--1379.}

\bibitem{CST16} %
A. Canino, B. Sciunzi, A. Trombetta,
\emph{Existence and uniqueness for $p$-Laplace equations involving singular nonlinearities},
NoDEA Nonlinear Differential Equations Appl. \textbf{\href{https://doi.org/10.1007/s00030-016-0361-6}{23}} (2016), no.~2, Art. No. 8, pp.~18.

\bibitem{CaCo07}
{G. Carlier, M. Comte, \emph{On a weighted total variation minimization problem}, J. Funct. Anal. \textbf{\href{https://doi.org/10.1016/j.jfa.2007.05.022}{250}} (2007), no.~1, 214--226.}

\bibitem{CaSa23} %
M. Caroccia, G. Saracco,
\emph{Isoperimetric sets and $p$-Cheeger sets are in bijection},
J. Geom. Anal. \textbf{\href{https://doi.org/10.1007/s12220-022-01157-x}{33}} (2023), Art. No. 129, pp.~17.

\bibitem{CCN07}
{V. Caselles, A. Chambolle, M. Novaga, \emph{Uniqueness of the Cheeger set of a convex body}, Pacific J. Math. \textbf{\href{https://doi.org/10.2140/pjm.2007.232.77}{232}} (2007), no.~1, 77--90.}

\bibitem{CCN10} %
V. Caselles, A. Chambolle, M. Novaga,
\emph{Some remarks on uniqueness and regularity of Cheeger sets of convex bodies},
Rend. Sem. Mat. Univ. Padova \textbf{\href{https://doi.org/10.4171/RSMUP/123-9}{123}} (2010), 191--201.

\bibitem{ChFr99} %
{G.$\,$Q. Chen, H. Frid,
\emph{Divergence-measure fields and hyperbolic conservation laws},
Arch. Ration. Mech. Anal. \textbf{\href{https://doi.org/10.1007/s002050050146}{147}} (1999), no.~2, 89--118.}

\bibitem{CiTr03} %
M. Cicalese, C. Trombetti,
\emph{Asymptotic behaviour of solutions to $p$-Laplacian equation},
Asymptot. Anal. \textbf{\href{https://doi.org/10.3233/ASY-2003-564}{35}} (2003), no.~1, 27--40.

\bibitem{CRT77}
{M.$\,$G. Crandall, P.$\,$H. Rabinowitz, L. Tartar, \emph{On a Dirichlet problem with a singular nonlinearity}, Comm. Partial Differential Equations \textbf{\href{https://doi.org/10.1080/03605307708820029}{2}} (1977), 193--222.}

\bibitem{CrDC19} %
{G. Crasta, V. De Cicco,
\emph{Anzellotti's pairing theory and the Gauss--Green theorem},
Adv. Math. \textbf{\href{https://doi.org/10.1016/j.aim.2018.12.007}{343}} (2019), 935--970.}

\bibitem{DeC13} %
L.$\,$M. De Cave,
\emph{Nonlinear elliptic equations with singular nonlinearities},
Asymptot. Anal. \textbf{\href{https://doi.org/10.3233/ASY-131173}{84}} (2013), no.~3--4, 181--195.

\bibitem{DGOP19} %
V. De Cicco, D. Giachetti, F. Oliva, F. Petitta,
\emph{The Dirichlet problem for singular elliptic equations with general nonlinearities},
Calc. Var. Partial Differential Equations \textbf{\href{https://doi.org/10.1007/s00526-019-1582-4}{58}} (2019), Art. No. 129, pp.~40.

\bibitem{DGS19} %
V. De Cicco, D. Giachetti, S. Segura de León,
\emph{Elliptic problems involving the $1$-Laplacian and a singular lower order term},
J. Lond. Math. Soc. (2) \textbf{\href{https://doi.org/10.1112/jlms.12172}{99}} (2019), 349--376.

\bibitem{DuOl22} %
R. Durastanti, F. Oliva,
\emph{Comparison principle for elliptic equations with mixed singular nonlinearities},
Potential Anal. \textbf{\href{https://doi.org/10.1007/s11118-021-09906-3}{57}} (2022), 83--100.

\bibitem{Fir61}
W. J. Firey,
\emph{Polar means of convex bodies and a dual to
the Brunn--Minkowski theorem},
Canad. J. Math.
\textbf{\href{https://doi.org/10.4153/CJM-1961-037-0}{13}}
(1961), 444--453.

\bibitem{Fto25} %
I. Ftouhi,
\emph{The monotonicity of the Cheeger constant for parallel bodies},
J. Optim. Theory Appl. \textbf{\href{https://doi.org/10.1007/s10957-025-02727-z}{206}} (2025), Art. No. 44, pp.~26.

\bibitem{GKR68}
{D. Gale, V. Klee, R.$\,$T. Rockafellar, \emph{Convex functions on convex polytopes}, Proc. Amer. Math. Soc. \textbf{\href{https://doi.org/10.1090/S0002-9939-1968-0230219-6}{19}} (1968), 867--873.}

\bibitem{GaSq25} %
M. Gallo, M. Squassina,
\emph{Concavity and perturbed concavity for $p$-Laplace equations},
J. Differential Equations \textbf{\href{https://doi.org/10.1016/j.jde.2025.113452}{440}} (2025), Art. No. 113452, pp.~67.

\bibitem{GaZi77}
{R. Gariepy, W.$\,$P. Ziemer, \emph{A regularity condition at the boundary for solutions of quasilinear elliptic equations}, Arch. Ration. Mech. Anal. \textbf{\href{https://doi.org/10.1007/BF00280825}{67}} (1977), 25--39.}

\bibitem{GeHu23} %
E. Georgiou, R.$\,$R. Huilgol,
\emph{Cheeger sets and the minimum pressure gradient problem for viscoplastic
fluids},
J. Non-Newton. Fluid Mech. \textbf{\href{https://doi.org/10.1016/j.jnnfm.2023.104999}{313}} (2023), Art. No. 104999, pp.~7.

\bibitem{GOP22} %
{D. Giachetti, F. Oliva, F. Petitta,
\emph{$1$-Laplacian type problems with strongly singular nonlinearities and gradient terms},
Commun. Contemp. Math. \textbf{\href{https://doi.org/10.1142/S0219199721500814}{24}} (2022), no.~10, Art. No. 2150081, pp.~40.}

\bibitem{GST07}
{J. Giacomoni, I. Schindler, P. Tak\'a\v{c}, \emph{Sobolev versus H\"older local minimizers and existence of multiple solutions for a singular quasilinear equation}, Ann. Sc. Norm. Super. Pisa Cl. Sci. (5) \textbf{\href{https://doi.org/10.2422/2036-2145.2007.1.07}{6}} (2007), no.~1, 117--158.}

\bibitem{Gla02} %
F. Gladiali,
\emph{Proprietà qualitative e non esistenza di soluzioni positive per alcuni problemi ellittici non lineari},
PhD thesis, Università degli Studi di Roma ``La Sapienza'', \textbf{\href{http://www.bdim.eu/item?id=tesi_2002_GladialiFrancesca_1}{bdim.eu}}, 2002.

\bibitem{GoMa94} %
E.$\,$H.$\,$A. Gonzalez, U. Massari,
\emph{Variational mean curvatures},
Rend. Sem. Mat. Univ. Politec. Torino \textbf{\href{https://sfera.unife.it/handle/11392/462319}{52}} (1994), no.~1, 1--28.

\bibitem{GMT83} %
E.$\,$H.$\,$A. Gonzalez, U. Massari, I. Tamanini,
\emph{On the regularity of boundaries of sets minimizing perimeter with a volume constraint},
Indiana Univ. Math. J. \textbf{\href{https://doi.org/10.1512/iumj.1983.32.32003}{32}} (1983), 25--37.

\bibitem{HiRu97}
{T. Hildebrand, P. R\"uegsegger, \emph{A new method for the model-independent assessment of thickness in three-dimensional images}, J. Microsc. \textbf{\href{https://doi.org/10.1046/j.1365-2818.1997.1340694.x}{185}} (1997), 67--75.}

\bibitem{IoLR05}
{I.$\,$R. Ionescu, T. Lachand-Robert, \emph{Generalized Cheeger sets related to landslides}, Calc. Var. Partial Differential Equations \textbf{\href{https://doi.org/10.1007/s00526-004-0300-y}{23}} (2005), no.~2, 227--249.}

\bibitem{IST20}
{K. Ishige, P. Salani, A. Takatsu, \emph{To logconcavity and beyond}, Commun. Contemp. Math. \textbf{\href{https://doi.org/10.1142/S0219199719500093}{22}} (2020), no.~2, Art. No. 1950009, pp.~17.}

\bibitem{Kaw85}
{B. Kawohl, \emph{Rearrangements and convexity of level sets in PDE}, Lecture Notes in Math. \textbf{\href{https://doi.org/10.1007/BFb0075060}{1150}}, Springer, Berlin, 1985.}

\bibitem{Kaw90} %
B. Kawohl,
\emph{On a family of torsional creep problems},
J. Reine Angew. Math. \textbf{\href{https://doi.org/10.1515/crll.1990.410.1}{410}} (1990), 1--22.

\bibitem{KaFr03}
{B. Kawohl, V. Fridman, \emph{Isoperimetric estimates for the first eigenvalue of the $p$-Laplace operator and the Cheeger constant}, Comment. Math. Univ. Carolin. \textbf{\href{https://www.dml.cz/handle/10338.dmlcz/119420}{44}} (2003), no.~4, 659--667.}

\bibitem{KaLa06} %
B. Kawohl, T. Lachand-Robert,
\emph{Characterization of Cheeger sets for convex subsets of the plane},
Pacific J. Math. \textbf{\href{https://doi.org/10.2140/pjm.2006.225.103}{225}} (2006), 103--118.

\bibitem{Ken85}
A.$\,$U. Kennington,
\emph{Power concavity and boundary value problems},
Indiana Univ. Math. J. \textbf{\href{https://doi.org/10.1512/iumj.1985.34.34036}{34}} (1985), no.~3, 687--704.

\bibitem{Kor83}
N.$\,$J. Korevaar,
\emph{Convex solutions to nonlinear elliptic and parabolic boundary value problems},
Indiana Univ. Math. J. \textbf{\href{https://doi.org/10.1512/iumj.1983.32.32042}{32}} (1983), no.~4, 603--614.

\bibitem{LaUr68} %
O.$\,$A. Ladyzhenskaya, N.$\,$N. Ural'tseva,
\emph{Linear and quasilinear elliptic equations},
Mathematics in Science and Engineering \textbf{\href{https://shop.elsevier.com/books/linear-and-quasilinear-elliptic-equations/ladyzhenskaya/978-0-12-432850-1}{46}}, Academic Press, New York, 1968.

\bibitem{LOPdL21} %
M. Latorre, F. Oliva, F. Petitta, S. Segura de León,
\emph{The Dirichlet problem for the $1$-Laplacian with a general singular term and $L^1$-data},
Nonlinearity \textbf{\href{https://doi.org/10.1088/1361-6544/abc65b}{34}} (2021), no.~3, 1791--1816. %

\bibitem{LaMc91}
{A.$\,$C. Lazer, P.$\,$J. McKenna, \emph{On a singular nonlinear elliptic boundary-value problem}, Proc. Amer. Math. Soc. \textbf{\href{https://doi.org/10.1090/S0002-9939-1991-1037213-9}{111}} (1991), 721--730.}

\bibitem{Leo15}
{G.$\,$P. Leonardi, \emph{An overview on the Cheeger problem}, in: New trends in shape optimization, Progr. Nonlinear Differential Equations Appl. \textbf{\href{https://doi.org/10.1007/978-3-319-17563-8_6}{88}}, Birkh\"auser/Springer, Cham, 2015, 117--139.}

\bibitem{LNS17} %
G.$\,$P. Leonardi, R. Neumayer, G. Saracco,
\emph{The Cheeger constant of a Jordan domain without necks},
Calc. Var. Partial Differential Equations \textbf{\href{https://doi.org/10.1007/s00526-017-1263-0}{56}} (2017), Art. No. 164, pp.~29.

\bibitem{LePr16}
G.$\,$P. Leonardi, A. Pratelli,
\emph{On the Cheeger sets in strips and non-convex domains},
Calc. Var. Partial Differential Equations \textbf{\href{https://doi.org/10.1007/s00526-016-0953-3}{55}} (2016), Art. No. 15, pp.~28.

\bibitem{LeSa20} %
G.$\,$P. Leonardi, G. Saracco,
\emph{Minimizers of the prescribed curvature functional in a Jordan domain with no necks},
ESAIM Control Optim. Calc. Var. \textbf{\href{https://doi.org/10.1051/cocv/2020030}{26}} (2020), Art. No. 76, pp.~20.

\bibitem{Lut75}
E. Lutwak,
\emph{Dual cross-sectional measures},
Atti Accad. Naz. Lincei Rend. Cl. Sci. Fis.
Mat. Natur. (8)
\textbf{58} (1975), no.~1, 1--5.

\bibitem{Mag12} %
F. Maggi,
\emph{Sets of finite perimeter and geometric variational problems: an introduction to geometric measure
theory},
Cambridge Studies in Advanced Mathematics %
\textbf{\href{https://doi.org/10.1017/CBO9781139108133}{135}},
Cambridge University Press, Cambridge, 2012.

\bibitem{Mak71}
L.$\,$G. Makar-Limanov,
\emph{Solution of Dirichlet’s problem for the equation $\Delta u = - 1$ in a convex region},
Mat. Zametki \textbf{\href{https://www.mathnet.ru/eng/mzm9647}{9}} (1971), no.~1, 89--92.

\bibitem{MAOP25} %
A.$\,$J. Martínez Aparicio, F. Oliva, F. Petitta,
\emph{Optimal global $BV$ regularity for $1$-Laplace type BVP's with singular lower order terms},
{Commun. Contemp. Math. \textbf{\href{https://doi.org/10.1142/S0219199725500609}{28}} (2026), no.~6, Art. No. 2550060, pp.~18.}

\bibitem{Mas75} %
U. Massari,
\emph{Frontiere orientate di curvatura media assegnata in $L^p$},
Rend. Sem. Mat. Univ. Padova \textbf{\href{https://www.numdam.org/item/RSMUP_1975__53__37_0/}{53}} (1975), 37--52.

\bibitem{MdLT08} %
A. Mercaldo, S. Segura de León, C. Trombetti,
\emph{On the behaviour of the solutions to $p$-Laplacian equations as $p$ goes to $1$},
Publ. Mat. \textbf{\href{https://doi.org/10.5565/PUBLMAT_52208_07}{52}} (2008), no.~2, 377--411.

\bibitem{MdLT09} %
A. Mercaldo, S. Segura de León, C. Trombetti,
\emph{On the solutions to $1$-Laplacian equation with $L^1$ data},
J. Funct. Anal. \textbf{\href{https://doi.org/10.1016/j.jfa.2008.12.025}{256}} (2009), no.~8, 2387--2416.

\bibitem{MoPe15} %
S. Moll, F. Petitta,
\emph{Large solutions for the elliptic $1$-Laplacian with absorption},
J. Anal. Math. \textbf{\href{https://doi.org/10.1007/s11854-015-0004-6}{125}} (2015), no.~1, 113--138.

\bibitem{OlPe25} %
F. Oliva, F. Petitta,
\emph{Singular elliptic PDEs: an extensive overview},
Partial Differ. Equ. Appl. \textbf{\href{https://doi.org/10.1007/s42985-024-00308-9}{6}} (2025), Art. No. 6, pp.~82.

\bibitem{OrPo17} %
L. Orsina, A.$\,$C. Ponce,
\emph{Flat solutions of the $1$-Laplacian equation},
Rend. Istit. Mat. Univ. Trieste \textbf{\href{https://doi.org/10.13137/2464-8728/16204}{49}} (2017), 41--51.

\bibitem{Par11}
{E. Parini, \emph{An introduction to the Cheeger problem}, Surv. Math. Appl. \textbf{\href{https://www.utgjiu.ro/math/sma/v06/a02.html}{6}} (2011), 9--21.}

\bibitem{Rev16}
S.$\,$G. R\'ev\'esz,
\emph{A discrete extension of the Blaschke rolling ball theorem},
Geom. Dedicata \textbf{\href{https://doi.org/10.1007/s10711-015-0127-z}{182}} (2016), 51--72.

\bibitem{Roc97} %
R.$\,$T. Rockafellar,
\emph{Convex analysis},
Princeton Mathematical Series \textbf{\href{https://ingramacademic.com/products/convex-analysis-9780691015866}{28}}, Princeton University Press, Princeton, NJ, 1997.

\bibitem{Sak87} %
S. Sakaguchi,
\emph{Concavity properties of solutions to some degenerate quasilinear elliptic Dirichlet problems},
Ann. Scuola Norm. Sup. Pisa Cl. Sci. (4) \textbf{\href{https://numdam.org/item/ASNSP_1987_4_14_3_403_0/}{14}} (1987), no.~3, 403--421.

\bibitem{Sar18}
{G. Saracco, \emph{Weighted Cheeger sets are domains of isoperimetry}, Manuscripta Math. \textbf{\href{https://doi.org/10.1007/s00229-017-0974-z}{156}} (2018), no.~3--4, 371--381.}

\bibitem{ScSc23} %
T. Schmidt, J.$\,$H. Sch\"utt,
\emph{The optimal H\"older exponent in Massari's regularity theorem},
Calc. Var. Partial Differential Equations \textbf{\href{https://doi.org/10.1007/s00526-023-02495-6}{62}} (2023), Art. No. 161, pp.~23.

\bibitem{ScSc26}
T. Schmidt, J.$\,$H. Sch\"utt,
\emph{Variational mean curvatures for general $C^{1,\alpha}$ domains},
\href{https://doi.org/10.48550/arXiv.2610.24224}{arXiv:2610.24224} (2026).

\bibitem{Sch14}
{R. Schneider,
\emph{Convex bodies: the Brunn--Minkowski theory},
Encyclopedia of Mathematics and its Applications
\textbf{\href{https://doi.org/10.1017/CBO9781139003858}{151}},
Cambridge University Press, Cambridge, 2014.}

\bibitem{Ser82}
{J. Serra, \emph{Image analysis and mathematical morphology}, vol.~\textbf{\href{https://openlibrary.org/books/OL21984703M/Image_analysis_and_mathematical_morphology}{1}}, Academic Press, London, 1982.}

\bibitem{StZi97}
{E. Stredulinsky, W.$\,$P. Ziemer, \emph{Area minimizing sets subject to a volume constraint in a convex set}, J. Geom. Anal. \textbf{\href{https://doi.org/10.1007/BF02921639}{7}} (1997), no.~4, 653--677.}

\bibitem{Tru67} %
N.$\,$S. Trudinger,
\emph{On Harnack type inequalities and their application to quasilinear elliptic equations},
Comm. Pure Appl. Math. \textbf{\href{https://doi.org/10.1002/cpa.3160200406}{20}} (1967), no.~4, 721--747.



\end{thebibliography}
\end{document}